\documentclass[11pt]{amsart}
\usepackage[margin=1.3in]{geometry}
\usepackage[dvipsnames]{xcolor}

\definecolor{darkgoldenrod}{rgb}{0.72, 0.53, 0.04}
\definecolor{vegasgold}{rgb}{0.77, 0.7, 0.35}
\definecolor{gold(metallic)}{rgb}{0.83, 0.69, 0.22}
\definecolor{sepia}{rgb}{0.44, 0.26, 0.08}
\definecolor{silver}{rgb}{0.75, 0.75, 0.75}
\usepackage{
 amsmath, 
 amsxtra, 
 amsthm, 
 amssymb, 
 etex, 
 mathrsfs, 
 mathtools, 
 tikz-cd, 
 bbm,
 xr,
 scalerel,
 comment}
\usepackage[pdfencoding=unicode, psdextra]{hyperref}
\hypersetup{
 colorlinks=true,
 linkcolor=sepia,
 filecolor=vegasgold,
 citecolor=darkgoldenrod,
 urlcolor=sepia,
 pdftitle={Stats and noncom. Iwasawa theory},
 }
\usepackage[normalem]{ulem}
\usepackage[toc]{appendix} 
\usepackage[all]{xy}
\usepackage{hyperref}
\usepackage{url}
\usepackage{tipa}
\usepackage{dsfont}
\usepackage{ textcomp }
\usepackage{stmaryrd} 
\usepackage{amscd}
\usepackage{amsbsy}
\usepackage{enumerate}
\usepackage{color}
\usepackage{mathtools,caption}
\usepackage{tikz-cd}
\usepackage{longtable}
\usepackage[utf8]{inputenc}
\usepackage[OT2,T1]{fontenc}
\usepackage{changepage}
\usepackage{commath}

\DeclareSymbolFont{cyrletters}{OT2}{wncyr}{m}{n}
\DeclareMathSymbol{\Sha}{\mathalpha}{cyrletters}{"58}

\DeclareMathOperator{\trace}{trace}
\DeclareMathOperator{\ac}{ac}
\DeclareMathOperator{\local}{loc}
\DeclareMathOperator{\glob}{glob}
\DeclareMathOperator{\base}{base}
\DeclareMathOperator{\new}{new}

\DeclareMathOperator{\Hom}{Hom}

\DeclareMathOperator{\height}{ht}
\DeclareMathOperator{\Gal}{Gal}

\DeclareMathOperator{\rank}{rank}
\DeclareMathOperator{\ord}{ord}

\DeclareMathOperator{\corank}{corank}
\DeclareMathOperator{\Sel}{Sel}
\DeclareMathOperator{\Reg}{Reg}
\DeclareMathOperator{\Ak}{Ak}
\DeclareMathOperator{\coker}{coker}

\DeclareMathOperator{\GL}{GL}

\DeclareMathOperator{\Finite}{Fin}

 \theoremstyle{definition}
\newtheorem{theorem}{Theorem}[section]
\newtheorem{lemma}[theorem]{Lemma}

\newtheorem{conj}[theorem]{Conjecture}

\newtheorem{proposition}[theorem]{Proposition}
\newtheorem{corollary}[theorem]{Corollary}

\numberwithin{equation}{section}
\newtheorem{definition}[theorem]{Definition}
\newtheorem{Assumption}[theorem]{Assumption}
\theoremstyle{remark}
\newtheorem{remark}[theorem]{Remark}

\newcommand\EatDot[1]{}
\newcommand{\cyc}{\mathrm{cyc}}

\newcommand{\Fp}{\mathbb{F}_p}
\newcommand{\cF}{{\mathcal{F}}}
\newcommand{\cE}{{\mathcal{E}}}

\newcommand{\cT}{{\mathcal{T}}}

\newcommand{\cC}{\mathcal{C}}
\newcommand{\cD}{\mathcal{D}}
\newcommand{\cX}{\mathcal{X}}

\newcommand{\Q}{\mathbb{Q}}
\newcommand{\Z}{\mathbb{Z}}
\newcommand{\F}{\mathbb{F}}
\newcommand{\Qp}{\mathbb{Q}_p}
\newcommand{\Zp}{\mathbb{Z}_p}

\newcommand{\absolute}[1]{\left\lvert#1\right\rvert}
\definecolor{Green}{rgb}{0.0, 0.5, 0.0}

\newcommand{\Qcyc}{\Q_{\cyc}}

\newcommand{\fM}{{\mathfrak{M}}}
\newcommand{\Char}{\mathrm{Char}}

\newcommand{\Fin}{\cF_\infty}

\newcommand{\op}[1]{\operatorname{#1}}

\newcommand{\mup}{\mu_p(E/\cF_\infty)}
\title{Arithmetic Statistics and noncommutative Iwasawa Theory}

\makeatletter
\let\@wraptoccontribs\wraptoccontribs
\makeatother

\author[D.~Kundu]{Debanjana Kundu}
\address[Kundu]{Department of Mathematics \\ University of British Columbia\\
Vancouver, BC, Canada V6T 1Z2}
\email{dkundu@math.ubc.ca}

\author[A.~Lei]{Antonio Lei}
\address[Lei]{D\'epartement de Math\'ematiques et de Statistique\\
Universit\'e Laval, Pavillion Alexandre-Vachon\\
1045 Avenue de la M\'edecine\\
Qu\'ebec, QC\\
Canada G1V 0A6}
\email{antonio.lei@mat.ulaval.ca}

\author[A.~Ray]{Anwesh Ray}
\address[Ray]{Department of Mathematics \\ University of British Columbia\\
Vancouver, BC, Canada V6T 1Z2}
\email{anweshray@math.ubc.ca}

\subjclass[2020]{11R23, 11G05}
\keywords{Arithmetic statistics, noncommutative Iwasawa theory, Selmer groups, Euler characteristics, Akashi series, growth of Mordell--Weil ranks}

\begin{document}
\begin{abstract}
Let $p$ be an odd prime.
Associated to a pair $(E, \mathcal{F}_\infty)$ consisting of a rational elliptic curve $E$ and a $p$-adic Lie extension $\mathcal{F}_\infty$ of $\Q$, is the $p$-primary Selmer group $\Sel_{p^\infty}(E/\mathcal{F}_\infty)$ of $E$ over $\Fin$.
In this paper, we study the arithmetic statistics for the algebraic structure of this Selmer group.
The results provide insights into the asymptotics for the growth of Mordell--Weil ranks of elliptic curves in noncommutative towers.
\end{abstract}

\maketitle

\section{Introduction}
\label{S: Intro}
Given an elliptic curve $E$ defined over a number field $F$, the Mordell--Weil theorem states that the group of $F$-rational points, called the Mordell--Weil group and denoted by $E(F)$, is finitely generated (see \cite{Mor22, Wei29}).
A central question in the arithmetic of elliptic curves is the precise structure of this group, in particular its rank.
A motivating problem in Iwasawa theory is the question of determining the growth of the Mordell--Weil rank in certain infinite towers of number fields.
Such questions were first studied by B.~Mazur in \cite{mazur72}, where he showed that for a class of elliptic curves defined over $\Q$, the Mordell--Weil rank remains bounded in the cyclotomic $\Z_p$-extension $\Q_{\cyc}/\Q$.
The proof of Mazur  involved a thorough analysis of the $p$-primary Selmer group of $E$ over $\Qcyc$.
This result has been  extended to all rational elliptic curves by K.~Kato \cite{Kato} and D.~Rohrlich \cite{rohrlich88} who showed that given any elliptic curve $E$ defined over $\Q$, the rank of $E$ is bounded in $F_{\cyc}/F$, where $F/\Q$ is an abelian extension.

The study of Iwasawa theory in noncommutative $p$-adic Lie extensions was initiated by M.~Harris in \cite{harris,harris2}.
In the late 1990's and early 2000's, noncommutative Iwasawa theory became an active area of research, leading to a series of breakthrough results (see \cite{BH97, coateshowson, venjakob2002structure, OV02, CoatesSchneiderSujatha_Links_between, Gre03,HV03, OV03, venjakob2003iwasawa,CKFVS,CFKS}).
Such extensions are ubiquitous and \emph{often} arise naturally.
One such example is the following.
Let $F$ be a number field and $E_{/F}$ be an elliptic curve without complex multiplication (CM).
Consider the extension $\Fin/F$ given by the field of definition of the $p$-power torsion
points on $E$.
By Serre's Open Image Theorem~\cite{Ser72}, the Galois group $\Gal(\Fin/F)$ is isomorphic to a finite-index subgroup of $\GL_2(\Zp)$.
Hence, it is a non-abelian, $p$-adic Lie group of dimension 4.
We will often specialize our results to this widely-studied infinite extension, called the \emph{trivializing extension} of $E$.
We emphasize that the methods of classical (abelian) Iwasawa theory do not extend in any obvious fashion to the noncommutative theory and there are several pitfalls if one follows such an approach, see \cite[\S 2]{BH97} for a discussion.

The main object of study in noncommutative Iwasawa theory is the Selmer group defined over an infinite $p$-adic Lie-extension.
In this article, we are interested in investigating the cohomology groups of the Selmer group and calculating the \emph{Euler characteristic}, which is defined in terms of these cohomology groups, see \cite{howson2002euler, zerbes2009generalised, zerbes11} and Section \ref{S: alg prelim}.
Under appropriate hypotheses, the Euler characteristic of  Selmer groups of elliptic curves over $\Qcyc$ can be expressed in terms of invariants arising in the $p$-adic Birch--Swinnerton-Dyer (BSD) formula, see Section \ref{S: BSD and akashi}.
In the noncommutative setting, the formula is more involved.
In addition to the invariants arising from the BSD formula, this formula has  contributions from local Euler factors at specified auxiliary primes, see Theorem~\ref{EC formula theorem}.
We shall study these new invariants from the point of view of  arithmetic statistics.
The intricate relationship between the Euler characteristic formula and Iwasawa theoretic invariants coming from noncommutative Iwasawa theory gives rise to several new questions and provides us with a fertile ground for investigation on the structure of Selmer groups via the lens of arithmetic statistics.
We explain this further in the coming paragraphs.

Throughout, $E$ will denote an elliptic curve defined over $\Q$ with good \emph{ordinary} reduction at a prime $p\geq 5$.
Let $F$ be a number field and $\cF_\infty$ be an infinite Galois extension such that $G:=\Gal(\cF_\infty/ F)$ is a \emph{uniform} pro-$p$ group of dimension $d$.
Furthermore, it is assumed that $G$ is \emph{admissible}, i.e., $\cF_\infty$ contains the cyclotomic $\Z_p$-extension of $F$, is ramified at only finitely many primes, and $G$ contains no non-trivial $p$-torsion.
Structural properties of the $p$-primary Selmer group $\Sel_{p^\infty}(E/\cF_\infty)$ have been studied by O.~Venjakob in \cite{venjakob2002structure}.
The Selmer group is a module over the Iwasawa algebra $\Lambda(G)$, which is a noetherian Auslander regular local ring.
Much like modules over $\Z_p\llbracket T\rrbracket$, the dimension theory of modules over $\Lambda(G)$ is well understood.
A module is \emph{torsion} (resp. \emph{pseudonull}) if the codimension of its annihilator in $\Lambda(G)$ is $<d+1$ (resp. $<d$).
In this article, we are interested in understanding \emph{how often} the Selmer group is pseudonull over a pro-$p$ $p$-adic Lie extension.

In Propositions~\ref{prop: pseudonull} and \ref{proposition: akasi series to pseudonull}, we explain the relationship between the Euler characteristic and the \emph{Akashi series} (introduced in \cite{CoatesSchneiderSujatha_Links_between}).
For the Selmer group over the cyclotomic $\Z_p$-extension, the Akashi series is simply the characteristic element.
This generalized invariant provides deep insight into the algebraic structure of the Selmer group.
The Euler characteristic is closely related to the leading term of the Akashi series (see Proposition~\ref{prop: leading term}).
In the setting where the elliptic curve has Mordell--Weil rank 0, we utilize this interplay between the different invariants to prove statistical results on the pseudonullity of the Selmer group.
One of the main deviations from the case of the cyclotomic $\Zp$-extension is that  the extension $\Fin/\Q$ is usually not pro-$p$.
Consequently, many results in noncommutative Iwasawa theory do not apply since one often works with the Iwasawa algebra of a pro-$p$ group.
Nonetheless, even in the non-pro-$p$ setting, we can prove results about the Euler characteristics and the Akashi series.
In special cases, 
we can even infer that the $p$-primary Selmer group is in fact \emph{trivial} over the infinite extension.
As a by-product, we obtain examples where a conjecture of J.~Coates and R.~Sujatha on the pseudonullity of \emph{fine Selmer groups} is true, see \cite[Conjecture~B]{CoatesSujatha_fineSelmer}.

When the Mordell--Weil rank of $E(F)$ is 0, we will see in the course of this article that there are instances when the Iwasawa invariants for the Selmer group over $\cF_\infty$ vanish.
Some results in this direction are proved in Theorems~\ref{theorem varying imaginary quadratic}, \ref{thm: Hep-Gep density}, and \ref{theorem for vary p when FTC}(2),(3).
The arithmetic statistics of Iwasawa invariants of elliptic curves for the cyclotomic $\Z_p$-extensions have been studied in \cite{KR21, KR21b} by the first and third named authors of the present article.
In \cite{HKR}, these results have also been extended to the anticyclotomic setting by the first and third named authors in collaboration with J.~Hatley. In subsequent work, these methods shall be further developed to study statistics for the \emph{fine Selmer group} by the third named author. Unlike the cyclotomic $\Zp$-extension, primes other than  $p$ can ramify in $\cF_\infty$ in the noncommutative case.
This makes the task of determining Iwasawa invariants in the noncommutative case more challenging and intricate than  the cyclotomic case.
Another diverging point from the cyclotomic theory is that it is possible to vary $\Fin$ over certain natural infinite families even when the prime $p$ is fixed.
As in the cyclotomic setting, we can study the variation of Iwasawa invariants arising from the noncommutative setting via  statistical analysis.
More precisely, given a triple $(E, p, \cF_\infty)$, we study the variation of the algebraic structure of the Selmer group $\Sel_{p^\infty}(E/\cF_\infty)$ in three different contexts.
\begin{enumerate}[(a)]
    \item We fix the pair $(E,p)$ and let $\cF_\infty$ vary over a family of admissible extensions.
    \item We fix a pair $(p, \cF_\infty)$ and let $E$ vary over a subset of elliptic curves $E_{/\Q}$ of rank 0.
    \item We fix an elliptic curve $E$ and associate to each prime $p$, an extension $\cF_\infty$ in a natural way.
    Then, we vary $p$ over the primes at which $E$ has good \emph{ordinary} reduction.
\end{enumerate}
We study each of these three questions in three distinct settings.
\begin{enumerate}
    \item First, we consider the $\Zp^2$-extension of imaginary quadratic fields.
    This is a 2-dimensional \emph{abelian} extension and a \textit{metabelian} extension over $\Q$.
    Through the course of our investigation, we see that this case parallels the cyclotomic theory, see for example \eqref{ec are equal}, which says that the Euler characteristics for the $\Zp^2$-extension and the cyclotomic $\Zp$-extension coincide.
    \item Next, we specialize to the simplest noncommutative $2$-dimensional $p$-adic Lie extension, namely the false Tate curve extension.
    Given primes $p$ and $\ell$, we write
    \[
    \cF_\infty:=\Q(\mu_{p^\infty}, \  \ell^{\frac{1}{p^n}}: n = 1, 2, \cdots).
    \]
    Let us explain the three questions of interest in this case.
    \begin{enumerate}[(a)]
       \item We fix an elliptic curve $E$ of conductor $N_E$ and a prime $p$ of good \emph{ordinary} reduction of $E$.
       We consider the family of false Tate curve extensions obtained by varying $\ell\nmid N_E p$.
       In Theorem~\ref{thm: Hep-Gep density}, we study for what proportion of primes $\ell$ is the Selmer group trivial over $\Fin$.
       \item We fix the primes $p$ and $\ell$, and let $E$ vary over all elliptic curves defined over $\Q$ ordered by height.
       In Theorem~\ref{FTC thm varying elliptic curve}, we calculate an upper bound for the proportion of elliptic curves for which the Selmer group is \emph{not} trivial.
       \item We fix a rank 0 non-CM elliptic curve $E_{/\Q}$, a good prime $\ell$, and let $p$ vary over the primes at which $E$ has good \emph{ordinary} reduction.
       In Proposition~\ref{FTC: G EC formula same as Gamma EC formula}, we show that for \emph{at least half} of the primes $p$, the $G$-Euler characteristic coincides with the $\Gamma_F$-Euler characteristic.
       When $E$ has good \emph{supersingular} reduction at $\ell$, we show that this happens for \emph{exactly two-third} of the primes $p$ (see Proposition~\ref{prop:FT-ss}).
       For such primes $p$, the Selmer group  over the false Tate curve extension is trivial if and only if that over the cyclotomic $\Zp$-extension of $\Q(\mu_p)$ is trivial.
    \end{enumerate}
    \item We consider the trivializing extension $\cF_\infty/\Q$, generated by the $p$-primary torsion points of a non-CM elliptic curve (denoted by $A$, $E_0$ and $E'$ in the three questions we study).
    Since $\mathcal{G}:=\Gal(\Fin/\Q)$ is \emph{not} a pro-$p$ extension, our results on the $\mathcal{G}$-Euler characteristic formula do not imply pseudonullity of the Selmer group over the infinite extension.
    We prove the following results.
    \begin{enumerate}[(a)]
        \item We fix a rank 0 elliptic curve $E$ of conductor $N_E$ and a prime $p$ of good \emph{ordinary} reduction of $E$.
        We consider the family of extensions obtained by varying a non-CM elliptic curve $A_{/\Q}$.
        In Theorem~\ref{s7.3 mainthm}, we show that for density 0 (but infinitely many) such elliptic curves $A$, the $\mathcal{G}$-Euler characteristic is trivial.
        \item We fix $p$ and a non-CM elliptic curve ${E_0}_{/\Q}$.
        This fixes the $p$-adic Lie extension $\Q(E_0[p^\infty])/\Q$.
        As $E$ varies over all elliptic curves defined over $\Q$ and ordered by height, we calculate an upper bound for the proportion of elliptic curves for which the $\mathcal{G}$-Euler characteristic is \emph{not} trivial in Theorem~\ref{triv case: vary E}.
        \item For a pair of elliptic curves $(E,E^\prime)$ such that $E^\prime$ does not have CM, we consider the Selmer group of $E$ over the $p$-adic Lie extension $\Q(E^\prime[p^\infty])/\Q$ as $p$ varies.
        In Theorem~\ref{vary p main thm in triv case}, we show that for all but finitely many primes, the $\mathcal{G}_{E^\prime}$-Euler characteristic is equal to the $\Gamma_{\Q}$-Euler characteristic.
        This latter quantity is expected to be trivial \emph{most of the time} (see \cite[Conjecture~3.17]{KR21}).
    \end{enumerate}
\end{enumerate}

Similar to the cyclotomic setting discussed at the beginning of the introduction,  the structure of  the Selmer group over a $p$-adic Lie extension plays a crucial role in understanding the rate of growth of the Mordell--Weil rank of an elliptic  in towers of non-abelian extensions (see \cite{Bha07, darmontian,DL,DL2,HL,leisprung}).
In some cases, the Mordell--Weil ranks can be described very precisely, see in particular \cite{darmontian}, where special cases of  false Tate curve extensions have been studied.
More recently, the third named author has made progress in proving refined asymptotic bounds on the growth of Mordell--Weil ranks in general noncommutative towers, see \cite{Ray21_noncom_rank}.

Let $H := \Gal(\Fin/F_{\cyc})$.
To study the structure of the Selmer group, it is standard in noncommutative Iwasawa theory to assume the $\mathfrak{M}_H(G)$-conjecture, see Conjecture~\ref{MHG Conjecture} for a precise statement.
Conditional on this conjecture, in \cite{HL}, P.~C.~Hung and M.~F.~Lim have shown a close relationship between the structural invariants of the Selmer group over $\cF_\infty$, the pseudonullity of a certain quotient of the Selmer group over $\cF_\infty$, 
and the growth of Mordell--Weil ranks of $E$ inside this extension.
These results allow us to distinguish between the pseudonullity of the $p$-primary Selmer group and the aforementioned quotient (see Remark~\ref{implication pseudonullity}), thereby allowing us to prove refined estimates on the growth of Mordell--Weil ranks.
In particular, we can show in some cases (see for example, Proposition~\ref{positive Lambda H rank} and Corollary \ref{lambda H rank cor for varying E in FTC}) that the $p$-primary Selmer group is \emph{not} pseudonull over a noncommutative admissible pro-$p$ $p$-adic Lie extension even when the cyclotomic Euler characteristic is trivial.
In special cases (see for example Corollaries \ref{MW cor vary F Zp2} and \ref{MW rank corollary, vary p Zp2}), we prove how often the Mordell--Weil rank remains bounded at every finite layer of an infinite extension.

The structure of the paper is as follows.
Sections 2 to 7 are mostly reviews in nature.
Our main results are presented in Sections 8 to 10.
In \S\ref{S: Preliminaries}, we introduce the notation and definitions that will be used throughout the article.
In particular, we review the definition of Selmer groups, Iwasawa algebras and other related notions.
Next, we review various notions and basic properties in both commutative and noncommutative Iwasawa theory in \S\S\ref{S: alg prelim}-\ref{S: MW}, including Euler characteristics,  Akashi series, the $\fM_H(G)$-conjecture, as well as recent results on the asymptotic growth of Mordell--Weil ranks of an elliptic curve inside a $p$-adic Lie extension.
In \S\ref{S: examples of extensions}, we discuss the three families of $p$-adic Lie extensions over which we study the Iwasawa-theoretic properties of elliptic curves.
In \S\ref{S: Tamagawa calculations}, we review results on the behaviour of Tamagawa numbers under extensions of number fields, used in later sections of the article.
Our main results are proved in \S\S\ref{S: Vary extension}-\ref{S: Vary prime}, where we study arithmetic statistics of a fixed elliptic curve as the $p$-adic Lie extension varies, of families of elliptic curves over a fixed $p$-adic Lie extension, and of a fixed elliptic curve over families of $p$-adic Lie extensions as $p$ varies, respectively.
In Appendix \ref{appendix}, we discuss a classification of conjugacy classes in the finite group $\GL_2(\Z/p\Z)$, which is relevant to our discussion in \S\ref{S: Vary extension}.

\subsection*{Acknowledgements}
The first and second named authors thank Meng Fai Lim for his comments on an earlier draft of this article.
The first named author thanks Rahul Arora and R.~Sujatha for valuable inputs.
The second named author thanks Vorrapan Chandee, Chantal David, Xiannan Li and Meng Fai Lim for answering his questions during the preparation of the article.
The first named author acknowledges the support of the PIMS Postdoctoral Fellowship.
The second named author is supported by the NSERC Discovery Grants Program RGPIN-2020-04259 and RGPAS-2020-00096.
The authors thank the anonymous referee for careful and timely reading of an earlier version of the article and for pointing out corrections as well as numerous suggestions leading to improvements in the exposition of the article.

\section{Preliminaries}
\label{S: Preliminaries}

\subsection{}
\label{S: Selmer groups}
Throughout this article,  $p\geq 5$ is a prime number and $E$ is an elliptic curve over $\Q$ with good \emph{ordinary} reduction at $p$.
The prime $p$ is \emph{not} fixed forever, and in many settings, we shall vary $p$  in a suitable sense.
Moreover, even when $p$ is fixed, the estimates obtained in this article will crucially depend on $p$.
For $n\in \Z_{\geq 1}$, denote by $E[p^n]$ the $p^n$ torsion subgroup of $E(\bar{\Q})$.
We shall set $E[p^{\infty}]$ to be the union of $E[p^n]$ as $n$ ranges over $\Z_{\geq 1}$.
Let $S$ be a finite set of prime numbers containing $p$ and the primes at which $E$ has bad reduction.
Denote by $\Q_S$ the maximal algebraic extension of $\Q$ at which all primes $\ell\notin S$ are unramified.
Given a number field extension $F$ of $\Q$ contained in $\Q_S$, set $G_{F,S}:=\Gal(\Q_S/F)$.
Given a module $M$ over $G_{F,S}$, and $i\geq 0$, the cohomology group $H^i(\Q_S/F, M)$ is defined to be the discrete cohomology group $H^i(G_{F,S}, M)$.
For $n\geq 0$, let $\Q_{(n)}$ be the unique degree $p^n$-extension of $\Q$ contained in $\Q(\mu_{p^{n+1}})$.
We use $\Q_{(n)}$ instead of $\Q_n$ to avoid conflict in notation, since when $n=\ell$ is a prime, $\Q_\ell$ also denotes the $\ell$-adic numbers.
Also, note that the role of $p$ is suppressed in this notation.
However, we shall not suppress the role of $p$ when we introduce the Selmer groups that are studied.
The cyclotomic $\Z_p$-extension of $\Q$ is taken to be the union
\[
\Q_{\cyc}:=\bigcup_{n\geq 0} \Q_{(n)}.
\]
The Galois group $\Gal(\Qcyc/\Q)$ will be denoted by $\Gamma$.
For a number field $F$, we set $F_{\cyc}=F\cdot \Q_{\cyc}$ to be the cyclotomic $\Zp$-extension of $F$ and write $\Gamma_F:=\Gal(F_{\cyc}/F)$. Its $n$-th layer is the unique sub-extension $F_n$ such that $[F_n:F]=p^n$. Note that $F_n$ is contained in $F_{n+1}$ and there are isomorphisms of topological groups
\[
\Gal(F_{\cyc}/F)\xrightarrow{\sim} \varprojlim_n\Gal(F_{n}/F)\xrightarrow{\sim} \Zp.
\]
Further, when $F\cap \Q_{\cyc}=\Q$, $F_n = F\cdot \Q_{(n)}$.

Henceforth,  $\Fin/F$ denotes a pro-$p$, $p$-adic Lie extension of $F$.
In other words, as a topological group, $G:=\Gal(\Fin/F)$ is isomorphic to a pro-$p$ $p$-adic Lie group.
Furthermore, we shall require that $\Fin/F$ is \textit{admissible}, i.e., the following conditions are satisfied.
\begin{enumerate}[(a)]
 \item $\Fin$ contains $F_{\cyc}$, 
 \item $\Fin$ is ramified at finitely many primes, \emph{and}
 \item $G$ does not contain any non-zero elements of order $p$.
\end{enumerate}
Throughout, set $H:=\Gal(\Fin/F_{\cyc})$ and identify $G/H$ with $\Gamma_F$.

\subsection{}
Without loss of generality, assume that $S$ contains the set of primes that ramify in $\Fin$.
Let $L$ be a number field in $\Q_S$.
For $\ell\in S$, define the local condition at $\ell$ as follows
\[
J_\ell(E/L):= \bigoplus_{w|\ell} H^1\left( L_w, E\right)[p^\infty].
\]
In the above sum, $w$ runs through all primes of $L$ above $\ell$, and $L_w$ denotes the completion of $L$ at $w$.
The \textit{$p$-primary Selmer group} of $E$ over $L$ is defined as the kernel of the following restriction map
\[
\Sel_{p^\infty}(E/L):=\ker\left\{ H^1\left(\Q_S/L,E[p^{\infty}]\right)\xrightarrow{\Phi_{E,L}} \bigoplus_{\ell\in S} J_\ell(E/L)\right\}.
\]
Taking direct limits, the $p$-primary Selmer group of $E$ over $\Fin$ is defined to be
\[
\Sel_{p^\infty}(E/\Fin) := \varinjlim_{ L \subseteq \cF_{\infty}} \Sel_{p^\infty}(E/L),
\]
where $L$ runs through all number fields contained in $\Fin$.

\subsection{}
The \emph{Iwasawa algebra} $\Lambda(G)$ is the inverse limit of group rings
\[\Lambda(G):=\varprojlim_U \Z_p[G/U],\] where $U$ runs through all normal finite index subgroups of $G$.
The Iwasawa algebra $\Lambda(\Gamma)$ is defined similarly.
On choosing a topological generator $\gamma\in \Gamma$, we fix the ring-isomorphism $\Lambda(\Gamma)\simeq \Z_p\llbracket T\rrbracket$ identifying $\gamma-1$ with $T$.

It is shown in \cite[Theorem~3.26]{venjakob2002structure} that $\Lambda(G)$ is an \textit{Auslander regular local ring}.
Thus, there is an adequate dimension theory for modules over $\Lambda(G)$.
By a result of A.~Neumann (see \cite{Neu88}), it is known that the Iwasawa algebra $\Lambda(G)$ has no zero-divisors.
Note that unlike in the commutative case, it is possible that a noncommutative ring with no zero-divisors does not admit a skew field, see \cite[Chapter 4 \S{9B}]{Lam}.
However, a well-known result of M.~Lazard asserts that $\Lambda(G)$ is noetherian.
As a result, $\Lambda(G)$ admits a skew field by \cite[Chapter 4, Sections 9 and 10]{Lam}, which we shall denote by $Q(G)$.
Let $M$ be a module over $\Lambda(G)$, it is said to be \emph{finitely generated} (resp. \emph{torsion}) if $\dim_{Q(G)}\left(M\otimes_{\Lambda(G)} Q(G)\right)$ is finite (resp. zero).
The rank of $M$ as a $\Lambda(G)$-module is defined as
\[
\rank_{\Lambda(G)} M:=\dim_{Q(G)} \left(Q(G)\otimes_{\Lambda(G)} M\right).
\]
An application of Nakayama's lemma shows that the \emph{Pontryagin dual}
\[
\Sel_{p^\infty}(E/\Fin)^{\vee}:=\Hom(\Sel_{p^\infty}(E/\Fin), \Q_p/\Zp)
\]
is finitely generated as a $\Lambda(G)$-module.
By the result of Kato mentioned in the introduction, $\Sel_{p^\infty}(E/F_{\cyc})^{\vee}$ is a torsion $\Lambda(\Gamma)$-module  if $F/\Q$ is abelian.
Throughout, we make an analogous assumption for the extension $\Fin$.

\begin{Assumption}
\label{assumption: torsion}
Assume that $\Sel_{p^\infty}(E/\Fin)^{\vee}$ is a torsion $\Lambda(G)$-module.
\end{Assumption}

\begin{remark}
\label{torsion remark}
A result of P.~N.~Balister and S.~Howson (see \cite{BH97} or \cite[Lemma~2.6]{HO10}) says that if $G$ is a uniformly powerful, solvable group containing a closed normal subgroup $H$ such that $G/H \simeq \Zp$, then a finitely generated $\Lambda(G)$-module is torsion if $M_H$ is $\Lambda(\Gamma)$-torsion.
This can be used to show that 
if $\Sel_{p^\infty}(E/F_{\cyc})$ is $\Lambda(\Gamma)$-cotorsion, then $\Sel_{p^\infty}(E/\Fin)$ is $\Lambda(G)$-cotorsion (see \cite[Theorem~2.3]{HO10}).
\end{remark}

For a finitely generated torsion $\Lambda(G)$-module $M$, let $M(p)$ denote the $p$-primary torsion subgroup of $M$ and set 
\[
M_f:=M/M(p).
\]
Since the ring $\Lambda(G)$ is noetherian, one can find $r\in \Z_{\geq 1}$ such that $p^r$ annihilates $M(p)$.
Let $\Omega(G)$ denote the mod-$p$ reduction of the Iwasawa algebra $\Lambda(G)$.
This group algebra has no non-trivial zero divisors and hence, admits a skew field of fractions, see for example \cite{dixon}.
This implies that the notion of $\Omega(G)$-rank makes sense.
Now, following \cite{howson2002euler}, we define the $\mu$-invariant of $M$ by 
\[
\mu_p(M):=\sum_{i=0}^r \rank_{\Omega(G)} \left(p^i M(p)/p^{i+1}\right).
\]
Henceforth, we denote by $\mup$ the $\mu$-invariant of the Selmer group $\Sel_{p^\infty}(E/\Fin)^{\vee}$ as a $\Lambda(G)$-module.

\section{The Euler Characteristic}
\label{S: alg prelim}

If $M$ is any discrete cofinitely generated $p$-primary $\Lambda(G)$-module, we say that $M$ has finite $G$-Euler characteristic if the cohomology groups $H^i(G, M)$ are finite for all $i\ge0$.
Then, the (classical) \emph{Euler characteristic} $\chi(G, M)$ is defined as follows
\[
\chi(G, M)=\prod_{i\geq 0}\left( \# H^i(G,M)\right)^{(-1)^i}.
\]
For ease of notation, set 
\begin{align*}
 \chi(\Gamma, E,p)&:= \chi\left(\Gamma, \Sel_{p^\infty}(E/F_{\cyc})\right) \textrm{ and } \\
 \chi(G, E,p) &:= \chi\left( G, \Sel_{p^\infty}(E/\Fin)\right).
\end{align*}
When the cohomology groups $H^i(G, M)$ are \emph{not} finite, there is a generalization of the above notion.
When $G=\Gamma$, we identify $H^1(\Gamma, M)$ with the module of co-invariants $M_{\Gamma}$.
There is an obvious map 
\[
\Phi_M:M^{\Gamma}\rightarrow M_{\Gamma}
\]
sending $m$ to its residue class.
The \emph{truncated $\Gamma$-Euler characteristic} is well-defined if both $\ker \Phi_M$ and $\coker \Phi_M$ are finite, and it is given by 
\[
\chi_t(\Gamma, M):=\frac{\# \ker \Phi_M}{\# \coker \Phi_M}.
\]
Following the discussion on \cite[pp.~779-780]{zerbes2009generalised}, we recall the generalization of this notion to $\Lambda(G)$.
For a discrete $p$-primary $G$-module $M$, let $\mathfrak{d}_M^0$ be the composite of the maps 
\[
\mathfrak{d}_M^0:H^0(G,M)=H^0(\Gamma, M^H)\xrightarrow{\Phi_M}H^1(\Gamma, M^H)\hookrightarrow H^1(G, M),
\]
where the last map is the inflation.
For $j\geq 1$, define $\mathfrak{d}_M^j$ as the composite
\[
\mathfrak{d}_M^j:H^i(G, M)\rightarrow H^0\left(\Gamma, H^i(H, M)\right)\xrightarrow{\Phi_{H^i(H, M)}} H^1\left(\Gamma, H^i(H, M)\right)\hookrightarrow H^{i+1}(G, M).
\]
Let $\mathfrak{d}_{M}^{-1}$ denote the $0$-map.
Note that $\left(H^i(G, M), \mathfrak{d}_M^j\right)$ forms a complex, and we denote its $j$-th cohomology group by $\mathfrak{h}_j$.
\begin{definition}
The truncated (or generalized) Euler characteristic of a cofinitely generated $p$-primary $\Lambda(G)$-module $M$ is defined if the cohomology groups $\mathfrak{h}_j$ are all finite.
In this case, the \emph{truncated $G$-Euler characteristic} is defined as follows
\[
\chi_t(G, M):=\prod_j (\#\mathfrak{h}_j)^{(-1)^j}.
\]
\end{definition}

The terminology `generalized Euler characteristic' was introduced in \cite{zerbes2009generalised}.
In earlier works, such as \cite{CoatesSchneiderSujatha_Links_between, CKFVS}, it was referred to as `truncated Euler characteristic'.
We shall refer to this Euler characteristic as the truncated Euler characteristic, which is also consistent with the terminology used by the third named author in \cite{ray2020euler2, ray2021euler1}, where the behaviour of these invariants with respect to congruences is studied.

When the cohomology groups $H^i(G, M)$ are finite for all $i\ge0$, the truncated Euler characteristic $\chi_t(G, M)$ coincides with the usual Euler characteristic $\chi(G, M)$.
We now give a criterion for the $\Gamma$-Euler characteristic for Selmer groups over the cyclotomic $\Zp$-extension to be well-defined.

As in the case with the classical Euler characteristic $\chi(\cdot, \cdot)$, we adopt a similar shorthand for the truncated Euler characteristic, setting
\begin{align*}
\chi_t(\Gamma, E,p)&:= \chi_t\left(\Gamma, \Sel_{p^\infty}(E/F_{\cyc})\right) \textrm{ and }\\ 
\chi_t(G, E,p)&:= \chi_t\left(G, \Sel_{p^\infty}(E/\Fin)\right).
\end{align*}

\begin{lemma}
\label{lemma EC defined}Assume that $\Sha(E/F)[p^\infty]$ is finite.
The following conditions are equivalent.
\begin{enumerate}
 \item The classical $\Gamma$-Euler characteristic $\chi(\Gamma, E,p)$ is well-defined.
 \item $\Sel_{p^\infty}(E/F_{\cyc})^{\Gamma}$ is finite.
 \item The Selmer group $\Sel_{p^{\infty}}(E/F)$ is finite.
 \item The Mordell--Weil group $E(F)$ is finite.
\end{enumerate} 
\end{lemma}

\begin{proof}
The proof presented in \cite[Lemma~3.2]{KR21} can be adapted for any number fields.
\end{proof}

Let $M$ be a cofinitely generated cotorsion $\Lambda(\Gamma)$-module.
Express the characteristic element of $M^\vee$, denoted by $f_{M}(T)$, as a polynomial
\[
f_{M}(T)=c_0+c_1T+\dots +c_{d-1} T^{d-1}+T^d.
\]
Let $r_{M}$ denote the order of vanishing of $f_{M}(T)$ at $T=0$.
For $a,b\in \Qp$, we write $a\sim b$ if there is a unit $u\in \Zp^{\times}$ such that $a=bu$.

\begin{lemma}
\label{lemmazerbes}
Let $M$ be a cofinitely generated cotorsion $\Lambda(\Gamma)$-module.
Assume that the kernel and cokernel of $\Phi_{M}$ are finite.
Then,
\begin{enumerate}
\item ${r_{M}}=\corank_{\Zp}(M^{\Gamma})=\corank_{\Zp}(M_{\Gamma})$.
\item $c_{r_{M}}\neq 0$.
\item $c_{r_{M}}\sim \chi_t(\Gamma, M)$.
\end{enumerate}
Here, $c_{r_M}$ is the coefficient of $T^{r_M}$ in $f_M(T)$.
\end{lemma}

\begin{proof}
See \cite[Lemma~2.11]{zerbes2009generalised}.
\end{proof}

\section{Birch and Swinnerton-Dyer formulas and Akashi Series}
\label{S: BSD and akashi}

Let $E$ be an elliptic curve defined over $\Q$.
Fix a number field extension $F/\Q$ and consider the base-change of $E$ to $E_{/F}$.
Let $\Gamma_F$ denote the Galois group $\Gal(F_{\cyc}/F)$.
Recall that $G$ is the Galois group $\Gal(\Fin/F)$ for some $p$-adic Lie extension $\Fin$ of $F$.
We discuss explicit formulas for the truncated Euler characteristic $\chi_t(\Gamma_F, E,p)$ and $\chi_t(G, E,p)$.
These formulas are motivated by the $p$-adic Birch and Swinnerton-Dyer conjecture.

\subsection{}
It follows from Lemma~$\ref{lemmazerbes}$ that the truncated $\Gamma_F$-Euler characteristic, when defined, is always an integer.
By Lemma~\ref{lemma EC defined}, the $\Gamma_F$-Euler characteristic $\chi(\Gamma_F, M)$ is defined if and only if $r_{M}=0$.
In this case, the constant coefficient $c_0$ of the characteristic element of $\Sel_{p^\infty}(E/F_\cyc)^\vee$ satisfies $c_0\sim \chi(\Gamma_F, M)$.
Furthermore, we have the following formula (see \cite[Chapter 3]{CoatesSujatha_book}):
\[
\chi(\Gamma_F, E,p) \sim \frac{\#\Sha(E/F)[p^\infty]\cdot \prod_{v\nmid p}c_v^{(p)}(E/F)}{\left(\# E(F)[p^\infty] \right)^2} \cdot \prod_{v|p}\left(\# \widetilde{E}(\kappa_v)_{p^\infty} \right)^2.
\]
Here, $\Sha(E/F)$ is the Tate--Shafarevich group, which is assumed to be finite throughout this article.
At a finite prime $v$ of $F$, the residue field is denoted by $\kappa_v$.
Let $\absolute{\cdot}_p$ be the absolute value on $\bar{\Q}_p$ normalized by setting $\absolute{p}_p^{-1}=p$.
The notation $c_v(E/F)$ is used for the Tamagawa number at $v\nmid p$, and $c_v^{(p)}(E/F)$ is its $p$-part, given by 
\[
c_v^{(p)}(E/F):=\absolute{c_v(E/F)}_p^{-1}.
\]
At a prime $v|p$, denote by $\widetilde{E}$ the reduction of $E$ at $v$ and $\widetilde{E}(\kappa_v)$ be the group of $\kappa_v$-valued points on $\widetilde{E}$.
The next result provides conditions for the truncated $\Gamma$-Euler characteristic to be defined.

\begin{lemma}
\label{truncdefined}
Let $M$ be a $p$-primary cotorsion $\Lambda(\Gamma)$-module.
Let $f_1(T), \dots, f_n(T)$ be distinguished polynomials such that $M^{\vee}_f$ is pseudo-isomorphic to $\bigoplus_{i=1}^n \Lambda(\Gamma)/(f_i(T))$.
If $T^2 \nmid f_i(T)$ for all $i$, then the kernel and cokernel of $\Phi_M$ are finite and the truncated $\Gamma$-Euler characteristic $\chi_t(\Gamma, M)$ is defined.
In particular, $\chi_t(\Gamma, M)$ is defined when $r_M\leq 1$.
\end{lemma}

\begin{proof}
It follows from the proof of \cite[Lemma~2.11]{zerbes2009generalised}.
\end{proof}

\subsection{}
When $E$ has good \emph{ordinary} reduction at $p$, there is a $p$-adic analog of the usual height pairing, which was studied extensively by P.~Schneider in \cite{Schneider82, Schneider85}.
This $p$-adic height pairing is conjectured to be non-degenerate, and its determinant is called the \emph{$p$-adic regulator} (denoted by $\Reg_p(E/F)$).
In Iwasawa theory, it is standard to use the following normalized $p$-adic regulator, which is well-defined up to a $p$-adic unit
\[
\mathcal{R}_p(E/F) = \frac{\Reg_p(E/F)}{p^{\rank_{\Z} E(F)}}.
\]
The following result gives a formula for the truncated $\Gamma_F$-Euler characteristic of the $p$-primary Selmer group (when it is defined).
In the CM case, this was proven by B.~Perrin-Riou (see \cite{PR82}) and in the general case by Schneider (see \cite{Schneider85}).

\begin{theorem}
\label{pbsdconj}
Assume that the elliptic curve $E$ has good \emph{ordinary} reduction at $p$.
The order of vanishing of the characteristic element $f_E(T)$ of $\Sel_{p^\infty}(E/F_{\cyc})^\vee$ at $T=0$ is at least equal to $\rank_{\Z}E(F)$.
Furthermore, if
\begin{enumerate}[(i)]
 \item $\mathcal{R}_p(E/F)\neq 0$,
 \item $\Sha(E/F)[p^{\infty}]$ is finite,
\end{enumerate}then, \[\ord_{T=0} f_E(T)=\rank_{\Z} E(F).\] Further, if the truncated $\Gamma_F$-Euler characteristic $\chi_t(\Gamma_F, E,p)$ is defined, then, one has the following $p$-adic Birch and Swinnerton-Dyer formula for the truncated Euler characteristic
\begin{equation}
\label{BSDformula}
\chi_{t}(\Gamma_F, E,p) \sim \mathcal{R}_p(E/F) \times \frac{\# \Sha(E/F)[p^\infty] \times \prod_{v\nmid p} c_{v}^{(p)}(E/F) \times \prod_{v| p}\left(\# \widetilde{E}(\kappa_v)[p^\infty]\right)^2}{\left(\# E(F)[p^\infty]\right)^2}.
\end{equation}
\end{theorem}

\begin{corollary}
\label{cor to PR-Sch}
Let $E$ be an elliptic curve with good \emph{ordinary} reduction at an odd prime $p$ and assume that 
\begin{enumerate}[(i)]
 \item $\rank_{\Z}E(F)\leq 1$,
 \item the $p$-adic regulator $\mathcal{R}_p(E/F)$ is non-zero, \item $\Sha(E/F)[p^{\infty}]$ is finite.
\end{enumerate}
Then, the truncated Euler characteristic $\chi_t(\Gamma_F, E,p)$ is defined and given by \eqref{BSDformula}.
\end{corollary}

\begin{proof}
By Theorem~\ref{pbsdconj}, the order of vanishing of $f_E(T)$ is $\leq 1$.
Hence, by Lemma~\ref{truncdefined}, the truncated Euler characteristic is defined.
Therefore, by Theorem~\ref{pbsdconj} the truncated $\Gamma_F$-Euler characteristic up to a $p$-adic unit is given by \eqref{BSDformula}.
\end{proof}

\subsection{} 
Following \cite{zerbes2009generalised}, we introduce conditions under which the truncated $G$-Euler characteristic $\chi_t(G, E,p)$ is defined, and give an explicit formula for it.

\begin{Assumption}
\label{assumption: Fin}
Assume that the following conditions are satisfied.
\begin{align*}
(\Finite_{\glob}): &\ H^i(H, E(\Fin)[p^\infty]) \textrm{ is finite for any } i\geq 0, \\
(\Finite_{\local}): & \textrm{ For primes }w|p \textrm{ of } \Fin \textrm{ and } i\geq 0, \textrm{ the group }H^i\left(H_w, \widetilde{E}(\kappa_{\infty, w})[p^\infty]\right) \textrm{ is finite}.
\end{align*}
\end{Assumption}
Here, $\kappa_{\infty,w}$ is the residue field of $F_{\infty,w}$ and $H_w$ is the decomposition group of $w$ in $H$.
By \cite[Proposition~5.6]{zerbes2009generalised}, the local finiteness assumption $(\Finite_{\local})$ is satisfied in our current setting.
The assumption $(\Finite_{\glob})$ is satisfied under the following additional condition.
\begin{proposition}
\label{prop finglob}
If the Lie algebra of $H$ is reductive, then, $(\Finite_{\glob})$ is satisfied.
\end{proposition}

\begin{proof}
The result follows from \cite[Proposition~5.4]{zerbes2009generalised}.
\end{proof}
In Section \ref{S: examples of extensions}, we will show that these assumptions are indeed satisfied in the cases of interest.
Let $G$ be any admissible $p$-adic Lie-extension, \emph{not necessarily pro-$p$}.
Let $\mathfrak{M}$ be the set of primes $v\nmid p$ of $F$ whose inertia group in $G$ is infinite and $L_v(E,s)$ denotes the local $L$-factor at $v$.
By definition, when $E$ has good reduction at $v$,
\[
L_v(E,s) = \left(1 -a_v q_v^{-s} + q_v^{1-2s} \right)^{-1},
\]
where $q_v$ is the order of the residue field  $\kappa_v$ and $a_v = q_v + 1 - \# \widetilde{E}(\kappa_v)$.
Evaluating this local Euler factor at $s=1$ yields
\[
L_v(E,1) = \frac{q_v}{\# \widetilde{E}(\kappa_v)}.
\]
When $E$ has bad reduction, 
\[
L_v(E,s) = 1, \ (1-q_v^{-s})^{-1}, \textrm{ and } \ (1+q_v^{-s})^{-1}
\]
according as $E$ has additive, split multiplicative, and non-split multiplicative reduction, respectively.
When evaluated at $s=1$ the Euler factors become
\[
L_v(E,1) = 1, \ \frac{q_v}{q_v -1}, \textrm{ and } \frac{q_v}{q_v +1}, 
\]
respectively.

The following is an immediate consequence of the above calculations.
\begin{lemma}
\label{criterion for p to divide Lv(E,1)}
Let $v\nmid p$, the prime $p$ divides $\absolute{L_v(E,1)}_p$ in precisely the following situations
\begin{enumerate}
 \item $E$ has good reduction at $v$ and $p|\#\widetilde{E}(\kappa_v)$,
 \item $E$ has split multiplicative reduction at $v$ \emph{and} $q_v\equiv 1\mod{p}$, 
 \item $E$ has non-split multiplicative reduction at $v$ \emph{and} $q_v\equiv -1\mod{p}$.
\end{enumerate}
\end{lemma}

We now introduce an important conjecture in noncommutative Iwasawa theory, which will be assumed throughout our discussion.

\begin{conj}[Conjecture $\fM_H(G)$ {\cite{CKFVS,CS12}}]
\label{MHG Conjecture}
Let $E_{/F}$ be an elliptic curve with good \emph{ordinary} reduction at all primes above $p$.
Denote by $\mathcal{X}(E/\cF_\infty)$ the Pontryagin dual of the Selmer group $\Sel_{p^\infty}(E/\cF_\infty)$ and define the quotient,
\[
\mathcal{X}_f(E/\cF_\infty) := \frac{\mathcal{X}(E/\cF_\infty)}{\mathcal{X}(E/\cF_\infty)(p)}.
\]
Set $H:=\Gal(\cF_\infty/F_{\cyc})$.
Then, $\mathcal{X}_f(E/\cF_\infty)$ is a finitely generated $\Lambda(H)$-module, and hence it makes sense to speak of
$\rank_{\Lambda(H)}\left(\mathcal{X}_f(E/\cF_\infty)\right)$.
\end{conj}

\begin{remark}
When $G=\Gal(\Fin/F)$ is a pro-$p$ extension, $E_{/F}$ is an elliptic curve with good ordinary reduction at all primves above $p$, and $\Sel_{p^\infty}(E/F_{\cyc})$ is a cofinitely generated $\Zp$-module, i.e., $\Sel_{p^\infty}(E/F_{\cyc})$ is $\Lambda(\Gamma_F)$-cotorsion with $\mu_p(E/F_{\cyc})=0$, it follows from \cite[Theorem~2.1]{CS12} (see also \cite[paragraph above Lemma~4.6]{HL}) that $\Sel_{p^\infty}(E/\Fin)$ satisfies $\fM_{H}(G)$.
\end{remark}

Next, we recall the explicit formula for the $G$-Euler characteristic $\chi_t(G, E,p)$ in terms of the $\Gamma_F$-Euler characteristic $\chi_t(\Gamma_F, E,p)$.
\begin{theorem}
\label{EC formula theorem}
Let $E$ be an elliptic curve and $p\geq 5$ a prime at which $E$ has good \emph{ordinary} reduction.
Assume that \emph{all} of the following conditions are satisfied
\begin{enumerate}[(i)]
 \item $\Sha(E/F)[p^\infty]$ is finite,
 \item $\Sel_{p^\infty}(E/\Fin)^\vee$ satisfies $\mathfrak{M}_H(G)$, \emph{and}
 \item both $(\Finite_{\glob})$ and $(\Finite_{\local})$ hold.
\end{enumerate}
Then $\chi_t(G, E,p)$ is defined if and only if $\chi_t(\Gamma_F, E,p)$ is, and are related as follows
\[
\chi_t(G, E,p)=\chi_t(\Gamma_F, E,p)\times \prod_{v\in \mathfrak{M}} \absolute{L_v(E, 1)}_p.
\]
\end{theorem}

\begin{proof}
The result follows from \cite[Theorem~1.1]{zerbes2009generalised}.
In certain special cases, this result has been proved separately.
For example, in the case of false Tate curve extension, this result was first proven in \cite{HV03}.
In the $\GL_2$-setting, this was proved in \cite{CoatesSchneiderSujatha_Links_between}.
\end{proof}

\begin{definition}
Assume that $\cF_\infty$ is an admissible $p$-adic Lie extension and that $M$ is a (compact) finitely generated $\Lambda(G)$-module satisfying $\mathfrak{M}_H(G)$.
It follows that the homology groups $H_i(H, M)$ are all finitely generated torsion $\Lambda(\Gamma)$-modules for all $i\geq 0$ (see \cite[Lemma~3.1]{CKFVS}).
Let $g_{M,i}$ denote its characteristic element.
The \emph{Akashi series} is defined as follows
\[
\Ak_M:=\prod_{i\geq 0} g_{M,i}^{(-1)^i}.
\]
When $M = \Sel_{p^\infty}(E/\Fin)^\vee$, we write
\[
\Ak_{E/\Fin}:=\Ak_{\Sel_{p^\infty}(E/\Fin)^\vee}.
\]
\end{definition}

The next result relates $\Ak_M$ and the truncated $G$-Euler characteristic of $M^\vee$.

\begin{proposition}
\label{prop: leading term}
Suppose that $M$ is a finitely generated $\Lambda(G)$-module that satisfies $\mathfrak{M}_H(G)$ and that the truncated $G$-Euler characteristic $\chi_t(G, M^\vee)$ is defined.
Let $r$ denote the alternating sum 
\[
r:=\sum_{i\geq 0} (-1)^i{\corank}_{\Z_p} \left(H^i(H, M^\vee)^{\Gamma}\right).
\]
Then, the leading term of $\Ak_M$ is $\alpha_M T^{r}$, where 
\[
\absolute{\alpha_M}_p^{-1}=\chi_t(G, M^\vee).
\]
\end{proposition}

\begin{proof}
This is \cite[Proposition~2.10]{zerbes2009generalised}.
\end{proof}

\begin{definition}
A $p$-adic Lie extension $\Fin/F$ is \emph{strongly admissible} if it is admissible \emph{and} for each prime $v|p$ in $F$, the extension $\cF_{\infty,w}$ contains the unramified $\Zp$-extension of $F_v$ for all $w|v$.
\end{definition}

\begin{theorem}
Suppose that $\Fin/F$ is strongly admissible and that $G$ has no element of order $p$.
Let $E_{/\Q}$ be an elliptic curve with good \emph{ordinary} reduction at $p$ and $\Sel_{p^\infty}(E/\Fin)^\vee$ satisfies $\mathfrak{M}_H(G)$.
Then, the following relation holds
\[
\Ak_{E/\Fin}\equiv \Char_{\Lambda(\Gamma)}\left(\Sel_{p^\infty}(E/F_{\cyc})^\vee\right)\times \prod_{v\in S'}\Char_{\Lambda(\Gamma)}\left(J_v(F_{\cyc})^\vee\right)\mod\Lambda(\Gamma)^\times,
\]
where $S'$ is the set of primes of $F$ not dividing $p$ such that the inertia group of $v$ in $G$ is infinite and $J_v(F_{\cyc})$ is defined to be $\bigoplus_{w|p}H^1(F_{\cyc,w},E[p^\infty])$.
\end{theorem}

\begin{proof}
See \cite[Theorem~1.3]{zerbes11}.
Our assumption that $E$ has good \emph{ordinary} reduction at $p$ means that the factor $T^r$ in \emph{loc. cit.} is trivial.
\end{proof}

Next, we consider the special case when $G$ is pro-$p$ and $\Ak_M$ is a unit in $\Lambda(\Gamma)$.

\begin{proposition}
\label{prop: pseudonull}
Let $G$ be a compact pro-$p$, $p$-adic Lie group and $H$ be a closed normal subgroup of $G$ with $G/H\simeq \Z_p$.
Let $M$ be a finitely generated $\Lambda(G)$-module which lies in $\mathfrak{M}_H(G)$.
If $\Ak_M$ is a unit in $\Lambda(\Gamma)$, then $M$ is a pseudonull $\Lambda(G)$-module.
Further, if $M$ contains no non-trivial pseudonull submodules, then $\Ak_M$ is a unit if and only if $M=0$.
\end{proposition}

\begin{proof}
See \cite[Proposition~5.9]{lim2015remark}.
\end{proof}

For an elliptic curve defined over $\Q$, the next result gives a criterion for the dual Selmer group $\Sel_{p^\infty}(E/\Fin)^{\vee}$ to be pseudonull as a $\Lambda(G)$-module.

\begin{proposition}
\label{proposition: akasi series to pseudonull}
Let $E$ be an elliptic curve defined over $\Q$ and $F$ be a number field.
Let $\Fin/F$ be a strongly admissible pro-$p$, $p$-adic Lie extension of $F$.
Assume that \emph{all} of the following conditions are satisfied
\begin{enumerate}[(i)]
 \item $\rank_{\Z}E(F)=0$,
 \item $\Sha(E/F)[p^{\infty}]$ is finite,
 \item $\Sel_{p^\infty}(E/\Fin)^{\vee}$ satisfies $\mathfrak{M}_H(G)$, \emph{and} 
 \item the Lie algebra of $H$ is reductive.
\end{enumerate}
Then the following are equivalent
 \begin{enumerate}
 \item $\chi(G, E,p)=1$.
 \item $\Ak_{E/\Fin}$ is a unit in $\Lambda(\Gamma_F)$ and $\Sel_{p^\infty}(E/\Fin)^{\vee}$ is a pseudonull $\Lambda(G)$-module.
 \end{enumerate}
\end{proposition}

\begin{proof}
We first prove (1)$\Rightarrow$(2). By Proposition~\ref{prop finglob}, the assumption $(\Finite_{\glob})$ is satisfied.
Since $\rank_{\Z}E(F)=0$, the normalized regulator $\mathcal{R}_p(E/F)=1$.
Thus, the conditions of Theorem~\ref{pbsdconj} are satisfied, whereby ${\ord}_{T=0}f_E(T)=0$.
It follows that the $\Gamma_F$-Euler characteristic $\chi(\Gamma_F, E,p)$ is defined.
Moreover, according to Theorem~\ref{EC formula theorem}, the $G$-Euler characteristic is defined.
Since $\chi(G, E,p)=1$, it follows from Proposition~\ref{prop: leading term} that the leading term of $\Ak_{E/\Fin}$ is a unit.
Therefore, in order to show that $\Ak_{E/\Fin}$ is a unit in $\Lambda(\Gamma_F)$, it suffices to show that ${\ord}_{T=0}\Ak_{E/\Fin}=0$.
It follows from \cite[Remark~1.4]{zerbes11} that 
\[
{\ord}_{T=0}\Ak_{E/\Fin}={\ord}_{T=0}f_E(T)=0.
\]
Hence, the Akashi series $\Ak_{E/\Fin}$ is a unit in $\Lambda(\Gamma_F)$.
By Proposition~\ref{prop: pseudonull}, the dual Selmer group $\Sel_{p^\infty}(E/\Fin)^\vee$ is a pseudonull $\Lambda(G)$-module.

Finally, the implication (2)$\Rightarrow$(1) follows from Proposition~\ref{prop: leading term}.
\end{proof}

\section{Growth of Mordell--Weil ranks}\label{S: MW}
Let $\Fin$ be an admissible $p$-adic Lie extension of dimension $d\ge2$, and assume that $G=\Gal(\Fin/F)$ is a uniform pro-$p$ group.
For $n\ge0$, we write $G_n=G^{p^n}$, $H_n=H^{p^n}$, and define
$F_n=F^{G_n}$.
In particular, $F_n/F$ is a finite extension of degree $p^{dn}$.

Throughout, we assume that $\Sel_{p^\infty}(E/F_{\cyc})^\vee$ is $\Lambda(\Gamma_F)$-torsion and that $\Sel_{p^\infty}(E/\Fin)^\vee$ satisfies $\fM_H(G)$.
As before, define
\[
\cX(E/\Fin)_f:=\Sel_{p^\infty}(E/\Fin)^\vee/\Sel_{p^\infty}(E/\Fin)^\vee(p).
\]
We are interested in the $\Lambda(H)$-rank of $\cX(E/\Fin)_f$.
As discussed in the introduction (see also \cite{howson2002euler, CoatesSchneiderSujatha_Links_between}), for a finitely generated $\Lambda(G)$-module $M$ satisfying $\fM_H(G)$, the $\Lambda(H)$-rank of $M_f$ can be regarded as the higher-dimensional analogue of the $\lambda$-invariant.
In \cite{mazur72}, Mazur proved that the $\lambda$-invariant of $\Sel_{p^\infty}(E/F_{\cyc})^\vee$ gives a bound on the rank of the Mordell--Weil groups of $E$ over sub-extensions inside $F_{\cyc}$.
More recently, P.~C.~Hung and M.~F.~Lim proved the following higher-dimensional generalization of Mazur's result.

\begin{theorem}
\label{Thm: Hung-Lim Theorem}
Suppose that $H_i(H_n,\Sel_{p^\infty}(E/\Fin)^\vee)$ is finite for every $i\ge1$ and $n\ge0$.
Then,
\[
\rank_\Z E(F_n)\le \rank_{\Lambda(H)}\cX(E/\Fin)_f\cdot p^{(d-1)n}+d\corank_{\Zp}E(\Fin)(p).
\]
\end{theorem}

\begin{proof}
See {\cite[Theorem~3.2]{HL}}.
\end{proof}

The hypothesis on the finiteness of $H_i(H_n,\Sel_{p^\infty}(E/\Fin)^\vee)$ is known to hold in the settings we will study in subsequent sections.
As detailed in the remark right after \cite[Theorem~3.2]{HL}, when $d=2$ or $3$, this has been proved in \cite{lim2015remark} and \cite{DL}.
More generally, it holds under the hypotheses (Fin$_\mathrm{glob}$) and (Fin$_\mathrm{loc}$) by \cite[Lemma~4.3]{zerbes2009generalised}.

\begin{remark}
\label{rmk: finiteness of p-primary torsion points over certain extensions}
For elliptic curve without CM, we know that $E(\Fin)[p^\infty]$ is finite if $E[p^\infty]$ is \emph{not} rational over $\Fin$, see \cite[Lemma~6.2]{LM14}.
\end{remark}

The following result gives an explicit relation between the $\Lambda(H)$-rank of $\cX(E/\Fin)_f$ and the cyclotomic $\lambda$-invariant, $\lambda_p(E/F_{\cyc})$.

\begin{proposition}
\label{prop:Lambda-H-rank}
Given a prime $v$ of $F$, define 
\[
Z_v=\begin{cases}
 E(\overline{F_v})(p)&\text{if }v\nmid p,\\
\widetilde{E}(\overline{\kappa_v})(p)&\text{otherwise.}
\end{cases}
\]
Denote by $v$ the rational prime below a prime $w$ of $F_{\cyc}$.
Then, we have that
\begin{equation}
\label{LambdaH rank formula}
\rank_{\Lambda(H)}\cX(E/\Fin)_f=\lambda_p(E/F_\cyc)+\sum_{\substack{w\in S(F_\cyc), \\\dim H_w\ge1}}\corank_{\Zp} H^0(F_{\cyc,w}, Z_v).
\end{equation}
Here, $S$ is the set of rational primes $\ell$ consisting of $p$, the set of primes at which $E$ has bad reduction, and the primes that are ramified in $\Fin$; $S(F_\cyc)$ consists of all primes $w$ of $F_\cyc$ that lie above the set $S$.
\end{proposition}

\begin{proof}
See \cite[Proposition~4.1]{HL}.
\end{proof}

\begin{remark}
\label{implication pseudonullity}
If $M$ is a finitely generated $\Lambda(H)$-module, by \cite[p.~208]{CoatesSchneiderSujatha_Links_between} we know that
\begin{equation}
\label{eqn: CSS03 p. 208}
\rank_{\Lambda(H)}M=0 \textrm{ if and only if } M \textrm{ is } \Lambda(G)\textrm{-pseudonull}.
\end{equation}
If $\mathcal{X}(E/\Fin)$ lies in $\fM_H(G)$, then the above statement holds for $M = \mathcal{X}(E/\Fin)_f$.
On the other hand, there are a large class of $p$-adic Lie extensions such that $M = \mathcal{X}(E/\Fin)$ is a finitely generated $\Lambda(H)$-module.
In particular, for $p$-adic Lie extensions of interest (see \S\ref{S: examples of extensions}), it is known that $\mathcal{X}(E/\Fin)$ is a finitely generated $\Lambda(H)$-module if and only if $\mathcal{X}(E/F_{\cyc})$ is a finitely generated $\Z_p$-module, see \cite[Theorem~3.1(i)]{HV03} and \cite{coateshowson}.
In such cases, \eqref{eqn: CSS03 p. 208} holds for both $M=\mathcal{X}(E/\Fin)$ and $\mathcal{X}(E/\Fin)_f$.
From our earlier discussion in Proposition~\ref{proposition: akasi series to pseudonull}, we know that
\[
\chi(G, E,p)= 1 \Rightarrow \mathcal{X}(E/\Fin) \textrm{ is pseudonull}.
\]
The following implication is straightforward,
\[
\rank_{\Lambda(H)}\mathcal{X}(E/\Fin)=0 \Rightarrow \rank_{\Lambda(H)}\mathcal{X}(E/\Fin)_f=0.
\]
This shows that if the $G$-Euler characteristic of the Selmer group is a $p$-adic unit, then $\mathcal{X}(E/F)_f$ is $\Lambda(H)$-torsion.
However, as pointed out to us by the referee, the converse is \emph{not} true in general.
Consider the elliptic curve with Cremona label \href{https://www.lmfdb.org/EllipticCurve/Q/11a1/}{11a1} or \href{https://www.lmfdb.org/EllipticCurve/Q/11a2/}{11a2} and $p=5$.
It is known that $\mathcal{X}(E/F_{\cyc})$ has \emph{positive} $\mu$-invariant but trivial $\lambda$-invariant, see \cite[Chapter 5]{CoatesSujatha_book}.
Let $\mathcal{F}_\infty = F_{\cyc}(5^{5^{-\infty}})$.
An application of \cite[Lemma~5.6]{CKFVS} shows that $\fM_H(G)$-conjecture is satisfied in this case.
Therefore, by \cite[Theorem~3.1]{lim2015remark} we know that
\[
\mu_{G}\left( \mathcal{X}(E/\mathcal{F}_{\infty})\right) = \mu_{\Gamma}\left( \mathcal{X}(E/F_{\cyc})\right)>0.
\]
By the previous proposition, we also know that 
\[
\rank_{\Lambda(H)}\left( \mathcal{X}(E/\mathcal{F}_{\infty})\right)_f =0.
\]
However, the $G$-Euler characteristic is $p^{\mu_{G}\left( \mathcal{X}(E/\mathcal{F}_{\infty})\right)}$, see \cite{AW06}.
\end{remark}

From here on, we set $Z_v(F_{\cyc,w})$ to simply denote $H^0(F_{\cyc,w},Z_v)$.
We may describe the coranks of the local terms $Z_v(F_{\cyc,w})$ explicitly as follows.

\begin{lemma}
\label{Zp corank cases}
Let $S(F_\cyc)$ be the set of primes described in Proposition~\ref{prop:Lambda-H-rank}.
Then,
\[
\corank_{\Zp} Z_v(F_{\cyc,w})=\begin{cases}
2&\text{$w\nmid p$, $E$ has good reduction at $w$, $E(F_v)[p^\infty]\ne 0$,}\\
1&\text{$w\nmid p$, $E$ has split multiplicative reduction at $w$,} \\
0&\text{otherwise.}
\end{cases}
\]
\end{lemma}

\begin{proof}
When $w\nmid p$, then $Z_v(F_{\cyc,w})$ is described in \cite[Proposition~5.1]{HM99}.
When $w|p$, $Z_v(F_{\cyc,w})$ is in fact always finite.
See the discussion in \cite[\S5]{LimKida}.
\end{proof}

\begin{remark}
The case $F=\Q(\mu_p)$ and $\Fin=F(\mu_{p^\infty},\sqrt[p^\infty]{m})$ is discussed in \cite[Theorem~3.1]{HV03} under the hypothesis that $\Sel_{p^\infty}(E/F_{\cyc})$ is a cofinitely generated $\Zp$-module.
\end{remark}

\section{$p$-adic Lie Extensions of Interest}
\label{S: examples of extensions}
In this section, we discuss examples of $p$-adic Lie-extensions for which the results of the previous sections apply.
Recall that these assumptions are:
\begin{itemize}
 \item Assumption \ref{assumption: torsion}, which asserts that $\Sel_{p^\infty}(E/\Fin)^\vee$ is torsion as a $\Lambda(G)$-module.
 \item The global (resp. local) finiteness assumption $(\Finite_{\glob})$ (resp. $(\Finite_{\local})$) from Assumption \ref{assumption: Fin}.
\end{itemize}

\subsection{}
Let $F=\Q(\sqrt{-d})$ be an imaginary quadratic field and $\Fin$ be the compositum of all $\Zp$-extensions over $F$.
Note that $G:=\Gal(\Fin/F)$ is isomorphic to $\Z_p^2$.
In this setting, Remark~\ref{torsion remark} guarantees that $\Sel_{p^\infty}(E/\Fin)$ is a cotorsion $\Lambda(G)$-module.
Recall that $\fM$ is the set of primes $v\nmid p$ of $F$ whose inertia group in $G$ is infinite.
It is a simple exercise to show that $\Fin$ is unramified at all primes $v\nmid p$.
Hence, the set $\fM$ is empty.
Since $H = \Gal(\Fin/F_{\cyc})\simeq \Zp$, the global hypothesis ($\Finite_{\glob}$) holds by Proposition~\ref{prop finglob}.
For the local hypothesis, if $\widetilde{E}(\cF_{\infty,w})[p^\infty]$ is finite then there is nothing to prove; else, it follows from \cite[Proposition~5.6]{zerbes2009generalised}.
Thus, Theorem~\ref{EC formula theorem} simplifies to give
\begin{equation}
\label{ec are equal}
\chi_t(G, E,p)=\chi_t(\Gamma_F, E,p).
\end{equation}

We review a result of R.~Greenberg \cite[Proposition~4.1.1]{greenberg16} regarding sufficient conditions for $\Sel_{p^\infty}(E/\Fin)^\vee$ to admit no non-trivial pseudonull submodule.
Let $\cT=T_p(E)\otimes \Lambda(G)^\iota$, where $\iota$ is the involution on $\Lambda(G)$ sending a group-like element to its inverse.
We write $\cD=\cT\otimes_{\Lambda(G)} \Lambda(G)^\vee$.

In the notation of \cite[\S2.1]{greenberg16}, the condition \textbf{RFX($\cD$)}, which asserts that $\cT$ is a reflexive $\Lambda(G)$-module holds since it is free over $\Lambda(G)$.
The condition \textbf{LEO($\cD$)} says that 
\[
\ker\left(H^2(F_S/F,\cD)\rightarrow \prod_{v\in \Sigma}H^2(F_v,\cD)\right)
\]
is a cotorsion $\Lambda(G)$-module.
Recall from \cite[Theorem~3]{greenberg06} that there is an isomorphism of $\Lambda(G)$-modules $H^2(F_S/F,\cD)\cong H^2(F_S/\Fin,E[p^\infty])$.

Recall that $\Sel_{p^\infty}(E/\Fin)^\vee$ is $\Lambda(G)$-torsion by Remark~\ref{torsion remark}.
Further, 
\[
\rank_{\Lambda(G)}\bigoplus_{\ell\in S}J_\ell(E/\Fin)^\vee=\rank_{\Lambda(G)}H^1(F_S/\Fin,E[p^\infty])-\rank_{\Lambda(G)}H^2(F_S/\Fin,E[p^\infty])=2
\]
by \cite[Theorems~3.2 and 4.1]{OV03}.
A standard argument with Poitou--Tate exact sequence then tells us that $H^2(F_S/\Fin,E[p^\infty])$ is cotorsion over $\Lambda(G)$, whereas $H^1(F_S/\Fin,E[p^\infty])$ is of corank two. In particular, \textbf{LEO($\cD$)} holds.

The condition \textbf{CRK}$(\cD,\mathcal{L})$, which says that 
\[
\corank_\Lambda H^1(F_S/\Fin,E[p^\infty])=\corank_\Lambda \Sel_{p^\infty}(E/\Fin)+\corank_\Lambda\bigoplus_{\ell\in S}J_\ell(E/\Fin)
\]
also holds since both sides equal to $2$ in our current setting.

We now consider the conditions \textbf{LOC}$_v^{(i)}(\cD)$, $i=1,2$.
Let $\cT^*=\Hom(\cD,\mu_{p^\infty})$.
The conditions say that for $v\in S$, we have $(\cT^*)^{G_{F_v}}=0$ and $\cT^*/(\cT^*)^{G_{F_v}}$ is a reflexive $\Lambda(G)$-module, respectively.
Since $p\ne2$, we have $(\cT^*)^{G_{F_v}}=0$ when $v$ is an archimedean prime.
Furthermore, if $v$ is a non-archimedean prime, it does not split completely in $\Fin$.
By \cite[Lemma~5.2.2]{greenberg10}, $(\cT^*)^{G_{F_v}}=0$.
As $\cT^*$ is a free $\Lambda(G)$-module, the conditions \textbf{LOC}$_v^{(1)}(\cD)$ and \textbf{LOC}$_v^{(2)}(\cD)$ both hold for all $v\in S$.

We can now state the following result due to Greenberg.
\begin{proposition}\label{prop:Greenberg}
If $E(F)$ has no element of order $p$, then $\Sel_{p^\infty}(E/\Fin)^\vee$ admits no non-trivial pseudonull submodule.
\end{proposition}

\begin{proof}
We have verified the hypotheses \textbf{RFX($\cD$)}, \textbf{LEO($\cD$)}, \textbf{CRK}$(\cD,\mathcal{L})$, \textbf{LOC}$_v^{(1)}(\cD)$, and \textbf{LOC}$_v^{(2)}(\cD)$ hold for all $v\in S$.
Next, the condition $\cD[\mathfrak{m}]$ admits no quotient isomorphic to $\mu_p$ for the action of $G_F$ (assumption (b) in \emph{loc. cit.}) is equivalent to $E(F)[p]=0$ via the Weil pairing (see the last paragraph on p.~248 of \emph{op. cit.}).
Therefore, the result is a direct consequence of \cite[Proposition~4.1.1]{greenberg16}.
\end{proof}

\subsection{}
We now move on to noncommutative $p$-adic Lie extensions.
A prototypical example of a noncommutative $p$-adic Lie extension is the \emph{false Tate curve extension}.
This extension is obtained by adjoining the $p$-power roots of a fixed ($p$-power free) integer $m>1$ to the cyclotomic $\Zp$-extension of $F=\Q(\mu_p)$.
More precisely,
\[
\Fin = \Q\left( \mu_{p^\infty}, \ m^{\frac{1}{p^n}} : n = 1, 2, \ldots \right).
\]
Recall that $G:=\Gal(\Fin/F)$ and $H:=\Gal(\Fin/F_{\cyc})$.
There is a section to the quotient map $G\rightarrow \Gamma_F$, thus $G$ is a semi-direct product $H\rtimes \Gamma_F$.
Fix non-canonical isomorphisms $H\simeq \Z_p$ and $\Gamma_F\simeq \Z_p$.
By Kummer theory, the action of $\Gamma_F$ on $H$ is via the cyclotomic character.
It follows from \cite[Lemma~7.3]{HV03} that
\[
\rank_{\Lambda(G)} \Sel_{p^\infty}(E/\cF_{\infty})^{\vee} \leq \rank_{\Lambda(\Gamma_F)} \Sel_{p^\infty}(E/F_{\cyc})^{\vee}.
\]
Since Kato's result guarantees that $\Sel_{p^\infty}(E/F_{\cyc})$ is $\Lambda(\Gamma_F)$-cotorsion,  $\Sel_{p^\infty}(E/\cF_{\infty})$ is thus  $\Lambda(G)$-cotorsion.

To discuss the precise formula for the $G$-Euler characteristic of $\Sel_{p^\infty}(E/\Fin)$, we need to introduce two sets of primes in $F$.
Define
\begin{small}
\begin{equation}
\label{definition of P1 P2}
\begin{split}
\mathcal{P}_1(E, \Fin) := & \{v \nmid p: v |m \textrm{ and } E \textrm{ has split multiplicative reduction at } v.
\},\\
\mathcal{P}_2(E, \Fin) := &\{v \nmid p: v |m, E \textrm{ has good \emph{ordinary} reduction at } v, \textrm{ and } E(F_v)[p^\infty]\neq 0\}.
\end{split}
\end{equation}
\end{small}
By \cite[\S 4.1]{HV03}, we know that the set $\fM$ appearing in Theorem~\ref{EC formula theorem} is given by \begin{equation}
\label{set M for FTC}
\fM = \fM(E, \Fin) = \mathcal{P}_1(E, \Fin) \cup \mathcal{P}_2(E, \Fin).
\end{equation}
Since $H\simeq \Z_p$, the Lie algebra of $H$ is reductive, and ($\Finite_{\glob}$) is true by Proposition~\ref{prop finglob}.

For the false Tate curve extension, it is known that the $\Sel_{p^\infty}(E/\Fin)^\vee$ has no non-zero pseudonull submodules.
We record this result below.
\begin{theorem}
\label{FT no non zero pseudonull submodules}
Let $p$ be an odd prime and $E_{/\Q}$ be an elliptic curve with good \emph{ordinary} reduction at $p$.
Then, $\Sel_{p^\infty}(E/\Fin)^\vee$ has no non-zero pseudonull submodules.
\end{theorem}

\begin{proof}
See \cite[Theorem~2.6 and proof of Theorem~2.8]{HV03}.
\end{proof}

\subsection{}
Consider a pair of elliptic curves $(E,A)$ both defined over $\Q$.
Assume that $A$ does not have CM.
Let $F:=\Q(A[p])$ and consider the pro-$p$ $p$-adic Lie extension ${\Fin}_{,A}:=\Q(A[p^\infty])$  with corresponding Galois group $G_A$.
For the remainder of this section we assume that $p\geq 5$ to ensure that $G_A$ has no $p$-torsion.
Then $\Fin/F$ is an admissible pro-$p$, $p$-adic Lie extension.
We shall study  the Selmer group $\Sel_{p^\infty}(E/\Fin)$ as a $\Lambda(G_A)$-module.

We can describe the set $\fM$ explicitly in this setting.
By \cite[Lemma~2.8(i)]{Coates_Fragments}, the decomposition group at $v\nmid p$ has dimension 1 (resp. 2) if $v$ is a prime of potentially good reduction (resp. potentially multiplicative reduction) for $A$.
Hence, the primes of potentially multiplicative reduction are ramified in the trivializing extension and their inertia group is infinite.
Therefore, the  set $\mathfrak{M}$ consists of precisely the primes at which $A$ has potentially multiplicative reduction (or those primes for which the $j$-invariant of $A$ has negative valuation), see \cite[Theorem~3.1]{CoatesSchneiderSujatha_Links_between}.
Finally, we remark that since $H=H_A = \Gal({\Fin}_{,A}/F_{\cyc})$ is semi-simple, it follows from \cite[Lemma~5.4]{zerbes2009generalised} that ($\Finite_{\glob}$) holds.
However, we are not aware of any unconditional results on Assumption~\ref{assumption: torsion} in this setting.

We end this section with the following remark which allows us to relate the Euler characteristic formula to the characteristic element, in our $p$-adic Lie extensions of interest.

\begin{remark}
Let $F$ be a number field which is \emph{not} totally real, and $\Fin/F$ be a pro-$p$ strongly admissible extension.
For our $p$-adic Lie extensions of interest, both conditions are satisfied.
If $E_{/\Q}$ is an elliptic curve satisfying the conditions of Proposition~\ref{proposition: akasi series to pseudonull} and $\Fin/F$ is such that $\Sel_{p^\infty}(E/\Fin)^\vee$ admits no non-zero pseudonull submodules, then it is possible to draw conclusions about the characteristic element required in the formulation of the main conjectures of
noncommutative Iwasawa theory.
In particular, it follows from \cite[Theorem~5.11 and Proposition~6.2]{lim2015remark} that the Euler characteristic is a $p$-adic unit if and only if the characteristic element is a unit as well.
\end{remark}

\section{Tamagawa number calculations}
\label{S: Tamagawa calculations}

In this section, we will perform routine calculations on Tamagawa numbers upon base-change.
These results will be required in subsequent discussions.

Throughout, $p\geq 5$ is fixed.
For a rational prime $\ell\neq p$, set
\begin{align*}
\tau_{\ell}= \tau_{\ell}^{(\Q)}:= c_{\ell}^{(p)}(E/\Q) \quad \textrm{and} \quad
\tau_{\ell}^{(F)}:=\prod_{v| \ell} c_v^{(p)}(E/F).
\end{align*}
In this section, the goal is to compute how the Tamagawa numbers behave when the base field is changed from $\Q$ to $F$.
We remind the reader that by \cite[p.~448]{silverman2009}, $c_\ell$ is divisible by $p\geq 5$ precisely when the Kodaira type of $E$ at $\ell$ is $\textrm{I}_n$ with $p|n$.

\subsection{}
Let $E_{/\Q}$ be an elliptic curve with good \emph{ordinary} reduction at fixed $p\geq 5$ and conductor $N_E$.
Consider the imaginary quadratic field $F=\Q(\sqrt{-d})$ as $d$ varies over all positive square-free numbers coprime to $p N_E$.
We write $\tau_{\ell}^{(d)}$ in place of $\tau_{\ell}^{(F)}$.
The main result in this direction is the following.

\begin{theorem}
\label{Tamagawa result from HKR}
With notation as above, let $\ell$ be a prime such that $\ell| N_E$ and $\tau_{\ell}=1$.
Then, the following assertions hold.
\begin{enumerate}
\item If $\ell\neq 2$ then $\tau_{\ell}^{(d)}=1$ for $\ell\nmid d$.
\item If $\ell=2$ then $\tau_{\ell}^{(d)}= 1$ for all $d$.
\end{enumerate}
\end{theorem}

\begin{proof}
The proof follows from that of \cite[Theorem~7.2]{HKR}, where the assertion is made when $d$ is a prime number.
The result holds even in this more general setting.
\end{proof}

\subsection{}
Consider the case when $F= \Q(\mu_p)$.
By the assumption on $p$, it follows that the degree $[F:\Q] =p-1 \geq 4$.
The only prime that ramifies in $F$ is $p$, and this prime is totally ramified.
Fix a prime $\ell\neq p$, and denote by $T_{\base}$ the Kodaira symbol of $E$ over $\Q_\ell$.
In \cite{kida03}, M.~Kida studied the variation of the reduction type under a finite extension $F_v/\Q_{\ell}$.
For $\ell\geq 5$, the Kodaira symbol of the base-change $E_{/F_v}$ is determined by $T_{\base}$ and the ramification index of $F_v/\Q_{\ell}$, as we now explain.
\begin{enumerate}[(a)]
\item If $\ell$ is completely split in $F$, then $F_v = \Q_{\ell}$.
It is immediate that $\tau_{\ell}^{(F)}=\tau_{\ell}$.
\item Otherwise, let $v| \ell$ in $F$.
Since $\ell$ is unramified in $F$, the ramification index of $v$ is $e=1$.
Varying over all extensions $F_v$ with $v|\ell$, we read off the ``new'' Kodaira type $T_{\new}$ from \cite[Table 1, pp.~556-557]{kida03}.
According to this table, base changing from $\Q_\ell$ to $F_v$, the Tamagawa number $c_{v}$ becomes divisible by $p$ precisely when
\[
T_{\base} = \textrm{I}_n \textrm{ such that } p|n.
\]
\end{enumerate}

\section{Results for  fixed   $E_{/\Q}$ and  $p$ as $\Fin$ varies}
\label{S: Vary extension}

We remind the readers that $E$ is defined over $\Q$ with good \emph{ordinary} reduction at $p$.
This will be the standing assumption throughout this section.
We are interested in answering the following related questions
\begin{enumerate}[(a)]
 \item Is the truncated Euler characteristic $\chi_t(G, E,p)$ a $p$-adic unit?
 \item What can be said about the Akashi series $\Ak_{E/\Fin}$?
 \item Suppose $\rank_{\Z}E(\Q)=0$.
 When is the Selmer group $\Sel_{p^\infty}(E/\Fin)^\vee$ pseudonull as a $\Lambda(G)$-module? When is it equal to $0$?
 \item What can be said about the growth of $\op{rank}E(F_n)$ as $n\rightarrow \infty$?
\end{enumerate}
In this section, we make progress on such questions on average, which is to say that we fix the pair $(E,p)$ and let $\Fin$ vary over certain \textit{families of extensions}.
As $\Fin$ varies over the family, the number field $F$ does not need  to be fixed.
Though there is no formal definition of a family of extensions we are aware of, the families we study are quite natural to consider.

In \S\ref{Euler characteristic vary F Zp2}, we vary over all imaginary quadratic fields $F=\Q(\sqrt{-d})$ and consider the (unique) $\Zp^2$-extension $\Fin^{(d)}/F$.
The main result in this section is Theorem~\ref{theorem varying imaginary quadratic} which provides equivalent criteria for the ($p$-adic) triviality of the $G$-Euler characteristic.
This allows the formulation of Conjecture~\ref{conjecture vary Fin zp2}.
A result on the variation of the Mordell--Weil rank growth is recorded in Corollary~\ref{MW cor vary F Zp2}.
In \S\ref{sec:8.2}, we fix the field $F=\Q(\mu_p)$ and consider a family of false Tate curve extensions.
In Theorem~\ref{thm: Hep-Gep density}, we study the question pertaining to how often the $G$-Euler characteristic is a $p$-adic unit.
We study the variation of the $\Lambda(H)$-corank of the Selmer group and the growth of the Mordell--Weil ranks in \S\ref{rank questions for false Tate curve extensions}.
In \S\ref{section 8-3}, we vary over non-CM elliptic curves $A_{/\Q}$ to obtain a family of $\GL_2(\Zp)$-extensions.
Under some reasonable hypotheses, we show that the $\mathcal{G}_A$-Euler characteristic is a $p$-adic unit for infinitely many (but density 0) elliptic curves $A_{/\Q}$.
In \S\ref{GA Ec Zp2 case}, we make some observations on the $\Lambda(H)$-coranks of Selmer groups.

\subsection{}
\label{Euler characteristic vary F Zp2}
Let $d$ range over all positive square-free numbers coprime to the conductor $N_E$.
In this section, we write $F=F^{(d)}$ to denote the imaginary quadratic field $\Q(\sqrt{-d})$.
Let $\Fin^{(d)}$ be the $\Z_p^2$-extension obtained from taking the composite of all $\Z_p$-extensions of $F^{(d)}$.
It is easy to see that $\Fin^{(d)}$ is a strongly admissible extension of $F^{(d)}$.
Set $G^{(d)}:=\Gal(\Fin^{(d)}/F^{(d)})$ and $\Gamma_F = \Gamma^{(d)}:=\Gal(F^{(d)}_{\cyc}/F^{(d)})$.
We want to understand how often the $G^{(d)}$-Euler characteristic $\chi_t(G^{(d)}, E,p)$ is equal to $1$ as $d$ varies over all positive square-free numbers.
Throughout, we shall impose the following standard hypotheses.
\begin{itemize}
 \item $\Sha(E/F^{(d)})[p^\infty]$ is finite, 
 \item The normalized $p$-adic regulator $\mathcal{R}_p(E/F^{(d)})$ is non-zero.
\end{itemize}
The second hypothesis is satisfied whenever $\rank_{\Z}E(F^{(d)})=0$, in which case the normalized $p$-adic regulator is equal to $1$ by definition.
The following result shall motivate the next assumption on $E$.

\begin{lemma}
\label{Gamma Q Gamma F basic lemma}
Let $F$ be a number field with $[F:\Q]$ prime to $p$.
Assume that $\chi_t(\Gamma_{\Q}, E,p)$ and $\chi_t(\Gamma_F, E,p)$ are both defined.
Then, $\chi_t(\Gamma_{\Q}, E,p)$ divides $\chi_t(\Gamma_F, E,p)$.
\end{lemma}

\begin{proof}
Under the assumption that $\chi_t(\Gamma_{\Q}, E,p)$ and $\chi_t(\Gamma_F, E,p)$ are defined, it follows that $\Sel_{p^\infty}(E/\Q_{\cyc})^{\vee}$ and $\Sel_{p^\infty}(E/F_{\cyc})^{\vee}$ are both torsion over their respective Iwasawa algebra.
Let $f(E/\Q_{\cyc})$ (resp. $f(E/F_{\cyc})$) be the characteristic element of $\Sel_{p^\infty}(E/\Q_{\cyc})^{\vee}$ (resp. $\Sel_{p^\infty}(E/F_{\cyc})^{\vee}$) as a $\Lambda(\Gamma_\Q)$-module after identifying $\Gamma_\Q$ with $\Gamma_F$.
Since $[F:\Q]$ is coprime to $p$, it follows that the map induced by restriction 
\[
\Sel_{p^\infty}(E/\Q_{\cyc})\rightarrow \Sel_{p^\infty}(E/F_{\cyc})
\]
is injective, and hence, $f(E/\Q_{\cyc})$ divides $f(E/F_{\cyc})$.
Therefore, the leading coefficient of $f(E/\Q_{\cyc})$ divides that of $f(E/F_{\cyc})$.
The result now follows from Lemma~\ref{lemmazerbes}.
\end{proof}

If $p$ divides $\chi_t(\Gamma_{\Q}, E,p)$, then $p$ must divide $\chi_t(\Gamma^{(d)}, E,p)$ for all $d$.
By \eqref{ec are equal}, we know that $\chi_t(G^{(d)}, E,p)$ is equal to $\chi_t(\Gamma^{(d)}, E,p)$.
Thus for all $d$, the $\Gamma^{(d)}$-Euler characteristic $\chi_t(\Gamma^{(d)}, E,p)$ is divisible by $p$.
On the other hand, it is indeed possible for $p$ to divide $\chi_t(G^{(d)}, E,p)$ when $\chi_t(\Gamma, E,p)=1$.
We study the case when $\rank_{\Z}E(\Q)=0$ and we assume that the following equivalent conditions hold for $E$
\begin{enumerate}[(a)]
 \item $\chi(\Gamma_{\Q}, E,p)=1$,
 \item $\mu_p(E/\Q_{\cyc})=0$ and $\lambda_p(E/\Q_{\cyc})=0$,
 \item $\Sel_{p^\infty}(E/\Q_\cyc)=0$.
\end{enumerate}
When the residual representation on $E[p]$ is irreducible, it has been conjectured by Greenberg that $\mu_p(E/\Q_{\cyc})=0$, see \cite[Conjecture~1.11]{Greenberg}.
We have the following result.

\begin{theorem}
\label{theorem varying imaginary quadratic}
Let $E$ be an elliptic curve defined over $\Q$ with conductor $N_E$ and $p\geq 5$ a prime for which the following hypotheses are satisfied
\begin{enumerate}[(i)]
 \item $E$ has good \emph{ordinary} reduction at $p$,
 \item $\rank_{\Z}E(\Q)=0$ and $E(\Q)[p^\infty]=0$,
 \item $\chi(\Gamma_\Q, E,p)=1$,
 \item $E$ has good reduction at $\ell=2,3$.
\end{enumerate}
Let $F^{(d)}:=\Q(\sqrt{-d})$ be an imaginary quadratic field, for which the following conditions are satisfied
\begin{enumerate}[(i)]
\item $\rank_{\Z}E(F^{(d)})=0$,
 \item $\Sel_{p^\infty}(E/\Fin^{(d)})^{\vee}$ satisfies $\mathfrak{M}_H(G)$,
 \item $a_p(E)\not\equiv -1\mod{p}$ if $p$ is inert in $F^{(d)}$,
 \item $\gcd(N_E,d)=1$.
\end{enumerate}
Then, the following are equivalent
\begin{enumerate}
 \item \label{72:1} $\Sha(E/F^{(d)})[p^\infty]=0$.
 \item \label{72:2}
 $E(F^{(d)})[p^\infty]=0$ and $\chi_t(G^{(d)}, E,p)=1$.
 \item \label{72:3} $E(F^{(d)})[p^\infty]=0$ and the Akashi-series $\Ak_{E/\Fin^{(d)}}$ is a unit in $\Lambda(\Gamma)$ and $\Sel_{p^\infty}(E/\Fin^{(d)})$ is pseudonull.
 \item \label{72:4} $E(F^{(d)})[p^\infty]=0$ and $\Sel_{p^\infty}(E/\Fin^{(d)})=0$.
 \end{enumerate}
\end{theorem}

\begin{proof}
First, we note that $\Sel_{p^\infty}(E/\Fin^{(d)})^{\vee}$ is torsion as a $\Lambda(G^{(d)})$-module.
This follows from the  argument given in \cite[Remark~2.2]{HV03}.
By Proposition~\ref{proposition: akasi series to pseudonull}, statements \eqref{72:2} and \eqref{72:3} are equivalent.
Statement \eqref{72:3} implies that $\Sel_{p^\infty}(E/\Fin^{(d)})^{\vee}$ is pseudonull as a $\Lambda(G^{(d)})$-module.
Thus, by Proposition~\ref{prop:Greenberg}, $\Sel_{p^\infty}(E/\Fin^{(d)})=0$.
Hence, \eqref{72:3} and \eqref{72:4} are equivalent.
It suffices to prove that statements \eqref{72:1} and \eqref{72:2} are equivalent.

First assume that \eqref{72:2} holds, i.e., that $\chi(G^{(d)}, E,p)=1$.
By \eqref{ec are equal}, we have the following equality relating the Euler characteristic for $G^{(d)}$ with that over for $\Gamma^{(d)}$
\[
\chi_t(G^{(d)}, E,p)=\chi_t(\Gamma^{(d)}, E,p).
\]
By the Euler characteristic formula,
\begin{equation}
\label{EC formula Fd}
\chi_t(\Gamma^{(d)}, E,p)=\frac{\#\Sha(E/F^{(d)})[p^\infty]\times \prod_{v} c_v^{(p)}(E/F^{(d)})\times \prod_{v|p}\# \left(\widetilde{E}(\kappa_v)[p^\infty]\right)^2}{\left(\#E\left(F^{(d)}\right)[p^\infty]\right)^2}.
\end{equation}
Since the Euler characteristic is an integral power of $p$, it suffices to show that the terms in the numerator are all equal to $1$.
First, by assumption, $\#\Sha(E/F^{(d)})[p^\infty]=1$.
Next, it is assumed that $\chi(\Gamma_\Q, E,p)=1$.
It follows that $c_\ell^{(p)}(E/\Q)=1$ for all primes $\ell\neq p$.
Since $d$ is assumed to be coprime to the level $N_E$, by Theorem~\ref{Tamagawa result from HKR}, the Tamagawa product $\prod_{v\nmid p} c_v^{(p)}(E/F^{(d)})$ is equal to $1$.
Once again, since $\chi(\Gamma_\Q, E,p)=1$, it follows that $a_p(E)\not\equiv 1\mod{p}$.
This also implies that $p\nmid \#\widetilde{E}(\F_{p^2})$.

Conversely, suppose that $E(F^{(d)})[p^\infty]$ is trivial and $\chi_t(\Gamma^{(d)}, E,p)=1$.
Then the terms in the numerator of \eqref{EC formula Fd} are all equal to $1$.
In particular, $\#\Sha(E/F^{(d)})[p^\infty]=1$.

The last assertion of the theorem follows from Propositions~\ref{prop:Greenberg} and \ref{prop: pseudonull}.
\end{proof}

\begin{corollary}
\label{MW cor vary F Zp2}
Let $E_{/\Q}$ be a fixed elliptic curve of conductor $N_E$ and set $F^{(d)}=\Q(\sqrt{-d})$.
Suppose that $d$ is a square-free integer coprime to $N_E$ with the properties that the conditions of Theorem~\ref{theorem varying imaginary quadratic} hold for the pair $(E,F^{(d)})$ and that $ \#\Sha(E/F^{(d)})[p^\infty]=0$.
As $d$ varies over all such square-free integers, the Mordell--Weil rank of $E(F_n^{(d)})=0$, for all $n$, where $F_n^{(d)}$ is the unique sub-extension of $\Fin^{(d)}$ with $\Gal(F_n^{(d)}/F^{(d)})\simeq (\Z/p^n)^2$.
\end{corollary}

\begin{proof}
By Theorem~\ref{theorem varying imaginary quadratic}, $\Sel_{p^\infty}(E/\Fin^{(d)})=0$ for all $d\nmid N_E$.
Consequently, $E(\Fin^{(d)})$ is finite since $E(\Fin^{(d)})\otimes\Qp/\Zp$ injects into $\Sel_{p^\infty}(E/\Fin^{(d)})$.
Thus, $E(F_n^{(d)})$ is also finite for all $n$.
\end{proof}

\begin{remark}
More generally, the following is true.
The only primes that ramify in the unique $\Zp^2$-extension above $F$, are the primes above $p$.
Therefore, $S(F_\cyc)$ is the set of primes of $F_\cyc$ lying above the primes above $p$ and the primes where $E$ has bad reduction in $F$.
Also, for any prime $w\nmid p$, it is known that $\dim H_w =0$.
From \eqref{LambdaH rank formula}, we deduce that
\[
\rank_{\Lambda(H)}\cX(E/\Fin)_f=\lambda_p(E/F_\cyc).
\]
It follows from \eqref{LambdaH rank formula} that
\[
\rank_{\Z}E(F_n) \leq \lambda_p(E/F_{\cyc})p^{n} + 2\corank_{\Zp}E(\Fin)(p).
\]
In the non-CM case, $E(\Fin)(p)$ is finite; hence $\corank_{\Zp}E(\Fin)(p)=0$.
In the CM case, this need not be true.
But the recent result in \cite{Ray21_noncom_rank} implies that even in this case,
\[
\rank_{\Z}E(F_n) \leq \lambda_p(E/F_{\cyc})p^{n}.
\]
\end{remark}

Theorem~\ref{theorem varying imaginary quadratic} provides insight into how often the Selmer group $\Sel_{p^\infty}(E/\Fin^{(d)})$ is zero.
As before, consider the imaginary quadratic field $F^{(d)}:=\Q(\sqrt{-d})$, and the associated character $\chi_d:\Gal(F^{(d)}/\Q)\rightarrow \{\pm1\}$.
Let $E^{(-d)}$ be the elliptic curve over $\Q$ defined by the twist of $E$ by the character $\chi_d$.
Then, we have that 
\[
\rank_{\Z}E(F^{(d)})=\rank_{\Z}E(\Q)+\rank_{\Z} E^{(-d)}(\Q).
\]
Since we have assumed that $\rank_{\Z}E(\Q)=0$, it follows that $\rank_{\Z}E(F^{(d)})=\rank_{\Z}E^{(-d)}(\Q)$.

\begin{conj}[Goldfeld]
For $x>0$ and $r\in \Z_{\geq1}$, define 
\[
N_r(E,x):=\#\{\absolute{d}<x| \rank_{\op{an}}(E_{d})=r\}.
\]
Then for $r\in \{0,1\}$, 
\[
N_r(E,x)\sim \frac{1}{2}\sum_{\absolute{d}<x} 1
\]
as $x\rightarrow \infty$, and the sum is over all square-free integers $d$.
\end{conj}
As $d$ ranges over all positive square-free integers for which $p$ splits in $F^{(d)}$, it is reasonable to expect that for $1/2$ of the values of $d$, upon base-change, $\rank_{\Z}E(F^{(d)})=0$.
Explicit calculations show that given an imaginary quadratic field $K$ and an elliptic curve $E_{/\Q}$ for which $\Sha(E/\Q)[p^\infty]=0$, it is rare for $\Sha(E/K)[p^\infty]\neq 0$ (in both the variation of $K$ and the prime $p$).
The reader is referred to \cite[Table 1]{HKR} for data on the growth of the $\Sha$-group upon base-change by an imaginary quadratic field.
Therefore, putting everything together, Theorem~\ref{theorem varying imaginary quadratic} shows that for  elliptic curves for which the hypotheses are satisfied, it is a rare occurrence for the Selmer group $\Sel_{p^\infty}(E/\Fin^{(d)})$ to be non-zero as $d$ varies over all positive square-free integers.
We are led to make the following conjecture.

\begin{conj}
\label{conjecture vary Fin zp2}
Let $E$ be an elliptic curve defined over $\Q$ and $p\geq 5$ a prime such that the hypotheses of Theorem~\ref{theorem varying imaginary quadratic} are satisfied.
Let $\mathcal{D}$ be the set of all positive square-free integers $d$ such that $p$ splits in $\Q(\sqrt{-d})$.
Then, for an infinite subset $\mathcal{D}'$ contained in $\mathcal{D}$, the Selmer group $\Sel_{p^\infty}(E/\Fin^{(d)})=0$ for all $d\in \mathcal{D}'$.
\end{conj}

\begin{remark}\label{BKLOS remark}
 When $E_{/\Q}$ and $F=\Q(\sqrt{-d})$, a refinement of H.~Yu's result by D.~Qiu (see \cite[p.~5051]{Qiu14}) proves that if $E$ has no $p$-torsion over $\Q$, then 
\[
\# \Sha(E/\Q)[p^\infty] \cdot \# \Sha(E^{(-d)}/\Q)[p^\infty] = \# \Sha(E/F)[p^\infty].
\]
In particular, given $E_{/\Q}$, for all but finitely many primes, we have 
\[
\# \Sha(E^{(-d)}/\Q)[p^\infty] = \# \Sha(E/F)[p^\infty].
\]
It is conjectured that for all primes $p$, the Tate--Shafarevich group of the twisted elliptic curve $\Sha(E_{s}/\Q)$ has an element of order $p$ for a positive proportion of $s\in \Q^{\times}\setminus \Q^{\times 2}$, when the elliptic curves are ordered by height \cite[Conjecture~1.1]{BKLOS}.
\end{remark}

\subsection{}\label{sec:8.2}
The next family we consider arises from false Tate curve extensions.
Let $F:=\Q(\mu_p)$ and define the false Tate curve extension as follows
\[
\Fin^{(m)} = \Q\left( \mu_{p^\infty}, \ m^{\frac{1}{p^n}} : n = 1, 2, \ldots \right).
\]
Let $G^{(m)}:=\Gal(\Fin^{(m)}/F)$.
As $m\nmid N_E p$ varies over all primes, $\{\Fin^{(m)}\}$ is viewed as a family of noncommutative admissible pro-$p$, $p$-adic Lie extensions of the fixed number field $F$.
We assume that the truncated $\Gamma_F$-Euler characteristic $\chi_t(\Gamma_F, E,p)$ is equal to $1$.

Recall that
\[
\chi_t(G^{(m)}, E,p)=\chi_t(\Gamma_F, E,p)\times \prod_{v\in \mathfrak{M}} \absolute{L_v(E, 1)}_p=\prod_{v\in \mathfrak{M}} \absolute{L_v(E, 1)}_p,
\]
where $\fM = \fM_m = \mathcal{P}_1(E, \Fin^{(m)}) \cup \mathcal{P}_2(E, \Fin^{(m)})$ (introduced in \eqref{definition of P1 P2}).
Since $E$ and $p$ are fixed, the truncated $\Gamma_F$-Euler characteristic remains unchanged upon varying $m$.
Therefore, to study the variation of the truncated $G^{(m)}$-Euler characteristic, we must study the variation of the local Euler factors.
For all but finitely many $m$, the set $\mathcal{P}_1=\emptyset$.
This is because $E$ is fixed and hence the primes of split multiplicative reduction (call them $\ell_1, \ldots, \ell_k$) are also fixed.
Thus, $\mathcal{P}_1\neq \emptyset$, precisely when $m=\ell_i$ for some $i$.

Next, we analyze the set $\mathcal{P}_2$.
First, we evaluate the proportion of primes $m$ such that $E(\kappa_v)[p]= 0$ for all primes $v|m$ of $\Q(\mu_p)$.
Here, $\kappa_v$ is the residue field at $v$, and $\kappa_v=\F_{m^f}$, where $f$ is the smallest positive integer such that $m^f\equiv 1\mod{p}$.
The value of $f$ is a divisor of $p-1=[\Q(\mu_p):\Q]$; it equals $1$ if $m$ splits completely in $\Q(\mu_p)$ and equals $p-1$ if it is  inert in $\Q(\mu_p)$, see \cite[Theorem~2.13]{Was97}.

Consider the Galois group $G_{E,p}:=\Gal(\Q(E[p])/\Q)$, and note that $G_{E,p}$ may be viewed as a subgroup of ${\GL}_2(\Z/p\Z)$ via the residual representation 
\[
\bar{\rho}_{E,p}:G_{E,p}\hookrightarrow {\GL}_2(\Z/p\Z).
\]
Assume that $m$ is coprime to the conductor of $E$.
In particular, $m$ is unramified in $\Q(E[p])$.
Let $\sigma_m\in G_{E,p}$ be the Frobenius at $m$.
The trace and determinant of $\bar{\rho}(\sigma_m)$ are as follows
\[
\trace\bar{\rho}(\sigma_m)=a_m(E)=m+1-\# \widetilde{E}(\F_m)\text{, and } \det\bar{\rho}(\sigma_m)=m.
\]
For the prime $v|m$ of $\Q(\mu_p)$, the field $\kappa_v=\F_{m^f}$.
According to a formula of A.~Weil (see \cite[Theorem~V.2.3.1]{silverman2009}),
\[
\#E(\kappa_v)=m^f+1-\alpha^f-\beta^f\equiv 2-\alpha^f-\beta^f\mod{p},
\]
where $\alpha$ and $\beta$ are the eigenvalues of $\bar{\rho}( \sigma_m)$.
This brings us to the following definition.
For $g\in G_{E,p}$, let $f(g)$ be the smallest integer $f\in \Z_{\geq 1}$ such that ${\det}\bar{\rho}(g)^f=1$.

\begin{definition}
Let $H_{E,p}$ consist of all $g\in G_{E,p}$ such that the eigenvalues $\alpha,\beta\in \bar{\F}_p$ of $\bar{\rho}(g)$ satisfy
\[
\alpha^{f(g)}+\beta^{f(g)}\neq 2.
\]
Since $(\alpha\beta)^{f(g)}=m^{f(g)}=1$, this condition is equivalent to $\alpha^{f(g)}\neq 1$.
\end{definition}

\begin{lemma}
\label{Hep over Gep lemma}
For a prime $m\neq p$, let $v$ be the prime of $\Q(\mu_p)$ above $m$, and $\kappa_v$ be the residue field at $v$.
The density of primes $m$, coprime to the conductor of $E$, for which $E(\kappa_v)[p]=0$ is $\left(\frac{\# H_{E,p}}{\# G_{E,p}}\right)$.
\end{lemma}

\begin{proof}
It follows from the definition of $H_{E,p}$ that $\sigma_m\in H_{E,p}$ if and only if $E(\kappa_v)[p]=0$.
The result follows from the Chebotarev density theorem.
\end{proof}

\begin{corollary}
\label{corollary chi=1}
Let $E_{/\Q}$ be an elliptic curve and $p$ be an odd prime.
Let $F:=\Q(\mu_p)$ and assume that the following conditions hold.
\begin{enumerate}[(i)]
 \item $\Sha(E/F)[p^{\infty}]$ is finite.
 \item The truncated Euler characteristic $\chi_t(\Gamma_F, E,p)$ is defined and equal to $1$.
 \item The $\mathfrak{M}_H(G)$ conjecture is true for $\Sel_{p^\infty}(E/\Fin^{(m)})$ at every prime $m\neq p$.
\end{enumerate}
Then, the density of primes $m\nmid N_E p$ for which $\chi_t(G^{(m)}, E,p)=1$ is at least $\left(\frac{\# H_{E,p}}{\# G_{E,p}}\right)$.
\end{corollary}

\begin{proof}
Note that the assumptions made in the statement of this result ensure that the hypotheses in Lemma~\ref{Hep over Gep lemma} hold.
The proof is immediate from the aforementioned lemma.
\end{proof}

\begin{theorem}
\label{thm: Hep-Gep density}
Let $E_{/\Q}$ be an elliptic curve such that $\rank_{\Z}E(\Q(\mu_p))=0$ and such that the conditions of Corollary \ref{corollary chi=1} are satisfied.
Then, the density of primes $m\nmid N_E p$ for which $\Sel_{p^\infty}(E/\Fin^{(m)})=0$ is at least $\left(\frac{\# H_{E,p}}{\# G_{E,p}}\right)$.
\end{theorem}

\begin{proof}
Let $m$ be a prime number for which $\chi_t(G^{(m)}, E,p)=1$.
Then, it follows from Proposition~\ref{proposition: akasi series to pseudonull} that $\Sel_{p^\infty}(E/\Fin^{(m)})^\vee$ is pseudonull as a $\Lambda(G)$-module and the associated Akashi series is a unit.
However, we have noted in Theorem~\ref{FT no non zero pseudonull submodules}, that $\Sel_{p^\infty}(E/\Fin^{(m)})^\vee$ has no non-trivial pseudonull submodules.
Hence, it follows from Proposition~\ref{prop: pseudonull} that $\Sel_{p^\infty}(E/\Fin^{(m)})=0$.
The assertion is now immediate from Corollary \ref{corollary chi=1}.
\end{proof}

\begin{remark}
\label{rk: coming from appendix}
In the above theorem, if the residual representation
\[
\bar{\rho}:\Gal(\bar{\Q}/\Q)\rightarrow \GL_2(\F_p)
\]
on $E[p]$ is surjective, it follows from Remark~\ref{rk:1st-portion} that $\left(\frac{\# H_{E,p}}{\# G_{E,p}}\right)\geq \frac{p^2}{2(p^2-1)}>\frac{1}{2}$.
Recall that if $E$ is an elliptic curve without CM, then  $\bar\rho$ is surjective for $p\gg 0$.
Therefore, on combining Remarks~\ref{rk:1st-portion} and \ref{rk:2nd-portion}, we have
\[
\liminf_{p\rightarrow \infty} \left(\frac{\# H_{E,p}}{\# G_{E,p}}\right)\geq\frac{5}{8}.
\]
\end{remark}

Similar to Corollary~\ref{MW cor vary F Zp2}, combining Theorem~\ref{thm: Hep-Gep density} and Remark~\ref{rk: coming from appendix} gives
\begin{corollary}
Let $E_{/\Q}$ be a non-CM elliptic curve with $\rank_{\Z} E(\Q(\mu_p))=0$ and such that the conditions of Corollary \ref{corollary chi=1} are satisfied.
Furthermore, suppose that the residual representation on $E[p]$ is surjective.
For at least \emph{half} of the prime numbers $m$ not dividing $pN_E$, we have $\rank_{\Z}E(F_n)=0$ at each finite layer of $\Fin^{(m)}/F$.
\end{corollary}

We illustrate this corollary via an explicit example.
Consider the rank 0 elliptic curve with LMFDB label \href{https://www.lmfdb.org/EllipticCurve/Q/11/a/1}{11.a1} with good \emph{ordinary} reduction at $p=3$.
Now, consider the quadratic field extension $F = \Q(\mu_3) = \Q(\sqrt{-3})$.
The base-change curve is \href{https://www.lmfdb.org/EllipticCurve/2.0.3.1/121.1/a/1}{121.1-a1}; it can be checked that the order of the Shafarevich--Tate group is trivial and so is the Mordell--Weil rank.
The Tamagawa number $\tau_{11}=1$.
Since $3$ is coprime to the conductor of $E$, Theorem~\ref{Tamagawa result from HKR} asserts that $\tau_{11}^{(F)}=1$.
On the other hand, $\#\widetilde{E}(\F_3)=5$.
Hence, by the Euler characteristic formula, $\chi(\Gamma_F, E,p)=1$.
As $m$ varies over all primes (not equal to $3, 11$), for \emph{at least} $\frac{9}{16}$ of the primes, $\rank_{\Z}(E(F_n^{(m)}))=0$ at each finite layer of $\Fin^{(m)}/F$.

\subsubsection{}
\label{rank questions for false Tate curve extensions}
We now analyze the $\Lambda(H)$-rank of the Selmer group and the growth of the Mordell--Weil rank of the elliptic curve over a false Tate curve extension of the cyclotomic number field, $\Q(\mu_p)$.
In this case, \eqref{LambdaH rank formula} is
\begin{equation}
\label{lambda-H formula for FTC}
\rank_{\Lambda(H)}\cX(E/\Fin^{(m)})_f=\lambda_p(E/F_\cyc)+\sum_{\substack{v|\ell, \ \ell|N_E, \ \ell|m\\ \ell \textrm{ split multiplicative}}}1 + \sum_{\substack{q|m, \\ v|q, E(F_v)[p^\infty]\neq 0}} 2.
\end{equation}
Here we have used that $\dim H_w =1$ for precisely those primes (away from $p$) which ramify in the false Tate curve extension, $\Fin^{(m)}/F$.

\begin{corollary}
Let $E_{/\Q}$ be a non-CM elliptic curve.
As $m\nmid pN_E$ varies over all primes, for \emph{at least half} of such primes, 
\[
\rank_{\Z}E(F_n)\leq \lambda_p(E/F_{\cyc})p^n
\]
at each finite layer of $\Fin^{(m)}/F$.
\end{corollary}

\begin{proof}
By choosing $m$ such that $\gcd(m,N_E)=1$, we have
\[
\sum_{\substack{\ell|N_E, \ \ell|m\\ \ell \textrm{ split multiplicative}}}1 =0.
\]
Next, as $m\nmid pN_E$ varies over all primes, Lemma~\ref{Hep over Gep lemma} in conjunction with Remark~\ref{rk: coming from appendix} implies that for \emph{at least half} of the primes
\[
\rank_{\Lambda(H)}\cX(E/\Fin^{(m)})_f=\lambda_p(E/F_\cyc).
\]
The result is immediate from Theorem~\ref{Thm: Hung-Lim Theorem} and Remark~\ref{rmk: finiteness of p-primary torsion points over certain extensions}.
\end{proof}

\subsection{}
\label{section 8-3}
We fix a prime $p\geq 5$ and an elliptic curve $E_{/\Q}$ with good \emph{ordinary} reduction at $p$.
Consider the pair $(E, A)$ such that $A$ varies over a twist-equivalent family of non-CM elliptic curves over $\Q$.
This gives rise to (varying) extensions ${\Fin}_{,A}=\Q(A[p^\infty])$ of dimension $4$ over $\Q$.
We set $\mathcal{G}_A := \Gal({\Fin}_{,A}/\Q)$ and study the variation of the truncated Euler characteristic $\chi_t(\mathcal{G}_A, E,p)$, i.e., we study how often this quantity is equal to $1$.

\begin{remark}
\label{difficulty in pro-p case for triv ext}
In this case, $\mathcal{G}_A$ is \emph{not} a pro-$p$ extension.
Even though we can apply Theorem~\ref{EC formula theorem} to study how often the truncated $\mathcal{G}_A$-Euler characteristic is a unit, it gives no information on pseudonullity of the $p$-primary Selmer group over ${\Fin}_{,A}$.
However, it provides information regarding the Akashi series, since the Euler characteristic is equal to $1$ if and only if the leading term of the Akashi series is a unit (in $\Z_p$).
The extension $\cF_{\infty,A}$ is a pro-$p$ extension of $F_A=\Q(A[p])$, i.e., the Galois group $G_A= \Gal(\cF_{\infty,A}/F_A)$ is a pro-$p$ group.
Unfortunately, it is difficult to study the $\Gamma_F$-Euler characteristic $\chi_t(\Gamma_F, E,p)$ on average.
The main difficulty is in studying the behaviour of the Tate--Shafarevich group over $\Q(A[p])$.
\end{remark}

The question is simple to answer when $p$ divides the truncated $\Gamma_{\Q}$-Euler characteristic $\chi_t(\Gamma_{\Q}, E,p)$.
Indeed, the same reasoning as Lemma~\ref{Gamma Q Gamma F basic lemma} shows that $\chi_t(\Gamma_{F}, E,p)$ is divisible by $p$.
Therefore, we assume that $\chi_t(\Gamma_\Q, E,p)=1$.
In view of results proven in \cite[Section 3]{KR21}, the aforementioned hypothesis is satisfied \textit{most of the time}.
Our main result on the question is the following, which we prove at the end of this section.

\begin{theorem}
\label{s7.3 mainthm}
Let $E_{/\Q}$ be an elliptic curve and $p\geq 5$ be a prime of good \emph{ordinary} reduction of $E$.
Assume that the following equivalent conditions are satisfied
\begin{enumerate}[(i)]
 \item $\mu_p(E/\Q_{\cyc})=0$ and $\lambda_p(E/\Q_{\cyc})=\rank_{\Z}E(\Q)$,
 \item $\chi_t(\Gamma_\Q, E,p)=1$.
\end{enumerate}
Then, there are infinitely many non-CM elliptic curves $A_{/\Q}$ such that
\[
\chi_t(\mathcal{G}_A, E,p)=1.
\]
\end{theorem}

\begin{remark}
Suppose in addition that $E$ does not have CM and has Mordell--Weil rank 0.
Then for $100\%$ of the primes $p$ where $E$ has  good \emph{ordinary} reduction, the conditions of Theorem~\ref{s7.3 mainthm} are satisfied, see \cite[Theorem~5.1]{Greenberg}.
\end{remark}

Recall that
\begin{equation}
\label{GECeqn}
\chi_t(\mathcal{G}_A, E,p)=\chi_t(\Gamma_\Q, E,p)\times \prod_{v\in \fM_A} \absolute{L_v(E, 1)}_p=\prod_{v\in \fM_A} \absolute{L_v(E, 1)}_p,
\end{equation}
where $\fM = \fM_A$ consists of precisely those primes $v\neq p$ for which the inertia group in $\mathcal{G}_A$ is infinite.
Note that $v\neq p$ is contained in $\mathfrak{M}_A$ if and only if $A$ has potentially multiplicative reduction at $v$, see \cite[Theorem~3.1]{CoatesSchneiderSujatha_Links_between}.
Lemma~\ref{criterion for p to divide Lv(E,1)} gives a criterion for $p$ to divide $\absolute{L_v(E,1)}_p$.
In this section, we vary over all ${\Fin}_{,A}/\Q$ and the goal is to estimate how often is $\chi_t(\mathcal{G}_A, E,p) =1$.
The above theorem asserts that such a property holds for \emph{infinitely many} non-CM elliptic curves $A_{/\Q}$.
However, it will follow from Lemma~\ref{chi is 1 condition for triv case, vary A} and the estimates in Lemma~\ref{0 percent} that the proportion of such elliptic curves is 0\%.

Define 
\begin{align*}
\mathfrak{T} := & \{v\neq p: E \textrm{ has good reduction at }v \textrm{ and }p\nmid\# \widetilde{E}(\kappa_v)\} \cup \\
& \{v\not \equiv 1\pmod{p}: E \textrm{ has split multiplicative reduction at }v \} \cup \\
& \{v\not\equiv-1\pmod{p}: E \textrm{ has non-split multiplicative reduction at }v \} \cup \\
&\{
v\neq p: E \textrm{ has additive reduction at }v\}.
\end{align*}
This is precisely the set of primes in $\Q$ where $\absolute{L_v(E,1)}_p =1$.
Since the set of bad primes of $E$ is finite, it follows that $\mathfrak{T}$ has natural density
\[
\lim_{x\rightarrow \infty} \frac{\#\{v\in \mathfrak{T}\mid v\leq x\}}{\pi(x)}=1-\frac{1}{p},
\]
see \cite[Theorem~1]{Coj04}.
Here, $\pi(x)$ denotes the prime counting function.

\begin{lemma}
\label{chi is 1 condition for triv case, vary A}
Let $E_{/\Q}$ be an elliptic curve satisfying the hypotheses of Theorem~\ref{s7.3 mainthm}, and $A_{/\Q}$ be any elliptic curve.
Then, the following conditions are equivalent
\begin{enumerate}
 \item $\mathfrak{M}_A$ is contained in $\mathfrak{T}$,
 \item $\chi_t(\mathcal{G}_A, E,p)=1$.
\end{enumerate}
\end{lemma}

\begin{proof}
According to \eqref{GECeqn}, $\chi_t(\mathcal{G}_A, E,p)=1$ if and only if $\absolute{L_v(E, 1)}_p=1$ for all primes $v\in \mathfrak{M}_A$.
Moreover, $\absolute{L_v(E, 1)}_p=1$ if and only if $v\in \mathfrak{T}$.
Hence, the result follows.
\end{proof}

\begin{proof}[Proof of Theorem~\ref{s7.3 mainthm}]
One needs to show that there are infinitely many non-CM elliptic curves $A_{/\Q}$ such that $\mathfrak{M}_A$ is contained in $\mathfrak{T}$.
Observe that a curve $A$ for which the $j$-invariant is an integer has potentially good reduction at all primes; hence $\mathfrak{M}_A=\emptyset$ and $\chi_t(\mathcal{G}_A, E,p)=1$ for all pairs $(E,p)$ satisfying the assumptions of the theorem.

Let $A_0$ be the elliptic curve with Cremona label \href{https://www.lmfdb.org/EllipticCurve/Q/128/a/1}{128a2}.
This is a non-CM elliptic curve with $j$-invariant, $j(A_0)=2^7$; hence, $\mathfrak{M}_{A_0}=\emptyset$.
For any odd prime $q\in \mathfrak{T}\setminus \mathfrak{M}_A$, let $A_q$ be the quadratic twist of $A$ by the non-trivial quadratic character ramified only at $q$.
Since $\mathfrak{M}_{A_q}$ is contained in $\mathfrak{T}$, we deduce that $\chi_t(G_{A_q}, E,p)=1$.
This completes the proof.
\end{proof}

\subsubsection{}
One way of expressing density results is to define the \emph{height function} of a long Weierstrass equation
\[
y^2+a_1 xy+a_3 y=x^3+a_2 x^2+a_4 x+a_6
\]
with integer coefficients $\boldsymbol{a} = (a_1, a_2, a_3, a_4, a_6)\in \Z^5$ to be
\[
\height(\boldsymbol{a}) = \max_{i} \absolute{a_i}^{1/i},
\]
and then order such equations by height.
The proportion of curves that lie in a set $\mathcal{S}\subseteq \Z^5$ is then given by 
\[
\mathfrak{d}(\mathcal{S}) = \lim_{x\rightarrow \infty}\frac{\#\left\{\boldsymbol{a}\in \mathcal{S} \ | \ \height(\boldsymbol{a}\leq x)\right\}}{ \#\left\{\boldsymbol{a}\in \Z^5 \ | \ \height(\boldsymbol{a}\leq x)\right\}}.
\]

We now show that if we arrange all elliptic curves over $\Q$ by height, the proportion of elliptic curves $A_{/\Q}$ in Theorem~\ref{s7.3 mainthm} has density 0.
The set $\fM_A$ consists precisely of the primes of potentially multiplicative reduction of $A_{/\Q}$ and Lemma~\ref{chi is 1 condition for triv case, vary A} asserts that $\chi_t(\mathcal{G}_A, E,p)=1$ if and only if $\fM_A$ is contained in $\mathfrak{T}$.
Thus, to count the proportion of $A_{/\Q}$ such that $\chi_t(\mathcal{G}_A, E,p)=1$, it suffices to count elliptic curves $A_{/\Q}$ with good reduction or potentially good reduction at \emph{all the primes} in the complement of $\mathfrak{T}$ (say, $\mathfrak{T}^\prime$).

The following proposition will be useful for the estimates established in this section.

\begin{lemma}
\label{lemma: reduction type proportions}
Let $q$ be any prime.
Suppose we order all elliptic curves defined over $\Q$ by height.
Of all such curves, 
\begin{enumerate}
\item the proportion with multiplicative reduction at $q$ is $\frac{q^8(q-1)}{q^{10}-1}$.
\item the proportion with potentially multiplicative reduction (but not multiplicative reduction) at $q$ is $\frac{q^3(q-1)}{q^{10}-1}$.
\item the proportion with additive reduction at $q$ is $\frac{q^8-1}{q^{10}-1}$.
\item the proportion with good reduction at $q$ is $\frac{q^9(q-1)}{q^{10}-1}$.
\end{enumerate} 
\end{lemma}

\begin{proof}
See \cite[Propositions~2.2 and 2.6]{CS20}.
\end{proof}

\begin{remark}
\label{remark for minimal}
If we restrict our attention to \emph{minimal} Weierstrass equations then we have to multiply each proportion in Lemma~\ref{lemma: reduction type proportions} by $\frac{q^{10}-1}{q^{10}}$.
We will often do so in Section \ref{S: Vary EC}.
\end{remark}

\begin{lemma}
\label{0 percent}
Let $\mathfrak{T}^\prime$ be a fixed subset of primes in $\Q$.
The proportion of elliptic curves defined over $\Q$ ordered by height with potentially good reduction at all primes $q_i\in \mathfrak{T}^\prime$ is
\[
\prod_{q_i\in \mathfrak{T}^\prime}\left( 1 - \frac{q_i^3 (q_i-1)(q_i^5 + 1)}{q_i^{10}-1}\right).
\]
Furthermore, if $\mathfrak{T}^\prime$ has positive density, then, the proportion of such elliptic curves is 0.
\end{lemma}

\begin{proof}
For each prime $q_i\in \mathfrak{T}^\prime$, the proportion of elliptic curves with either good or potentially good reduction at $q_i$ is given by the complement of the proportion of elliptic curves with multiplicative or potentially multiplicative reduction at $q_i$.
It follows from Lemma~\ref{lemma: reduction type proportions} that this proportion is
\[
1 - \frac{q_i^8(q_i-1)}{q_i^{10}-1} - \frac{q_i^3(q_i-1)}{q_i^{10}-1} = 1 - \frac{q_i^3 (q_i-1)(q_i^5 + 1)}{q_i^{10}-1}.
\]
The first assertion now follows, see \cite[Section 3]{CS20}.

To prove the second assertion, observe that for $q_i\gg0$,
\[
1 - \frac{q_i^3 (q_i-1)(q_i^5 + 1)}{q_i^{10}-1}
 \leq 1- \frac{1}{2q_i}.
\]
It is an easy exercise to show that $\prod_{q_i \in \mathfrak{T}^\prime} \left( 1- \frac{1}{2q_i}\right)=0$ if and only if $\sum_{q_i\in \mathfrak{T}^\prime} \frac{1}{q_i}$ diverges.
Since the density of $\mathfrak{T}^{\prime}$ is positive, it follows that $\sum_{q_i\in \mathfrak{T}^\prime} \frac{1}{q_i}$ diverges.
\end{proof}

\subsubsection{}
\label{GA Ec Zp2 case}
In Remark~\ref{difficulty in pro-p case for triv ext}, we mentioned that it has not been possible for us to study the $G_A$-Euler characteristic directly.
However, we now make some observations on the $\Lambda(H)$-rank of $\mathcal{X}(E/\Fin)_f$ which will shed some light on the $G_A$-Euler characteristic formula.

Consider the pair of elliptic curves $(E, A)$, both defined over $\Q$ and such that $A$ is not a CM elliptic curve.
Throughout, $p\geq 5$ is a fixed prime with good reduction at $p$ and the base field is $F=F_A= \Q(A[p])$.
Denote the pro-$p$ $p$-adic Lie extension by $\Fin = {\Fin}_{,A}= \Q(A[p^\infty])$.
The corresponding Galois group is $G_{A}$, and write $H =H_A= \Gal({\Fin}_{,A}/F_{\cyc})$.

For any prime $w|v$ (where $v\nmid p$ in $F$), $\dim {H}_w\geq 1$ precisely when the reduction type is potentially multiplicative, see \cite[Lemma~2.8(i)]{Coates_Fragments}.
Thus, \eqref{LambdaH rank formula} becomes
\[
\rank_{\Lambda(H)}\mathcal{X}(E/{\Fin})_f = \lambda_p(E/F_{\cyc}) + \sum_{\substack{v\nmid p,\\ v \textrm{ pot. mult. for }A}} \corank_{\Zp} Z_{v}(F_{\cyc,w}).
\]

Let $\mathcal{SM}$ denote the primes of split multiplicative reduction of $E$.
For simplicity, we assume that $\mathcal{SM}\neq \emptyset$ and contains a prime $\geq 5$.
As $A$ varies over all elliptic curves defined over $\Q$ (ordered by height) with good reduction at $p$, the base field $F_A=\Q(A[p])$ varies.
We would like to calculate for what proportion of $A_{/\Q}$ does the following inequality hold,
\[
\rank_{\Lambda(H_A)}\mathcal{X}(E/{\Fin}_{,A})_f > \lambda_p(E/F_{A,\cyc})\geq 0.
\]
In other words, as $A$ varies we want to find how often is
\[
\sum_{\substack{v\nmid p,\\ v \textrm{ pot. mult. for }A}} \corank_{\Zp} Z_{v}(F_{\cyc,w}) >0.
\]
The above inequality holds for all $A_{/\Q}$ with potentially multiplicative reduction at \emph{at least} one prime in $\mathcal{SM}$.
Therefore, to get a lower bound on the density of such elliptic curves, we require that $A$ has potentially multiplicative reduction at \emph{at least} one prime of $\mathcal{SM}$.

Since $A$ has potentially multiplicative reduction at \emph{at least} one prime $\geq 5$, it automatically follows that such elliptic curves are non-CM.
Indeed, if $p$ divides the discriminant of a quadratic order $\mathcal{O}$, then
all curves with endomorphism ring isomorphic to $\mathcal{O}$ have additive reduction at $p$, see \cite[p.~1]{CP19}.
By an application of Lemma~\ref{lemma: reduction type proportions}, we conclude that the proportion of such elliptic curves is
\[
1-\prod_{\ell\in \mathcal{SM}}\left( 1 - \frac{\ell^3(\ell-1)(\ell^5 +1)}{\ell^{10}-1}\right).
\]

We record this observation below.
\begin{proposition}
\label{positive Lambda H rank}
Let $E_{/\Q}$ be a fixed elliptic curve with good \emph{ordinary} reduction at a fixed prime $p\geq 5$.
Suppose further that $E$ has \emph{at least} one prime ($\geq 5$) of \emph{split multiplicative reduction}.
Then, as $A$ varies over all elliptic curves over $\Q$, for a \emph{positive proportion} of $A$, we have
\[
\rank_{\Lambda(H_A)}\mathcal{X}(E/{\Fin}_{,A})_f > 0.
\]
\end{proposition}

\section{Results for fixed $\Fin$ and $p$ as $E_{/\Q}$ varies}
\label{S: Vary EC}

We fix a prime $p\geq 5$, and an admissible $p$-adic Lie-extension $\cF_\infty$ of $\Q$.
In this section, we are once again interested in studying the related questions discussed in \S\ref{S: Vary extension}.
But now, the pair $(\Fin,p)$ are fixed and the elliptic curve $E$ varies.

The arguments of this section rely on short Weierstrass equations, hence we use a modified notion of height.
Any elliptic curve $E_{/\Q}$ admits a unique Weierstrass equation,
\begin{equation}
\label{weier}
E:Y^2 = X^3 + aX + b
\end{equation}
where $a, b$ are integers and $\gcd(a^3 , b^2)$ is not divisible by any twelfth power.
Since $p\geq 5$, such an equation is minimal.
Recall that the \emph{height of} $E$ satisfying the minimal equation $\eqref{weier}$ is given by $H_{\min}(E) := \max\left(\absolute{a}^3, b^2\right)$.
Let $\mathcal{E}$ be the set of isomorphism classes of elliptic curves defined over $\Q$.
For any subset $\mathcal{S}\subset \mathcal{E}$, let $\mathcal{S}(x)$ consist of all $E\in \mathcal{S}$ such that $H_{\min}(E)<x$.
The density of $\mathcal{S}$ (if it exists) is defined as the following limit 
\[
\mathfrak{d}(\mathcal{S}):=\lim_{x\rightarrow \infty} \frac{\#\mathcal{S}(x)}{\# \mathcal{E}(x)}.
\]
The \emph{upper density} $\bar{\mathfrak{d}}(\mathcal{S})$ (resp. \emph{lower density} $\underline{\mathfrak{d}}(\mathcal{S})$) is defined by replacing the above limit by $\limsup_{x\rightarrow \infty}\frac{\#\mathcal{S}(x)}{\# \mathcal{E}(x)}$ (resp. $\liminf_{x\rightarrow \infty}\frac{\#\mathcal{S}(x)}{\# \mathcal{E}(x)}$).
As $E$ ranges over the set of elliptic curves, we study the variation of invariants associated to the Selmer group $\Sel_{p^\infty}(E/\cF_\infty)$.

\begin{remark}
In this section, we restrict ourselves to minimal Weierstrass equations because for our main theorems here, we rely on a result of H.~W.~Lenstra (see \cite[Proposition~1.8]{Lenstra_annals}), where such an assumption is made.
\end{remark}

In \S\ref{Euler characteristic vary E Zp2}, we consider the case of the (unique) $\Zp^2$-extension over a fixed imaginary quadratic field, $F$.
Our results in this section indicate that \emph{most of the time}, the truncated $G$-Euler characteristic is a $p$-adic unit.
Furthermore, we study the variation of the Mordell--Weil rank growth at finite layers of the $\Zp^2$-extension and supplement our results with a concrete example.
In \S\ref{section 9-2}, we consider the case of false Tate curve extensions.
The main results are Theorems~\ref{FTC thm varying elliptic curve} and \ref{FTC vary EC}, where we estimate the upper density of the proportion of rank 0 elliptic curves over $\Q$ with non-trivial $G$-Euler characteristic (resp. $\mathcal{G}$-Euler characteristic).
In the pro-$p$ situation, we study a finer question pertaining to $\Lambda(H)$-ranks.
In \S\ref{section 9-3} we fix a non-CM elliptic curve ${E_0}_{/\Q}$ and consider the extension $\Q(E_0[p^\infty])/\Q$.
As $E_{/\Q}$ varies over rank 0 elliptic curves with good \emph{ordinary} reduction at $p$, We prove analogous estimates for the $\mathcal{G}_{E_0}$-Euler characteristic and variation of $\Lambda(H)$-ranks (in the pro-$p$ situation) in this four-dimensional non-abelian extension.

\subsection{}
\label{Euler characteristic vary E Zp2}
Let $p\geq 5$ be a fixed prime and $F= \Q(\sqrt{-d})$ be a fixed imaginary quadratic field.
Consider the unique $\Zp^2$-extension $\Fin/F$ and set $G=\Gal(\Fin/F)$.
As discussed previously,
\[
\chi_t(\Gamma_F, E,p) = \chi_t(G, E,p).
\]
Let $\mathcal{E}^F_p\subset \mathcal{E}$ be the subset of elliptic curves $E_{/\Q}$ such that
\begin{enumerate}[(a)]
 \item $\rank_{\Z}E(F)=0$,
 \item $E$ has good \emph{ordinary} reduction at $p$, 
 \item $E$ has good reduction at $\ell=2,3$,
 \item $\chi_t(\Gamma_F, E,p) \neq 1$.
\end{enumerate}
In this section, we show that there is an upper bound on the upper-density $\bar{\mathfrak{d}}(\mathcal{E}^F_p)$.
The methods employed here extend those in \cite[Section 8.2]{HKR}, where we studied how often the \emph{anticyclotomic} Euler characteristic $\chi(\Gal(F^{\ac}/F), E,p)$ is equal to 1.
The results we prove in the current setting require more detailed analysis of Euler characteristics, which translate to explicit estimates.

We consider the following terms
\begin{itemize}
 \item $\Sha_p(E/F) := \#\Sha(E/F)[p^{\infty}]$,
 \item $\tau_p(E/F) :=\prod_{v}c_v^{(p)}(E/F)$,
 \item $\alpha_p(E/F) :=\prod_{\mathfrak{p}|p}\#\widetilde{E}(k_\mathfrak{p})[p^{\infty}]$.
\end{itemize}

\begin{definition}
Let $E_{/\Q}$ be an elliptic curve.
We say that $E$ satisfies $(\dagger)$ if the Kodaira type at $\ell=2,3$ is not of the form $\textrm{I}_n$ for some $n$ divisible by $p$.
\end{definition}

We remark that an elliptic curve with good reduction at $\ell=2,3$ satisfies $(\dagger)$.
\begin{definition}
Let $\mathcal{E}_{1,F}(x)$, $\mathcal{E}_{2,F}(x)$, and $\mathcal{E}_{3,F}(x)$ be the subset of elliptic curves $E\in \mathcal{E}(x)$ with $\op{rank}_{\Z}(F)=0$, satisfying $(\dagger)$, for which $p$ divides $\Sha_p(E/F)$, $\tau_p(E/F)$, and $\alpha_p(E/F)$, respectively.
\end{definition}

Before stating the main theorem in this section, we have to introduce some additional notation.
For $\kappa=(a,b)\in \F_p\times \F_p$ with $\Delta(\kappa):=4a^3+27b^2$ non-zero, we write $E_\kappa$ for the  elliptic curve defined by the Weierstrass equation 
\[
E_{\kappa}:y^2=x^3+ax+b.
\]
This tuple $\kappa$ is not uniquely determined by the isomorphism class of $E_{\kappa}$.
If $p$ is split in $F$, the residue field $\kappa_v$ of $F$ at any prime $v|p$ is equal to $\F_p$.
Denote by $\mathfrak{S}_p$ the set of pairs $\kappa=(a,b)\in \F_p\times \F_p$ such that $E_{\kappa}$ contains a point of order $p$ over $\F_p$.
When $p$ is inert, the residue field is $\F_{p^2}$.
Denote by $\mathfrak{A}_p$ the set of pairs $\kappa=(a,b)\in \F_p\times \F_p$ such that $p$ divides $\# \widetilde{E}(\F_{p^2})$.
Define 
\[
b(p):=\begin{cases}\# \mathfrak{A}_p &\text{ if } p \text{ is inert in }F,\\
\# \mathfrak{S}_p &\text{ otherwise.}
\end{cases}
\]
Note that $p$ divides $\# \widetilde{E}(\F_{p^2})$ if and only if $a_p \equiv \pm 1 \pmod{p}$ (see \cite[Lemma~8.17]{HKR}).
Since the curves with $a_p=1$ are quadratic twists of curves for which $a_p=-1$, the numbers of curves with $a_p=1$ and $a_p=-1$ are the same.
Therefore, $\#\mathfrak{A}_p=2\#\mathfrak{S}_p$, see \cite[p.~22 last paragraph]{HKR}.

\begin{theorem}
\label{main result varying elliptic curve}
With the notation as before, 
\[
\bar{\mathfrak{d}}(\mathcal{E}^F_p)< \bar{\mathfrak{d}}(\cE_{1,F})+ \left( \zeta(p) -1 \right) +\zeta(10)\cdot \frac{b(p)}{p^2}.
\]
\end{theorem}

\begin{proof}
The proof of \cite[Theorem~8.19]{HKR} goes through verbatim.
\end{proof}

The heuristics for $\limsup_{x\rightarrow\infty}\frac{\#\cE_{1,F}(x)}{\#\cE(x)}$ are due to C.~Delaunay \cite{Del07} and can be considered as an analogue of the Cohen--Lenstra heuristics for the Tate--Shafarevich group of an elliptic curve.
We explain it briefly.

Let $\mathscr{E}$ denote the set of isomorphism classes of elliptic curves defined over $\Q$ with rank $0$.
For $x>0$, set $\mathscr{E}(x)$ to be the subset of $\mathscr{E}$ consisting of $E$ such that $H_{\op{min}}(E)\leq x$.
Assume that for every elliptic curve $E_{/\Q}$, the $p$-primary part of the Tate-Shafarevich group $\Sha(E/F)[p^\infty]$ is finite.
The heuristic of Delaunay states that
\[
\limsup_{x\rightarrow \infty} \frac{\#\{E\in \mathscr{E}(x)\mid \Sha(E/F)[p]\neq 0\}}{\#\mathscr{E}(x)}=f_0(p),
\]
where $f_0(p)$ is given by
\[f_0(p)=1-\prod_{j=1}^{\infty} \left(1-\frac{1}{p^{2j-1}}\right)=\frac{1}{p}+\frac{1}{p^3}-\frac{1}{p^4}+\frac{1}{p^5}-\frac{1}{p^6}\dots.\]

These values get smaller as $p$ gets larger.
For a brief summary of the Cohen-Lenstra philosophy, see the discussion preceding \cite[Heuristic~9.2 and Theorem~9.3]{HKR}.

Assuming the heuristic, it follows that 
\begin{equation}
\label{del heuristics}
\limsup_{x\rightarrow\infty}\frac{\#\cE_{1,F}(x)}{\#\cE(x)}\leq f_0(p).
\end{equation}

In \cite[Table 2]{HKR}, values of $\#\mathfrak{S}_p/p^2$ for the primes $7\leq p< 150$ are noted.
However, we are now able to say more about these values.

Let $\mathbb{F}_q$ be a finite field of characteristic $p\neq 2,3$.
Denote by $N(t)$ the number of $\mathbb{F}_q$-isomorphism classes of elliptic curves that have exactly $q+1-t$ points.
This quantity can be computed explicitly in terms of the \emph{Kronecker class number} (written as $H(\cdot)$) when $q$ is not a square (see \cite[p.~184]{Schoof87}).
More precisely, when $q$ is \emph{not} a square, for all $t\in \Z$,
\begin{equation}
N(t) = 
\begin{cases}
H(t^2 - 4q) & \textrm{ if } t^2<4q \textrm{ and } p\nmid t\\
H(-4p) & \textrm{ if } t=0.
\end{cases}
\end{equation}

\begin{lemma}
\label{getting hold of frakSp}
With the notation as above,
\[
\# \mathfrak{S}_p \leq \left(\frac{p-1}{2}\right) H(1-4p) \leq C p^{\frac{3}{2}} \log p \left( \log \log p\right)^2,
\]
where $C$ is an effectively computable positive constant.
In particular,
\[
\lim_{p\rightarrow \infty}\frac{\# \mathfrak{S}_p}{p^2} = 0.
\]
\end{lemma}

\begin{proof}
Let $E_{a,b}$ denote the elliptic curve $Y^2=X^3+aX+b$ with $(a,b)\in \F_p\times \F_p$.
Then $E_{a,b}$ is isomorphic to $E_{a',b'}$ over $\F_p$ if and only if 
\[a'=c^4 a\text{ and }b'=c^6 b\] for some element $c\in \F_p^\times$.
Thus, the number of curves $E_{a',b'}$ that are isomorphic to $E_{a,b}$ is at most $\frac{p-1}{2}$.
The number of elliptic curves up to isomorphism with $\#\widetilde{E}(\F_p)=p$ is $N(1)=H(1-4p)$.
This proves the first inequality.

The second inequality is proven in \cite[Proposition~1.8]{Lenstra_annals}.
The assertion that
\[
\lim_{p\rightarrow \infty}\frac{\# \mathfrak{S}_p}{p^2} = 0
\]
is an immediate consequence.
\end{proof}

\noindent Putting these assertions together and assuming the Cohen--Lenstra type heuristic for $p| \#\Sha(E/F)$ discussed in \eqref{del heuristics}, we have
\[
\limsup_{x\rightarrow \infty} \frac{\#\mathcal{E}^F_p(x)}{\#\mathcal{E}(x)}<1-\prod_{j\geq 1} \left(1-\frac{1}{p^{2j-1}}\right)+(\zeta(p)-1)+\zeta(10)C p^{\frac{-1}{2}} \log p \left( \log \log p\right)^2.
\]
In particular, the Cohen--Lenstra type heuristics indicate that 
\[
\limsup_{p\rightarrow\infty}\bar{\mathfrak{d}}(\mathcal{E}^F_p)=0.
\]
This leads to the realization that $\chi_t(G, E,p)=1$ is the \emph{generic case}.
In other words, one expects that \emph{most of the time} $p\nmid \chi_t(G, E,p)$.
Whereas, it \emph{rarely} happens that for a rank-zero elliptic curve over $\Q$, $p|\chi_t(G, E,p)$.
In other words, Cohen--Lenstra heuristics for the Tate--Shafarevich group indicates that \emph{most of the time}, the Akashi series $\Ak_{E/\mathcal{F}_\infty}$ is unit in $\Lambda(\Gamma_F)= \Z_p\llbracket T\rrbracket$.
Therefore, the Selmer group $\Sel_{p^\infty}(E/\cF_\infty)$ is $\Lambda(G)$-pseudonull.
It follows from Proposition~\ref{prop:Greenberg} that $\Sel_{p^\infty}(E/\cF_\infty)=0$.

We now discuss some implications on the growth of the Mordell--Weil rank in $\Zp^2$-extensions.
As $E_{/\Q}$ varies over all elliptic curves, it follows from Theorem~\ref{Thm: Hung-Lim Theorem} and \eqref{LambdaH rank formula} that
\[
\rank_{\Z}E(F_n) \leq \lambda_p(E/F_{\cyc})p^{n} + 2\corank_{\Zp}E(\Fin)(p).
\]
A question of interest is to find the proportion of elliptic curves for which the Mordell--Weil rank remains bounded in the $\Zp^2$-extension.
Under standard hypothesis and the Cohen--Lenstra heuristic for divisibility of the order of the Tate--Shafarevich by $p$, we have seen that as $p\rightarrow \infty$, the proportion of elliptic curves such that $\lambda_p(E/F_{\cyc})\neq r_E$ approaches 0.

Let $A_{/\Q}$ be a fixed non-CM rank 0 elliptic curve of conductor $N_A$.
Further, suppose that it has rank 0 over $F$.
As $s$ varies over all integers coprime to $N_A$, there exists an elliptic curve ${E_s}_{/\Q}$ with conductor $s^2 N_A$.
In fact, $E_s$ can be realized as a quadratic twist of the elliptic curve $A_{/\Q}$ and has additive reduction at $s$.
For each such elliptic curve $E_s$ as well, we have
\[
\rank_{\Z}E_s(F_n) \leq \lambda_p(E_s/F_{\cyc}) p^{n}.
\]
If further ${E_s}_{/F}$ is also of rank 0, then in the \emph{generic case} one expects $\lambda_p(E_s/F_{\cyc}) =0$.
In particular, the Mordell--Weil rank of $E_s$ remains 0 at each finite layer of the $\Zp^2$-extension.

We now illustrate this with an example.
Let $A$ be the elliptic curve \href{https://www.lmfdb.org/EllipticCurve/Q/11/a/1}{11a2} (Cremona label).
Fix $p=7$, and suppose that $s=5$.
Then, the elliptic curve $E_5 = \href{https://www.lmfdb.org/EllipticCurve/Q/275/b/1}{275b3}$ is a rank 0 elliptic curve over $\Q$ with additive reduction at 5.
When we consider its twist by $-3$, we get the curve ${E_{5}^{(-3)}}_{/\Q}$ which is \href{https://www.lmfdb.org/EllipticCurve/Q/2475/a/1}{2475h3} (Cremona label).
Therefore, one can check that
\[
\rank_{\Z}E_5(F) = \rank_{\Z}E_5(\Q) + \rank_{\Z}{E_{5}^{(-3)}}(\Q) = 0 + 0 =0.
\]
Since ${E_{5}}_{/\Q}$ has no 7-torsion and $\#\Sha(E_5/\Q)[7^{\infty}] =1$, it follows from Remark~\ref{BKLOS remark} that
\[
\# \Sha({E_{5}^{(-3)}}/\Q)[7^\infty] = \Sha(E_5/F)[7^\infty] =1.
\]
The prime 7 is \emph{not} an anomalous prime for $E_{5}$ which can be checked from the $q$-expansion.
Finally, since the Kodaira symbol at the prime 11 is of type $\textrm{I}_1$, we know that upon base-change to $F$, the Tamagawa number does not become divisible by 7.
Therefore,
\[
\chi(\Gamma_F, E_5) =1.
\]
Equivalently, $\mu_7(E_5/F_{\cyc}) = \lambda_7(E_5/F_{\cyc})=0$.
In particular, the Mordell--Weil rank of $E_5$ remains 0 at each finite layer of the $\Zp^2$-extension of $\Q(\sqrt{-3})$.

\subsection{}
\label{section 9-2}
Let $p\geq 5$ be a fixed prime and $F=\Q(\mu_p)$.
For simplicity, fix $m$ to be a prime number (say $\ell$) and suppose that $\ell\equiv 1 \mod{p}$.
Note that $\ell$ splits completely in $F$.
Let $\Fin = \Q\left( \mu_{p^\infty}, \ \ell^{\frac{1}{p^n}} : n = 1, 2, \ldots \right)$ be the false Tate curve extension.
For the pro-$p$ extension $G=\Gal(\Fin/F)$,
\[
\chi_t(G, E,p)=\chi_t(\Gamma_F, E,p)\times \prod_{v\in \fM_E} \absolute{L_v(E, 1)}_p,\\
\]
where $\fM_E = \mathcal{P}_1(E,\Fin) \cup \mathcal{P}_2(E,\Fin)$ is a set of primes in $F$ that lie above $\ell$.
The Euler characteristic $\chi_t(G, E,p)=1$ when \emph{all} of the following conditions hold.
\begin{enumerate}[(a)]
 \item $\mathcal{R}_p(E/F)$ is a unit in $\Z_p$.
 \item $\Sha(E/F)[p^\infty] =0$.
 \item $p$ is \emph{not} an anomalous prime for $E$.
 \item $p\nmid \tau_{\ell}^{(F)}$; equivalently $p\nmid \tau_{\ell}^{(\Q)}$.
 \item At $\ell$, the elliptic curve has additive reduction \emph{or} non-split multiplicative reduction \emph{or} good reduction with $p\nmid \#\widetilde{E}(\F_{\ell})$.
 (Note that since $\ell$ splits in $F$, the reduction type does not change from $\ell$ to a prime $v|\ell$ of $F$.)
\end{enumerate}
Note that the last two conditions \emph{are not} independent.
The condition on $\ell$ not being a prime of split multiplicative reduction automatically implies that $p$ does not divide $\tau_{\ell}^{(\Q)}=\tau_{\ell}^{(F)}$.

The proportion of Weierstrass equations (ordered by height) over $\Q$ with multiplicative reduction was recorded in Lemma~\ref{lemma: reduction type proportions}(1).
Only \emph{half} of these have split multiplicative reduction type (see \cite[Theorem~5.1(1)]{CS20}).
By Remark~\ref{remark for minimal}, the proportion of Weierstrass equations which are globally minimal and have split multiplicative reduction at $\ell$ is given by $\frac{(\ell-1)}{2\ell^{2}}$.
Having assumed that $\ell \equiv 1 \pmod{p}$, the condition of $p\nmid \#\widetilde{E}(\F_\ell)$ is equivalent to $a_\ell(E) \not\equiv 2 \pmod{p}$.
The same argument as in Lemma~\ref{getting hold of frakSp} proves that the number of isomorphism classes of elliptic curves over $\F_{\ell}$ such that $p|\#E(\F_{\ell})$ is given by
\begin{equation}
\label{schoof and lenstra}
\sum_{j=-\lfloor \frac{2\sqrt{\ell}}{p}\rfloor}^{\lfloor \frac{2\sqrt{\ell}}{p}\rfloor}N(2+p j )=\sum_{j=-\lfloor \frac{2\sqrt{\ell}}{p}\rfloor}^{\lfloor \frac{2\sqrt{\ell}}{p}\rfloor} H((2+p j)^2-4\ell)\leq C\left(\frac{4\sqrt{\ell}}{p}+1\right)\sqrt{\ell}\log \ell \left( \log \log \ell\right)^2
\end{equation}
for some effectively computable constant $C$.

Let $\mathcal{E}_p^F$ be the set of elliptic curves over $\Q$ such that the following conditions are satisfied
\begin{enumerate}[(a)]
 \item $E$ has good \emph{ordinary} reduction at $p$, 
 \item $E$ has good reduction at $\ell=2,3$, 
 \item $\rank_{\Z}E(F)=0$, 
 \item $\chi(G, E,p)\neq 1$.
\end{enumerate}
Let $\mathcal{E}_p^{(1)}$ be the set of all elliptic curves over $\Q$ such that $p$ divides the order of $\#\Sha(E/F)$.
Since $p$ and $F$ are fixed throughout and in view of the heuristics mentioned in \eqref{del heuristics}, it makes sense to assume that 
\[
\overline{\mathfrak{d}}(\mathcal{E}_p^{(1)})=1-\prod_{i\geq 1} \left(1-\frac{1}{p^{2i-1}}\right).
\]
We have the following result which follows immediately from the previous discussion.

\begin{theorem}
\label{FTC thm varying elliptic curve}
With notation as above, there is an effective constant $C>0$ such that 
\[
\begin{split}
\bar{\mathfrak{d}}(\mathcal{E}_p^F)< & \ \bar{\mathfrak{d}}(\mathcal{E}_p^{(1)})+(\zeta(p)-1) +\frac{(\ell-1)}{2\ell^2}\\
+ & \zeta(10)C\left(\frac{4\sqrt{\ell}}{p}+1\right)\sqrt{\ell}\log \ell \left( \log \log \ell\right)^2.
\end{split}
\]
\end{theorem}

\begin{proof}
It follows from the definition of upper density that
\[
\bar{\mathfrak{d}}(\mathcal{E}_p^F)<  \bar{\mathfrak{d}}(\mathcal{E}_p^{(1)})+\bar{\mathfrak{d}}(\mathcal{E}_p^{(2)}) +\bar{\mathfrak{d}}(\mathcal{E}_p^{(3)}),
\]
where $\mathcal{E}_p^{(2)}$ (resp. $\mathcal{E}_p^{(3)}$) is the set of all elliptic curves over $\Q$ such that $p$ divides $\tau_{\ell}^{(F)}$ (resp. $\# \tilde{E}(\mathbb{F}_{\ell})$).
From earlier discussion in this section, $p|\tau_{\ell}^{(F)}$ if and only if $p|\tau_{\ell}^{(\Q)}$.
Therefore, the estimate for the upper density of $\mathcal{E}_p^{(2)}$ is $\zeta(p)-1$, just as was calculated previously.
Finally, the estimate for $\bar{\mathfrak{d}}(\mathcal{E}_p^{(3)})$ follows from \eqref{schoof and lenstra}.
This completes the proof of the theorem.
\end{proof}

The above theorem counts the proportion of elliptic curves with non-trivial $G$-Euler characteristic.
In view of Propositions~\ref{prop: pseudonull} and \ref{proposition: akasi series to pseudonull}, we can count how often the Selmer group is \emph{not} pseudonull.
Rephrased in terms of the $\Lambda(H)$-rank, the above theorem answers the question how often is
\[
\rank_{\Lambda(H)} \mathcal{X}(E/\Fin) > 0?
\]
Next, we discuss a finer question pertaining to the $\Lambda(H)$-rank of $\mathcal{X}(E/\Fin)_f$.

\begin{corollary}
\label{lambda H rank cor for varying E in FTC}
Let $p>3$ be a fixed prime and fix another prime $\ell\equiv 1 \pmod{p}$.
Consider the false Tate curve extension $\Fin = \Fin^{(\ell)}$.
Varying over all elliptic curves (over $\Q$) with good reduction at $\ell$ and good \emph{ordinary} reduction at $p$, the \emph{upper density} of elliptic curves with 
\[
\rank_{\Lambda(H)}\cX(E/\Fin)_f > \lambda_p(E/F_\cyc)\geq 0
\]
is \emph{at most}
\[
\frac{(\ell -1)}{2\ell^2} + \zeta(10)C \left( \frac{4\sqrt{\ell}}{p} + 1\right)\sqrt{\ell}\log \ell (\log \log \ell)^2,
\]
where $C$ is an effective positive constant.
\end{corollary}

\begin{proof}
In this case, note that \eqref{LambdaH rank formula} gives
\[
\rank_{\Lambda(H)}\cX(E/\Fin)_f=\lambda_p(E/F_\cyc)+\sum_{\substack{v|\ell, \\ \ell \textrm{ split multiplicative}}}1 + \sum_{ \widetilde E(\mathbb{F}_{\ell})[p]\neq 0} 2.
\]
Observe that for any elliptic curve with 
\[
\rank_{\Lambda(H)}\cX(E/\Fin)_f > \lambda_p(E/F_\cyc)\geq 0,
\]
either of the following properties must hold:
\begin{enumerate}[(a)]
    \item $E$ has split multiplicative reduction at $\ell$, \emph{or}
    \item $p| \widetilde{E}(\mathbb{F}_{\ell})$.
\end{enumerate}
The result now follows from our previous calculations.
\end{proof}

We now prove an alternative result for the $\mathcal{G}$-Euler characteristic, where $\mathcal{G}:=\Gal(\cF_\infty/\Q)$.
Since this extension is \emph{not} pro-$p$, even if the associated Akashi series over $\mathcal{G}$ is a unit, we cannot deduce that the Selmer group $\Sel_{p^\infty}(E/\cF_\infty)$ is pseudonull as a $\Lambda(\mathcal{G})$-module.
In this setting, $m$ is a $p$-power free natural number.
We have 
\[
\chi_t(\mathcal{G}, E,p)=\chi_t(\Gamma_{\Q}, E,p)\times \prod_{v\in \fM_E} \absolute{L_v(E, 1)}_p.
\]
Now, $\fM_E=\mathcal{P}_1 \cup \mathcal{P}_2$ is a set of primes in $\Q$.
How often $\chi_t(\mathcal{G}, E,p)=1$ requires studying for what proportion of elliptic curves do the following properties hold \emph{simultaneously}.
\begin{enumerate}[(a)]
 \item The normalized $p$-adic regulator (over $\Q$) is a $p$-adic unit.
 \item $\Sha(E/\Q)[p] =0$.
 \item $p$ is \emph{not} an anomalous prime for $E$.
 \item $p\nmid \prod_{\ell\neq p} c_\ell(E/\Q)$.
 \item At all primes $\ell| m$, the elliptic curve has either additive reduction \emph{or} non-split multiplicative reduction \emph{or} good reduction with $p\nmid \#\widetilde{E}(\F_{\ell})$.
\end{enumerate}
Let $\mathcal{E}'_p$ be the set of rank 0 elliptic curves with good reduction at $\ell=2,3$, good \emph{ordinary} reduction at $p$, and $\chi(\mathcal{G}, E,p)\neq 1$.
The above discussion can be summarized as below.

\begin{theorem}
\label{FTC vary EC}
Assume Delaunay's heuristic for the Tate--Shafarevich group.
Then, there are effective constants $C_1,C_2>0$ for which
\[
\begin{split}
\bar{\mathfrak{d}}(\mathcal{E}'_p)< & \ 1-\prod_{j\geq 1} \left(1-\frac{1}{p^{2j-1}}\right)+(\zeta(p)-1)\\
& +\zeta(10)C_1 p^{\frac{-1}{2}} \log p \left( \log \log p\right)^2+ \sum_{\ell| m} \frac{(\ell-1)}{2\ell^2}\\
& + \sum_{\ell| m} \zeta(10)C_2\left(\frac{4\sqrt{\ell}}{p}+1\right)\sqrt{\ell}\log \ell \left( \log \log \ell\right)^2.
\end{split}
\]
\end{theorem}

\begin{proof}[Sketch of the proof]
As in the proof of Theorem~\ref{FTC thm varying elliptic curve}, the upper bound for $\bar{\mathfrak{d}}(\mathcal{E}'_p)$ is evaluated by obtaining estimates for the proportion of (rank 0) elliptic curves over $\Q$ ordered by height for which at least one of the properties (b)--(e) is \emph{not} satisfied.
The estimate for each of these quantities is calculated as before.
\end{proof}

\begin{remark}
Assuming Delaunay's heuristics, we find that as $p\rightarrow \infty$,
\[
\limsup_{p\rightarrow \infty}\left(\bar{\mathfrak{d}}(\mathcal{E}'_p)\right)<\sum_{\ell| m} \frac{(\ell-1)}{2\ell^2}+\sum_{\ell| m} \zeta(10)C_2\sqrt{\ell}\log \ell \left( \log \log \ell\right)^2.
\]
\end{remark}

\subsection{}
\label{section 9-3}
Let $p\geq 5$ be a fixed prime number and ${E_0}_{/\Q}$ a fixed non-CM elliptic curve with good reduction at $p$.
This elliptic curve determines the extension $\cF_{\infty} = \Q(E_0[p^\infty])$.
Write $\mathcal{G}_{E_0} = \Gal(\Fin/\Q)$.
Then,
\[
\chi_t(\mathcal{G}_{E_0}, E,p)=\chi_t(\Gamma_{\Q}, E,p)\times \prod_{v\in \fM_{E_0}} \absolute{L_v(E, 1)}_p,
\]
where $\fM_{E_0}$ is the finite set of primes of multiplicative or potentially multiplicative reduction of $E_0$.
Note that $\fM_{E_0}$ is a fixed set.

We would like to calculate the proportion of elliptic curves $E_{/\Q}$ with good \emph{ordinary} reduction at $p$, ordered by height, for which $\chi_t(\mathcal{G}_{E_0}, E,p)$ is a $p$-adic unit.
Since the $p$-adic Lie extension of interest is \emph{not} pro-$p$, even if the associated Euler characteristic is a $p$-adic unit, we will not be able to deduce the pseudonullity of  $\Sel_{p^\infty}(E/\cF_\infty)$  as a $\Lambda(\mathcal{G}_{E_0})$-module.
As in the previous case, we need to find the proportion of elliptic curves for which the following properties hold \emph{simultaneously}.
\begin{enumerate}[(A)]
 \item The normalized $p$-adic regulator (over $\Q$) is a $p$-adic unit.
 \item $\Sha(E/\Q)[p] =0$.
 \item $p$ is \emph{not} an anomalous prime for $E$.
 \item $p\nmid \prod_{\ell\neq p} c_{\ell}(E/\Q)$.
 \item At each $\ell\in \fM_{E_{0}}$, the elliptic curve has one of the following reduction types
 \begin{enumerate}
 \item additive reduction \emph{or}
 \item split multiplicative reduction if $\ell \not\equiv 1 \pmod{p}$ \emph{or}
 \item non-split multiplicative reduction if $\ell \not\equiv -1 \pmod{p}$ \emph{or}
 \item good reduction with $p\nmid \#\widetilde{E}(\F_{\ell})$.
 \end{enumerate}
\end{enumerate}
If the elliptic curves are stipulated to have rank 0, then the $p$-adic regulator is a unit in $\Z_p$.
Let $\mathcal{E}'_p$ be the set of rational elliptic curves with $\rank_{\Z}E(\Q)=0$, good reduction at $\ell=2,3$, good \emph{ordinary} reduction at $p$, and $\chi(\mathcal{G}_{E_0}, E,p)\neq 1$.
The next result can be proven in the same way as Theorem~\ref{FTC vary EC}.
The only difference in the formula stems from how often the local Euler factor is \emph{not} a $p$-adic unit.

\begin{theorem}
\label{triv case: vary E}
Assume Delaunay's heuristic for the Tate--Shafarevich group.
Then, there exist effective constants $C_1, C_2 >0$ for which
\[
\begin{split}
\bar{\mathfrak{d}}(\mathcal{E}'_p)< & \ 1-\prod_{j\geq 1} \left(1-\frac{1}{p^{2j-1}}\right)+(\zeta(p)-1)\\
& +\zeta(10)C_1 p^{\frac{-1}{2}} \log p \left( \log \log p\right)^2 \\ 
& + \sum_{\substack{\ell\in \mathfrak{M}_{E_0},\\ \ell \equiv 1 \pmod{p}}} \frac{(\ell-1)}{2\ell^2} + \sum_{\substack{\ell\in \mathfrak{M}_{E_0},\\ \ell \equiv -1 \pmod{p}}} \frac{(\ell-1)}{2\ell^2} \\
&+ \sum_{\ell\in \mathfrak{M}_{E_0}} \zeta(10)C_2\left(\frac{4\sqrt{\ell}}{p}+1\right)\sqrt{\ell}\log \ell \left( \log \log \ell\right)^2.
\end{split}
\]
\end{theorem}

Shifting focus, we record our observations regarding the $\Lambda(H)$-rank of the Selmer group when $H$ \emph{is} a pro-$p$ group.
Fix a non-CM elliptic curve ${E_0}_{/\Q}$ with good reduction at $p\geq 5$, such that $F=\Q(E_0[p])=\Q(\mu_p)$ and $\Fin = \Q(E_0[p^\infty])$.
The set of primes of potentially multiplicative reduction is determined explicitly.
Call this set $\mathcal{PM}_0$.
For simplicity, suppose that primes in $\mathcal{PM}_0$ \emph{split completely} in $F$.
Such examples exist.
Choose $E_0$ to be the elliptic curve with Cremona label \href{https://www.lmfdb.org/EllipticCurve/Q/11a1/}{11a1} and set $p=5$.
In this case, $F=\Q(\mu_5)$ and $\Fin/F$ is a pro-$p$ extension.
Also, since $11\equiv 1 \pmod{5}$, it is clear that 11 splits in $F$.

\begin{proposition}
Suppose that $E$ varies over all globally minimal Weierstrass equations with good \emph{ordinary} reduction at $p$ and good reduction at $2,3$.
The proportion of such elliptic curves with the additional property that 
\[
\rank_{\Lambda(H)}\mathcal{X}(E/\Fin)_f > \lambda_p(E/F_{\cyc})
\]
has an \emph{upper density}
\[
\sum_{\ell\in \mathcal{PM}_0} \frac{\ell-1}{2\ell^2} + \sum_{ \ell\in \mathcal{PM}_0} \zeta(10)C\left( \frac{4\sqrt{\ell}}{p}+1\right)\sqrt{\ell}\log{\ell}(\log \log \ell)^2,
\]
where $C>0$ is an effectively computable constant.
\end{proposition}

\begin{proof}
In this case, \eqref{LambdaH rank formula} is precisely
\[
\rank_{\Lambda(H)}\mathcal{X}(E/{\Fin})_f = \lambda_p(E/F_{\cyc}) + \sum_{ \ell\in \mathcal{PM}_0} \corank_{\Zp} Z_{v}(F_{\cyc,w}). 
\]
The proportion of elliptic curves $E$ such that
\[
\rank_{\Lambda(H)}\mathcal{X}(E/\Fin)_f > \lambda_p(E/F_{\cyc}),
\]
must satisfy the property that for some $\ell \in \mathcal{PM}_0$, the elliptic curve has \emph{either}
\begin{enumerate}[(a)]
 \item split multiplicative reduction \emph{or} 
 \item good reduction with $E(\mathbb{F}_\ell)[p]\neq0$.
\end{enumerate}
The claim follows from calculations identical to earlier results.
\end{proof}

\section{Results for a fixed $E_{/\Q}$ as $p$ varies}
\label{S: Vary prime}

Recall from previous sections that $G=\op{Gal}(\cF_\infty/F)$ and $\mathcal{G}=\op{Gal}(\cF_\infty/ \Q)$.
Thus, $G$ is pro-$p$, which $\mathcal{G}$ is not (unless $F/\Q$ is a $p$-extension).
Given an elliptic curve $E_{/\Q}$, the question of interest is the following:
as $p$ varies over all primes of good \emph{ordinary} reduction of $E$, for what proportion of primes is $\chi_t(G, E,p)$ (or $\chi_t(\mathcal{G}, E,p)$) a $p$-adic unit.
Note that as $p$ varies, so does the extension $\cF_\infty$, as we shall see below.

In \S\ref{section 10-1}, we fix a rank 0 elliptic curve $E_{/\Q}$ and an imaginary quadratic field $F$.
The goal is to study the variation of the Mordell--Weil rank at each finite layer of the $\Zp^2$-extension as $p$ varies.
In \S\ref{false tate gamma-F EC}, we fix a square-free integer $m$ and consider the false Tate curve extension $\Fin = \Q(\mu_{p^\infty}, m^{\frac{1}{p^\infty}})$ as $p$ varies.
The analysis is carried out in three steps depending on the reduction type of the fixed elliptic curve $E$ at a prime $\ell|m$.
The results are summarized in Theorem~\ref{theorem for vary p when FTC}.
In \S\ref{section 10-3}, we fix a pair of elliptic curves $(E,E')$ both defined over $\Q$ and such that $E'$ is non-CM.
We study the variation of truncated $\mathcal{G}_{E'}$-Euler characteristic and that of the $\Lambda(H)$-corank of the Selmer group, as $p$ varies.

\subsection{}
\label{section 10-1}
Let $F=\Q(\sqrt{-d})$ be a fixed imaginary quadratic field and we consider the unique $\Zp^2$-extension $\cF_{\infty}/F$.
In this case, $\fM= \emptyset$.

\begin{proposition}
\label{prop: vary p Zp2}
Let $E_{/\Q}$ be a fixed elliptic curve of Mordell--Weil rank 0 without complex multiplication.
Suppose that it remains of rank 0 upon base-change to $F$.
Further, assume that $\#\Sha(E/F)$ is finite.
Then, $\chi(G, E,p)=1$ for all primes $p$ at which $E$ has good \emph{ordinary} primes outside a set of density zero.
\end{proposition}

\begin{proof}
For the extensions $\Q_{\op{cyc}}/\Q$ (as $p$ varies), the result was proven by Greenberg \cite[Proposition~5.1]{Greenberg}.
For cyclotomic $\Z_p$-extensions of a general number field, the statement follows from a result of V.~K.~Murty \cite{Mur97}, and this was observed by the first named author in her thesis, see \cite[Theorem~5.1.1]{DK_thesis} for details.

The result for $\Z_p^2$-extensions $\cF_\infty/F$ of an imaginary quadratic field $F$ follows from the same argument as aforementioned results.
Consider the formula for the Euler characteristic
\[
\chi(G, E,p) = \chi(\Gamma_F, E,p) \sim \frac{\#\Sha(E/F)[p^\infty]\cdot \prod_{v\nmid p}c_v^{(p)}(E/F)}{\left(\# E(F)[p^\infty] \right)^2} \cdot \prod_{v|p}\left(\# \widetilde{E}(\kappa_v)[p^\infty] \right)^2.
\]
It follows from our assumptions on $E$ that the above Euler characteristic is defined.
Clearly all terms in the above formula are $p$-adic units for all but finitely many primes $p$, except possibly the term $\prod_{v|p}\left(\# \widetilde{E}(\kappa_v)[p^\infty] \right)$.
The result of \cite{Mur97} states that $p\nmid \prod_{v|p}\left(\# \widetilde{E}(\kappa_v)[p^\infty] \right)$ for all primes outside a set of density zero.
This sparse set of primes at which $p$ divides $\prod_{v|p}\left(\# \widetilde{E}(\kappa_v)[p^\infty] \right)$ is the set of \emph{anomalous primes}.
In conclusion, $\chi(G, E,p)=1$ for all primes $p$ outside a density zero set of primes.
\end{proof}

\begin{corollary}
\label{MW rank corollary, vary p Zp2}
With the same setting as Proposition~\ref{prop: vary p Zp2}, for all good \emph{ordinary} primes outside a set of density zero, $\rank_{\Z}E(F_n)=0$ at each finite layer of the $\Zp^2$-extension.
\end{corollary}

\begin{proof}
Combining Proposition~\ref{prop: vary p Zp2} and Remark~\ref{rmk: finiteness of p-primary torsion points over certain extensions}, we know that for density 1 good \emph{ordinary} primes the following inequality holds,
\[
\rank_{\Z}E(F_n) \leq \lambda_p(E/F_{\cyc}) \cdot p^{n} =0.
\]
In particular, the Mordell--Weil rank of $E(F_n)=0$ in each finite layer of the $\Zp^2$-extension for all good \emph{ordinary} primes $p$ outside a subset of density 0.
\end{proof}

\subsection{}
\label{false tate gamma-F EC}
We fix an elliptic curve $E_{/\Q}$ and a square-free integer $m$.
We  vary $p$ over all primes where $E$ has good \emph{ordinary} reduction.
Let $F=\Q(\mu_p)$ and consider the false Tate curve extension
\[\cF_{\infty}:=\Q(\mu_{p^\infty}, \ m^{\frac{1}{p^\infty}}),\]which varies as $p$ varies ($m$ fixed).
As noted in \eqref{set M for FTC}, the primes in the set $\fM$ divide $m$.
We begin the discussion with a few basic remarks.
\newline

\noindent \emph{Case 1:} Suppose that $\ell$ is a prime divisor of $m$ and $v|\ell$ is a prime of $F$.
If $E$ has either \emph{non-split multiplicative reduction} or \emph{additive reduction} of $E$, then $v\not\in \fM$.
It follows from Theorem~\ref{EC formula theorem} that such primes do not contribute to the $G$-Euler characteristic.
More precisely, if $E$ has non-split multiplicative reduction or additive reduction at \emph{all} the primes dividing $m$ then for all primes $p$,
\[
\chi(G, E,p) = \chi(\Gamma_F, E,p).
\]
In this case, \eqref{lambda-H formula for FTC} simplifies considerably and becomes
\[
\rank_{\Lambda(H)}\mathcal{X}(E/\Fin)_f = \lambda_p(E/F_{\cyc}).
\]
Now, Theorem~\ref{Thm: Hung-Lim Theorem} asserts that
\[
\rank_{\Z}(E/F_n)\leq \lambda_p(E/F_{\cyc})p^n.
\]

\noindent \emph{Case 2:} When $\ell$ is a prime of \emph{split multiplicative reduction} we have the following:
\begin{proposition}\label{prop:FT-EC}
Consider the false Tate curve extension $\Fin = \Fin^{(m)}$.
Let $E_{/\Q}$ be a fixed elliptic curve with \emph{split multiplicative reduction} at a prime dividing $m$.
As $p$ varies over all primes of good \emph{ordinary} reduction of $E$, the $G$-Euler characteristic is \emph{always} non-trivial.
\end{proposition}

\begin{proof}
Suppose that $\ell$ is a prime divisor of $m$ and $v|\ell$ is a prime of split multiplicative reduction of $E$.
Since $\ell\neq p$, we have that $v\in\fM$.
Recall that Lemma~\ref{criterion for p to divide Lv(E,1)}(2) asserts that 
\[
\absolute{L_v(E,1)}_p \neq 1 \textrm{ if and only if } q_v \equiv 1 \pmod{p}.
\]
We have $q_v = \absolute{\mathbb{F}_{\ell^f}}$ where $f=f_p$ is the \emph{degree of inertia} of $\ell$ in $F=\Q(\mu_p)$.
Recall from our discussion in \S\ref{sec:8.2} that $f$ is the smallest positive integer such that $\ell^{f} \equiv 1\pmod{p}$.
Therefore, we conclude that varying over all primes $p$, the $G$-Euler characteristic is \emph{always} non-trivial for an elliptic curve $E_{/\Q}$ with split multiplicative reduction at $\ell|m$.
\end{proof}

\begin{remark}
The growth of Mordell--Weil ranks inside $\Fin$ has been studied using Heegner points by H.~Darmon and Y.~Tian in \cite{darmontian} under certain hypotheses (see also \cite{DL2}).
Their results tell us that the Mordell--Weil ranks are expected to be unbounded inside $\Fin$.
\end{remark}

\noindent \emph{Case 3:}
If $\ell$ is a prime divisor of $m$ and $v|\ell$ is a prime of \emph{good reduction} of $E$, we need to analyse two cases.
If $\ell=p$, then $v\not\in \fM$ and there is no contribution to the $G$-Euler characteristic from the local Euler factor.
Otherwise, by definition,
\[
v\in\fM \textrm{ if and only if } E(F_v)[p^\infty]\neq 0.
\]
We would like to evaluate for what proportion of primes $p$ is $E(F_v)[p]\neq 0$ for a fixed $\ell$.
Equivalently (see \cite[Proposition~VII.3.1]{silverman2009}), how often does $p$ divide $\widetilde{E}(\kappa_{v})$?
Recall that 
\[ \#\widetilde{E}(\kappa_v) = \ell^f + 1 - a_v \equiv 2 - a_v \pmod{p},\]
where $f$ is the smallest positive integer such that $\ell^f\equiv1\pmod p$ as in the proof of Proposition~\ref{prop:FT-EC}.
Thus, the condition $\absolute{L_v(E,1)}_p \neq 1$ holds precisely when $a_v \equiv 2 \pmod{p}$.

\begin{proposition}
\label{prop: proportion p varies FTC}
Let $E_{/\Q}$ be a non-CM elliptic curve with good reduction at $\ell\geq 5$.
As $p$ varies over all primes (distinct from $\ell$), consider the number field $F=\Q(\mu_p)$ and the base-change of $E$ to $F$.
Then, for \emph{at least half} of the primes $p$ where $E$ has good \emph{ordinary} reduction, we have
\[a_v \not\equiv 2 \pmod{p}.\]
\end{proposition}

\begin{proof}
Since $E$ is assumed to be non-CM, it has been proved by Serre \cite{serre68} that the set of primes where $E$ has good \emph{supersingular} reduction has density zero.
Therefore, it is enough to consider  all primes $p$, not just those where $E$ has good \emph{ordinary} reduction.
Since $E_{/\Q}$ and $\ell$ are fixed, the value of $a_{\ell}$ is determined precisely.
Let $\alpha$ and $\beta$ be the roots of $X^2-a_\ell X+\ell$ in $\overline{\Q}$ and let $\bar\alpha$ and $\bar\beta$ be the roots of $X^2-a_\ell X+ \ell$ in $\overline{\Fp}$.
Then $a_v=\alpha^f+\beta^f\equiv 2\mod p$ if and only if $\bar\alpha^f=1$ since $(\alpha\beta)^f=\ell^f\equiv 1\mod p$.

When $p\neq \ell$, the constant term of the polynomial $X^2-a_\ell X+\ell$ is not zero modulo $p$.
In particular, this tells us that $\bar\alpha\ne 0$.
Thus,  if $\bar\alpha\notin\Fp$, then $\bar\alpha^{p-1}\neq 1$.
As  $f|(p-1)$ by definition, we have furthermore $\bar{\alpha}^f\ne 1$.

Note that $\bar\alpha\notin \Fp$ if and only if the polynomial $X^2-a_\ell X+\ell$ is irreducible over $\Fp$.
When $p$ is odd, this in turn is equivalent to the discriminant $a_\ell^2-4\ell$  not being a square modulo $p$.
By the Hasse--Weil bound, $a_\ell^2-4\ell<0$, so it is not a square in $\Q$.
Therefore, Chebotarev density theorem tells us that for exactly half of the primes $p$,  $a_\ell^2-4\ell$ is not a square modulo $p$. Therefore, the result follows.
\end{proof}

\begin{proposition}
\label{FTC: G EC formula same as Gamma EC formula}
Let $m$ be a fixed prime number and consider the false Tate curve extension $\Fin = \Fin^{(m)}$.
Let $E_{/\Q}$ be a fixed non-CM elliptic curve with good reduction at $m$.
As $p$ varies over all primes of good \emph{ordinary} reduction of $E$, for \emph{at least half} of the primes,
\[
\chi_t(G, E,p) = \chi_t(\Gamma_F, E,p).
\]
\end{proposition}

\begin{proof}
The proof is immediate from Theorem~\ref{EC formula theorem} and Proposition~\ref{prop: proportion p varies FTC}.
\end{proof}

We consider the following special case.
\begin{proposition}
Suppose that $m=\ell$ is a prime number and that  $E$ is a non-CM curve with good  \emph{supersingular} reduction at $\ell$ with $a_\ell(E)=0$.
Let $v$ denote a prime above $\ell$ in $F$.
For exactly  \emph{two-third} of the primes $p$, we have $a_v\not\equiv 2\pmod{p}$.
\end{proposition}

\begin{proof}
As in the proof of Proposition~\ref{prop: proportion p varies FTC}, it is enough to consider all primes $p$.
We have (see \cite[Exercise~5.15]{silverman2009})
\[
\#\widetilde{E}(\kappa_v) = \begin{cases}
\ell^f + 1 & \textrm{ if } f \textrm{ is odd,}\\
\left(\ell^{f/2} - (-1)^{f/2}\right)^2 & \textrm{ if } f \textrm{ is even.}
\end{cases}
\]
In particular, $p| \# \widetilde{E}(\kappa_v)$ if and only if $f \equiv 2 \pmod{4}$.
The latter condition is satisfied precisely \emph{one-third} of the time by \cite[Theorem~1.1 with $l =2$]{CM03}.
Hence the result.
\end{proof}

Similar to Proposition~\ref{FTC: G EC formula same as Gamma EC formula}, this implies the following.

\begin{proposition}\label{prop:FT-ss}
Let $m$ be a fixed prime number and consider the false Tate curve extension $\Fin = \Fin^{(m)}$.
Let $E_{/\Q}$ be a fixed non-CM elliptic curve with good \emph{supersingular} reduction at $m$ and $a_m(E)=0$.
As $p$ varies over all primes of good reduction of $E$, for exactly \emph{two-third} of the primes,
\[
\chi_t(G, E,p) = \chi_t(\Gamma_F, E,p).
\]
\end{proposition}

\begin{remark}
\label{Gamma-Q EC for false tate}
However, when studying the $\mathcal{G}^{(m)}$-Euler characteristic,  $\fM$ is a (finite) set of rational primes.
If $\ell|m$ is a prime of good reduction, 
\[
\absolute{L_v(E,1)}_p \neq 1 \textrm{ if and only if }\# \widetilde{E}(\mathbb{F}_\ell) = q_v + 1 -a_{\ell} = \ell+1 - a_{\ell} \equiv 0 \pmod{p}.
\]
As $p$ varies, we count how often $a_{\ell} \equiv \ell + 1 \pmod{p}$.
By an application of the Hasse-bound $\absolute{a_\ell}\leq 2\sqrt{\ell}$.
Thus, $\absolute{L_v(E,1)}_p$ is a $p$-adic unit for all but finitely many primes $p$.
\end{remark}

The above discussion can be summarized as follows.
\begin{theorem}
\label{theorem for vary p when FTC}
Let $E_{/\Q}$ be a fixed elliptic curve and $m$ be a fixed positive integer.
Let $F= \Q(\mu_p)$ and $\Fin$ be the false Tate curve extension.
As $p$ varies over all primes of good reduction of $E$, the following assertions hold.
\begin{enumerate}
 \item If $E$ has split multiplicative reduction at $\ell|m$, then the $G$-Euler characteristic is \emph{never} a $p$-adic unit.
 \item If $E$ has non-split multiplicative reduction or additive reduction or good reduction with $E(F_v)[p^\infty]=0$ at all primes $\ell|m$, then 
 \[
 \chi_t(G, E,p) = \chi_t(\Gamma_F, E,p).
 \]
 In particular, if $E$ has Mordell--Weil rank 0 over $\Q$ and $F$ then, $\chi(G, E,p)=1$ as $p$ varies over all good \emph{ordinary} primes outside a set of density 0.
 \item If $m=\ell$ is a fixed \emph{prime number}, and $E$ is non-CM with  good \emph{supersingular} reduction at $\ell$ with $a_\ell(E)=0$, then for two-thirds of the primes $p$
 \[
 \chi_t(G, E,p) = \chi_t(\Gamma_F, E,p).
 \]
 In particular, if $E$ has Mordell--Weil rank 0 over $\Q$ and $F$ then, $\chi(G, E,p)=1$ for all such primes $p$.
 \item If $E$ has good reduction reduction at all primes $\ell|m$ then the $\mathcal{G}$-Euler characteristic is given by
 \[
 \chi_t(\mathcal{G}, E,p) = \chi_t(\Gamma_\Q, E,p).
 \]
 for all but finitely many primes $p$.
\end{enumerate}
\end{theorem}

We conclude this section with some remarks about the $\Lambda(H)$-ranks of $\cX(E/\Fin)_f$.
We know from \eqref{lambda-H formula for FTC} that
\[
\rank_{\Lambda(H)}\cX(E/\Fin)_f=\lambda_p(E/F_\cyc)+ \sum_{\substack{\ell|N_E, \ \ell|m\\ \ell \textrm{ split multiplicative}}}1+ \sum_{ \substack{v|\ell, \ \ell|m,\\ E(F_v)[p]\neq 0}} 2.
\]
Therefore, if the fixed elliptic curve has a prime of split multiplicative reduction at a prime divisor of $m$, then for all primes $p$,
\[
\rank_{\Lambda(H)}\cX(E/\Fin)_f>\lambda_p(E/F_\cyc).
\]
By the same argument as before, one can also deduce the following.
Suppose that $m$ is a prime number such that $E$ has good reduction at $m$, then for \emph{at most half} of the primes $p$,
\[
\rank_{\Lambda(H)}\cX(E/\Fin)_f>\lambda_p(E/F_\cyc).
\]

\subsection{}
\label{section 10-3}
Fix a pair of elliptic curves $(E,E')$, both defined over $\Q$ and suppose that $E'/\Q$ is a non-CM elliptic curve.
Let $F= \Q(E'[p])$ and consider the $p$-adic Lie extension $\Fin = \Q(E'[p^\infty])$.
By the Weil pairing, we know that $F\supseteq \Q(\mu_p)$.

Let $\mathcal{G}_{E'} := \Gal\left(\Q(E'[p^\infty])/\Q\right)$.
We will study for what proportion of all primes is $\chi_t(\mathcal{G}_{E'}, E,p)$ a $p$-adic unit.
For us, $\fM_{E'}$ is a set of primes over $\Q$.
By definition, this set contains precisely the primes of potentially multiplicative reduction of $E'$.

Suppose that $v\in \fM_{E'}$ is a prime of split (resp. non-split) multiplicative reduction of $E$.
By Lemma~\ref{criterion for p to divide Lv(E,1)}(2) (resp. Lemma~\ref{criterion for p to divide Lv(E,1)}(3)) we know that \begin{align*}
 \absolute{L_v(E,1)}_p \neq 1 &\textrm{ if and only if } q_v = \ell \equiv 1 \pmod{p}\\
 (\textrm{resp. } \absolute{L_v(E,1)}_p \neq 1 & \textrm{ if and only if } q_v =\ell \equiv -1 \pmod{p}).
\end{align*}
In either case, $\ell$ is a prime which is independent of $p$.
As $p$ varies over all primes, either of the congruence conditions can be satisfied by \emph{at most} finitely many primes $p$.

If $v\in \fM_{E'}$ is a prime of good reduction of $E$, then Lemma~\ref{criterion for p to divide Lv(E,1)}(1) asserts that
\[
\absolute{L_v(E,1)}_p \neq 1 \textrm{ if and only if } \# \widetilde{E}(\mathbb{F}_\ell) = \ell + 1 - a_{\ell} \equiv 0 \pmod{p}.
\]
As before, we see that $\absolute{L_v(E,1)}_p$ is a $p$-adic unit for all but finitely many primes $p$.

We have the following theorem.
\begin{theorem}
\label{vary p main thm in triv case}
Let $(E, E')$ be a pair of elliptic curves over $\Q$ where $E'$ is a non-CM curve.
For all but finitely many primes $p$ where $E$ has good \emph{ordinary} reduction,
\[
\chi_{t}\left(\mathcal{G}_{E'}, E,p \right) = \chi_{t}\left(\Gamma_{\Q}, E,p \right).
\]
\end{theorem}

The proportion of primes for which $\chi_t(\Gamma_\Q, E,p)=1$ was studied in detail \cite[Section 3]{KR21}.
It was conjectured in \cite[Conjecture~3.17]{KR21} that $\chi_t(\Gamma_\Q, E,p)=1$ should be true for $100\%$ of the primes of good \emph{ordinary} reduction.

To study the $G_{E'}$-Euler characteristic, we consider the base-change curve ${E}_{/F}$.
Since $F\supseteq \Q(\mu_p)$ the arguments of Section \ref{false tate gamma-F EC} imply that $q_v \equiv 1 \pmod{p}$ for all values of $p$.
Indeed, this is because the inertia degree of $\ell$ for the extension $\Q(\mu_p)/\Q$ divides the of inertia degree of $\ell$ for the extension $F/\Q$.

This means, if $\ell$ is a prime of potentially multiplicative reduction of $E'$ which is also a prime of split multiplicative reduction of $E$, then the $G_{E'}$-Euler characteristic of $E$ is \emph{never} trivial.
This is well-known in the special case $E = E = E'$.
It was shown in \cite[Theorem~1.5]{coateshowson} that the $p$-primary Selmer group is infinite dimensional for all $p\geq 5$.

Suppose that we fix the two distinct elliptic curves $E$ and $E'$ (both defined over $\Q$) such that the following properties hold:
\begin{enumerate}[(a)]
    \item $E'$ is a non-CM elliptic curve.
    \item $E$ has additive reduction or non-split multiplicative reduction at the primes above potentially multiplicative reduction of $E'$.
\end{enumerate}
Then, the contribution of the local Euler factors to the $G_{E'}$-Euler characteristic of $E$ is trivial.
Equivalently,
\[
\chi_{t}\left(G_{E'}, E,p\right) = \chi_{t}\left(\Gamma_{F}, E,p \right).
\]

When $F/\Q$ is a number field and $\rank_{\Z}E(F) \geq 1$, the analysis of the Euler characteristic formula is more subtle.
The difficulty arises from our lack of knowledge regarding the normalized $p$-adic regulator and the $p$-part of the Tate--Shafarevich group over number fields.
Also, when $[F:\Q]$ is ``large'', computations are expensive and it is hard to obtain meaningful heuristics.
However, it is still possible to make some brief remarks about the Mordell--Weil rank growth of $E$ in the extension $\Fin/F$.

\begin{proposition}
Suppose $E, E'$ are two fixed elliptic curves chosen as described above with conductors $N_1, N_2$, respectively.
Then, for all primes $p\nmid N_1 N_2$,
\[
\rank_{\Lambda(H)}\mathcal{X}(E/\Fin)_f = \lambda_p(E/F_{\cyc}).
\]
\end{proposition}

\begin{proof}
The hypothesis on $E$ and $E'$ guarantee that $E[p^\infty]$ is not rational over $\Q(E'[p^\infty])$.
As $p$ varies over all odd primes away from the divisors of $N_1 N_2$, the base field $F= \Q(E'[p])$ varies as well.
It follows from \eqref{LambdaH rank formula} and Lemma~\ref{Zp corank cases} that
\[
\rank_{\Lambda(H)}\mathcal{X}(E/\Fin)_f = \lambda_p(E/F_{\cyc}).
\]
\end{proof}

\begin{remark}
In particular, Theorem~\ref{Thm: Hung-Lim Theorem} asserts that
\[
\rank_{\Z}E(F_n) \leq \lambda_p(E/F_{\cyc})p^{3n}.
\]
\end{remark}

\appendix
\section{Classifying conjugacy classes in $\GL_2(\Z/p\Z)$}
\label{appendix}

The results in this appendix are used in obtaining precise estimates of quantities arising in Corollary~\ref{corollary chi=1} and Theorem~\ref{thm: Hep-Gep density}.
 See in particular Remark~\ref{rk: coming from appendix}. Given a prime $p$, write $G_p=\GL_2(\Z/p\Z)$.
Given a matrix $\sigma\in G_p$, denote by $\alpha=\alpha(\sigma)$ and $\beta=\beta(\sigma)$ the two eigenvalues of $\sigma$.
Here, $\alpha$ and $\beta$ are interchangeable, however, assume that each $\sigma$ comes with a choice of $\alpha$ and $\beta$.
As is well known, $\absolute{G_p}=(p^2-1)(p^2-p)=p(p-1)^2(p+1)$.
Given an element $\sigma\in G_p$, write $f(\sigma)$ for the smallest integer $f\in\Z_{\ge1}$ such that $\det(\sigma)^f=1$.
Let $H_p$ be the subset of $G_p$ consisting of elements $\sigma\in G_p$ such that $\alpha^{f(\sigma)}=\beta^{-f(\sigma)}\ne1$, where $\alpha$ and $\beta$ are the eigenvalues of $\sigma$.
We divide $G_p$ into the following conjugacy classes.
(See \cite[Chapter XVIII, Table 12.4]{lang} or \url{http://www-math.mit.edu/~dav/gl2conj.pdf}.)

\begin{itemize}
 \item Let $C_{a,b}$ be the set of diagonalizable matrices with eigenvalues $a,b\in {\Fp}^\times$ with $a\ne b$.
 We have $(p-1)(p-2)/2$ choices of $C_{a,b}$ and for each choice, $\#C_{a,b}=p(p+1)$.
 \item Let $\cC_a$ be the set of non-diagonal matrices with one single eigenvalue $a\in {\Fp}^\times$.
 There are $(p-1)$ choices for $\cC_a$ and for each choice, $\#\cC_a=p^2-1$.
 \item Let $D_a=\left\{\begin{pmatrix}a&0\\0&a\end{pmatrix}\right\}$, $a\in \Fp^\times$.
 Then, there are $(p-1)$ choices for $a$ and for each choice $\#D_a=1$.
 \item Let $E_{\lambda}$ be the set of matrices whose eigenvalues are $\lambda$ and $\lambda'$, where $\lambda\in\mathbb{F}_{p^2}\setminus\Fp$ and $\lambda'$ is the conjugate of $\lambda$.
 There are $p(p-1)/2$ choices for $\lambda$ and for each choice of $\lambda$, $\#E_\lambda=p^2-p$.
\end{itemize}

Given an element $a\in \bar{\F}_p^\times$, we write $o(a)$ for the order of $a$, i.e., the minimal value of $N\in \Z_{\geq 1}$ such that $a^N=1$.
Also, let $\varphi$ denote Euler's totient function, where $\varphi(n)$ is the number of positive integers $m<n$ that are coprime to $n$.
Recall that $\varphi(n)=n\prod_\ell \left(1-\frac{1}{\ell}\right)$, where $\ell$ runs through all prime divisors of $n$ and that $\sum_{d|n}\varphi(d)=n$.

\subsection{Contributions from $E_\lambda$}

Let $\sigma\in E_\lambda $ for some $\lambda\in\mathbb{F}_{p^2}\setminus\Fp $.
Since $\lambda\lambda'\in \Fp^\times$, it follows that $f(\sigma)=o(\lambda\lambda')$ has to be a divisor of $p-1$.
However, the fact that $\lambda\notin\Fp^\times$ implies that $\lambda^{p-1}\ne 1$.
Thus, $\sigma \in H_p$.
In particular, this gives that the number of elements are
\[
\sum_\lambda \# E_\lambda=p^2(p-1)^2/2.
\]

\subsection{Contributions from $D_a$ and $\cC_a$}
Let $\sigma\in D_a$ or $\cC_a$ for some $a\in\Fp^\times$, then note that $f(\sigma)=o(a^2)$.
As we shall see, in this case, $\sigma\in H_p$ if and only if $o(a)$ is even.
First suppose that $o(a)$ is even, then, we have that $o(a)/2=f(\sigma)$.
Therefore, $a^{f(\sigma)}=-1$, and hence, $\sigma\in H_p$, since $a^{f(\sigma)}\neq 1$.
Else, when $o(a)$ is odd, then it is easily checked that $\sigma\notin H_p$.

It remains to determine which $a\in \Fp^\times$ has $o(a)\in 2\Z$.
Suppose that $p-1=2^en$, where $n$ is an odd integer.
Then, $o(a)$ is odd if and only if $o(a)|n$.
Therefore, we have in total \[
p-1-\sum_{r|n}\varphi(r)=p-1-n=(p-1)(1-2^{-e})
\]
choices of $a$.
It follows that the number of elements in $H_p$ is given by
\[
(1+(p^2-1))(p-1)(1-2^{-e})=p^2(p-1)(1-2^{-e}).
\]

\begin{remark}\label{rk:1st-portion}
Since $e\ge1$, on combining the contributions from $E_\lambda$, $D_a$ and $\cC_a$, we have
\[
\frac{\#H_p}{\#G_p}\ge \frac{p^2(p-1)^2/2+p^2(p-1)/2}{p(p-1)^2(p+1)}=\frac{p^2}{2 (p^2-1)}.
\]
\end{remark}
\subsection{Contributions from $C_{a,b}$}
Given $a\in\Fp^\times$, we would like to find $b$ such that $o(ab)$ is not a multiple of $o(a)$.
Recall that for each $d|(p-1)$, there are $\varphi(d)$ elements of order $d$ in $\Fp^\times$.
Therefore the number of choices of $ab$ is
\[
\sum_{o(a)\nmid d|p-1}\varphi(d)=p-1-\sum_{o(a)\absolute{d}p-1}\varphi(d).
\]
But $\{ab:b\in\Fp^\times\}=\Fp^\times$.
Therefore, we have the same number of choices for $b$.
Excluding those choices where $a=b$ (there are $(p-1)(1-2^{-e})$ such possibilities, where $e=\ord_2(p-1)$) and repetitions (swapping $a$ and $b$), there are in total
\begin{align*}
 &\frac{1}{2}\left(\sum_{r|p-1}\varphi(r)\left(p-1-\sum_{r\absolute{d}p-1}\varphi(d)\right)-(p-1)(1-2^{-e})\right)\\
 =\ &\frac{1}{2}\left((p-1)^2-\sum_{r\absolute{d}p-1}\varphi(r)\varphi(d)-(p-1)(1-2^{-e})\right)\\
 =\ &\frac{1}{2}\left((p-1)(p-2+2^{-e})-\sum_{d|p-1}d\varphi(d)\right) =\ \frac{1}{2}\left((p-1)(p-2+2^{-e})-\sum_{d|p-1}\varphi(d^2)\right)
\end{align*}
classes of $C_{a,b}$ belonging to $H_p$.
In total, the number of elements are
\[
\frac{1}{2}\left((p-1)(p-2+2^{-e})-\sum_{d|p-1}\varphi(d^2)\right)p(p+1).
\]

\begin{remark}
\label{rk:2nd-portion}
Writing $p-1=\prod_{i} q_i^{n_i}$, where $q_i$ are distinct primes and $n_i\ge 1$, we have

\begin{align*}
 \sum_{d|p-1}\varphi(d^2)&=\prod_i\left(\sum_{m=0}^{n_i}\varphi(q_i^{2m})\right)=\prod_i\left(\sum_{m=0}^{2n_i}(-1)^mq_i^m\right)\\
 &= \prod_iq_i^{2n_i}\frac{1+q_i^{-2n_i-1}}{1+q_i^{-1}}= (p-1)^2\prod_i\frac{1+q_i^{-2n_i-1}}{1+q_i^{-1}}.
\end{align*}
Let us write $\kappa_p=\prod_i\frac{1+q_i^{-2n_i-1}}{1+q_i^{-1}}$, which is a constant strictly smaller than $1$.
Then the proportion of elements in $H_p$ coming from $C_{a,b}$ is
\[
\frac{\frac{1}{2}\left((p-1)(p-2+2^{-e})-(p-1)^2\kappa_p\right)p(p+1)}{p(p-1)^2(p+1)}=\frac{1}{2}\left(\frac{p-2+2^{-e}}{p-1}-\kappa_p\right).
\]
Observe that
\[
\kappa_p\le\frac{1+2^{-2e-1}}{1+2^{-1}}=\frac{2+2^{-2e}}{3}.\]
Therefore, for $p\ge7$, we have
\begin{align*}
 \frac{1}{2}\left(\frac{p-2+2^{-e}}{p-1}-\kappa_p\right)&\geq\frac{1}{2}\left(\frac{p-2+2^{-e}}{p-1}-\frac{2+2^{-2e}}{3}\right)\\
 &=\frac{1}{2}\left(\frac{p-2}{p-1}-\frac{2}{3}+\frac{2^{-e}}{p-1}-\frac{2^{-2e}}{3}\right)\\
 &\geq 
 \frac{1}{2}\left(\frac{2p-3}{2(p-1)}-\frac{3}{4}\right),
\end{align*}
which tends to $1/8$ as $p\rightarrow\infty$.
The inequality in the second last line follows from the fact that the real function $f(x)=\frac{x}{p-1}-\frac{x^2}{3}$ attains a minimum at $x=1/2$ in the interval $(0,1/2]$ when $p\ge 7$.
\end{remark}
\bibliographystyle{amsalpha}
\bibliography{references}

\newcommand{\etalchar}[1]{$^{#1}$}
\providecommand{\bysame}{\leavevmode\hbox to3em{\hrulefill}\thinspace}
\providecommand{\MR}{\relax\ifhmode\unskip\space\fi MR }
\providecommand{\MRhref}[2]{%
  \href{http://www.ams.org/mathscinet-getitem?mr=#1}{#2}
}
\providecommand{\href}[2]{#2}
\begin{thebibliography}{BKLOS21}

\bibitem[AW06]{AW06}
Konstantin Ardakov and Simon Wadsley, \emph{Characteristic elements for
  $p$-torsion {I}wasawa modules}, J. Algebraic Geometry \textbf{15} (2006),
  no.~2, 339--378.

\bibitem[BH97]{BH97}
Paul~N Balister and Susan Howson, \emph{Note on {N}akayama’s lemma for
  compact ${\Lambda}$-modules}, Asian J. Math. \textbf{1} (1997), no.~2,
  224--229.

\bibitem[Bha07]{Bha07}
Amala Bhave, \emph{Analogue of {K}ida's formula for certain strongly admissible
  extensions}, J. Number Theory \textbf{122} (2007), no.~1, 100--120.

\bibitem[BKLOS21]{BKLOS}
Manjul Bhargava, Zev Klagsbrun, Robert~J Lemke~Oliver, and Ari Shnidman,
  \emph{Elements of given order in {T}ate--{S}hafarevich groups of abelian
  varieties in quadratic twist families}, Algebra Number Theory \textbf{15}
  (2021), no.~3, 627--655.

\bibitem[CFK{\etalchar{+}}05]{CKFVS}
John Coates, Takako Fukaya, Kazuya Kato, Ramdorai Sujatha, and Otmar Venjakob,
  \emph{The {$\operatorname{GL}_2$} main conjecture for elliptic curves without
  complex multiplication}, Publ. Math. Inst. Hautes {\'E}tudes Sci.
  \textbf{101} (2005), 163--208.

\bibitem[CFKS10]{CFKS}
John Coates, Takako Fukaya, Kazuya Kato, and Ramdorai Sujatha, \emph{Root
  numbers, {S}elmer groups, and non-commutative {I}wasawa theory}, J. Algebraic
  Geom. \textbf{19} (2010), no.~1, 19--97.

\bibitem[CH01]{coateshowson}
John Coates and Susan Howson, \emph{Euler characteristics and elliptic curves.
  {II}}, J. Math. Soc. Japan \textbf{53} (2001), no.~1, 175--235.

\bibitem[CM03]{CM03}
Koji Chinen and Leo Murata, \emph{On a distribution property of the residual
  order of {$a$} (mod {$p$})}, Proc. Japan Acad. Ser. A Math. Sci. \textbf{79}
  (2003), no.~2, 28--32.

\bibitem[Coa99]{Coates_Fragments}
John Coates, \emph{Fragments of the {${\rm GL}_2$} {I}wasawa theory of elliptic
  curves without complex multiplication}, Arithmetic theory of elliptic curves
  ({C}etraro, 1997), Lecture Notes in Math., vol. 1716, Springer, Berlin, 1999,
  pp.~1--50.

\bibitem[Coj04]{Coj04}
Alina~Carmen Cojocaru, \emph{Questions about the reductions modulo primes of an
  elliptic curve}, Proc. 7th Meeting of the Canadian Number Theory Association
  (Montreal, 2002), Citeseer, 2004, pp.~61--79.

\bibitem[CP19]{CP19}
John Cremona and Ariel Pacetti, \emph{On elliptic curves of prime power
  conductor over imaginary quadratic fields with class number 1}, Proc. London
  Math. \textbf{118} (2019), no.~5, 1245--1276.

\bibitem[CS05]{CoatesSujatha_fineSelmer}
John Coates and Ramdorai Sujatha, \emph{Fine {S}elmer groups of elliptic curves
  over {$p$}-adic {L}ie extensions}, Math. Ann. \textbf{331} (2005), no.~4,
  809--839.

\bibitem[CS10]{CoatesSujatha_book}
\bysame, \emph{Galois cohomology of elliptic curves}, second ed., Published by
  Narosa Publishing House, New Delhi; for the Tata Institute of Fundamental
  Research, Mumbai, 2010.

\bibitem[CS12]{CS12}
J~Coates and R~Sujatha, \emph{On the {$\mathfrak{M}_H(G)$}-conjecture},
  Non-abelian fundamental groups and Iwasawa theory, London Mathematical
  Society Lecture Note Series \textbf{393} (2012), 132--161.

\bibitem[CS21]{CS20}
J.~E. Cremona and M.~Sadek, \emph{Local and global densities for {W}eierstrass
  models of elliptic curves}, accepted for publication in Mathematical Research
  Letters (2021), preprint, arXiv:2003.08454.

\bibitem[CSS03]{CoatesSchneiderSujatha_Links_between}
John Coates, Peter Schneider, and Ramdorai Sujatha, \emph{Links between
  cyclotomic and {${\rm GL}_2$} {I}wasawa theory}, Doc. Math. (2003), no.~Extra
  Vol., 187--215, Kazuya Kato's fiftieth birthday.

\bibitem[DdSMS99]{dixon}
John Dixon, Marcus du~Sautoy, Avinoam Mann, and Dan Segal, \emph{Analytic
  pro-{$p$} groups}, second ed., Cambridge Studies in Advanced Mathematics,
  vol.~61, Cambridge University Press, Cambridge, 1999.

\bibitem[Del07]{Del07}
Christophe Delaunay, \emph{Heuristics on class groups and on
  {T}ate--{S}hafarevich groups: the magic of the {C}ohen--{L}enstra
  heuristics}, Ranks of elliptic curves and random matrix theory \textbf{341}
  (2007), 323--340.

\bibitem[DL15]{DL}
Daniel Delbourgo and Antonio Lei, \emph{Transition formulae for ranks of
  abelian varieties}, Rocky Mountain J. Math. \textbf{45} (2015), no.~6,
  1807--1838.

\bibitem[DL17]{DL2}
\bysame, \emph{Estimating the growth in {M}ordell-{W}eil ranks and
  {S}hafarevich-{T}ate groups over {L}ie extensions}, Ramanujan J. \textbf{43}
  (2017), no.~1, 29--68.

\bibitem[DT10]{darmontian}
Henri Darmon and Ye~Tian, \emph{Heegner points over towers of {K}ummer
  extensions}, Canad. J. Math. \textbf{62} (2010), no.~5, 1060--1081.

\bibitem[Gre99]{Greenberg}
Ralph Greenberg, \emph{Iwasawa theory for elliptic curves}, Arithmetic theory
  of elliptic curves ({C}etraro, 1997), Lecture Notes in Math., vol. 1716,
  Springer, Berlin, 1999, pp.~51--144.

\bibitem[Gre03]{Gre03}
\bysame, \emph{Galois theory for the {S}elmer group of an abelian variety},
  Compos. Math. \textbf{136} (2003), no.~3, 255--297.

\bibitem[Gre06]{greenberg06}
\bysame, \emph{On the structure of certain {G}alois cohomology groups}, Doc.
  Math. (2006), no.~Extra Vol., 335--391.

\bibitem[Gre10]{greenberg10}
\bysame, \emph{Surjectivity of the global-to-local map defining a {S}elmer
  group}, Kyoto J. Math. \textbf{50} (2010), no.~4, 853--888.

\bibitem[Gre16]{greenberg16}
\bysame, \emph{On the structure of {S}elmer groups}, Elliptic curves, modular
  forms and {I}wasawa theory, Springer Proc. Math. Stat., vol. 188, Springer,
  Cham, 2016, pp.~225--252.

\bibitem[Har79]{harris}
Michael Harris, \emph{{$p$}-adic representations arising from descent on
  abelian varieties}, Compos. Math. \textbf{39} (1979), no.~2, 177--245.

\bibitem[Har00]{harris2}
\bysame, \emph{Correction to: ``{$p$}-adic representations arising from descent
  on abelian varieties'' [{C}ompositio {M}ath. {\bf 39} (1979), no. 2,
  177--245]}, Compos. Math. \textbf{121} (2000), no.~1, 105--108.

\bibitem[HKR21]{HKR}
Jeffrey Hatley, Debanjana Kundu, and Anwesh Ray, \emph{Statistics for
  anticyclotomic {I}wasawa invariants of elliptic curves}, preprint,
  arXiv:2106.01517 (2021).

\bibitem[HL20]{HL}
Pin~Chi Hung and Meng~Fai Lim, \emph{On the growth of {M}ordell--{W}eil ranks
  in {$p$}-adic {L}ie extensions}, Asian J. Math. \textbf{24} (2020), no.~4,
  549--570.

\bibitem[HM99]{HM99}
Yoshitaka Hachimori and Kazuo Matsuno, \emph{An analogue of {K}ida's formula
  for the {S}elmer groups of elliptic curves}, J. Alg. Geom. \textbf{8} (1999),
  no.~3, 581--601.

\bibitem[HO10]{HO10}
Yoshitaka Hachimori and Tadashi Ochiai, \emph{Notes on non-commutative
  {I}wasawa theory}, Asian J. Math. \textbf{14} (2010), no.~1, 11--18.

\bibitem[How02]{howson2002euler}
Susan Howson, \emph{Euler characteristics as invariants of {I}wasawa modules},
  Proc. London Math. Soc. \textbf{85} (2002), no.~3, 634--658.

\bibitem[HV03]{HV03}
Yoshitaka Hachimori and Otmar Venjakob, \emph{Completely faithful {S}elmer
  groups over {K}ummer extensions}, Doc. Math. (2003), no.~Extra Vol.,
  443--478, Kazuya Kato's fiftieth birthday.

\bibitem[Kat04]{Kato}
Kazuya Kato, \emph{{$p$}-adic {H}odge theory and values of zeta functions of
  modular forms}, Ast\'{e}risque (2004), no.~295, ix, 117--290, Cohomologies
  $p$-adiques et applications arithm\'{e}tiques. III.

\bibitem[Kid03]{kida03}
Masanari Kida, \emph{Variation of the reduction type of elliptic curves under
  small base change with wild ramification}, Central European J. Math.
  \textbf{1} (2003), no.~4, 510--560.

\bibitem[KR21a]{KR21}
Debanjana Kundu and Anwesh Ray, \emph{Statistics for {I}wasawa invariants of
  elliptic curves}, Trans. American Math. Soc. \textbf{374} (2021), 7945--7965.

\bibitem[KR21b]{KR21b}
\bysame, \emph{Statistics for {I}wasawa invariants of elliptic curves,
  $\rm{II}$}, preprint, arXiv:2106.12095 (2021).

\bibitem[Kun20]{DK_thesis}
Debanjana Kundu, \emph{Iwasawa theory of fine {S}elmer groups}, Ph.D. thesis,
  University of Toronto (Canada), 2020.

\bibitem[Lam99]{Lam}
T.~Y. Lam, \emph{Lectures on modules and rings}, Graduate Texts in Mathematics,
  vol. 189, Springer-Verlag, New York, 1999.

\bibitem[Lan02]{lang}
Serge Lang, \emph{Algebra}, third ed., Graduate Texts in Mathematics, vol. 211,
  Springer-Verlag, New York, 2002.

\bibitem[Lim15]{lim2015remark}
Meng~Fai Lim, \emph{A remark on the $\mathfrak{M}_{H}({G})$-conjecture and
  {A}kashi series}, Int. J. Number Theory \textbf{11} (2015), no.~01, 269--297.

\bibitem[Lim21]{LimKida}
\bysame, \emph{Some remarks on {K}ida's formula when {$\mu\ne 0$}}, Ramanujan
  J. \textbf{55} (2021), no.~3, 1127--1144.

\bibitem[LJ87]{Lenstra_annals}
Hendrik~W Lenstra~Jr, \emph{Factoring integers with elliptic curves}, Ann.
  Math. (1987), 649--673.

\bibitem[LM14]{LM14}
Meng~Fai Lim and V~Kumar Murty, \emph{The growth of the selmer group of an
  elliptic curve with split multiplicative reduction}, Int. J. Number Theory
  \textbf{10} (2014), no.~03, 675--687.

\bibitem[LS20]{leisprung}
Antonio Lei and Florian Sprung, \emph{Ranks of elliptic curves over {$\Bbb
  Z_p^2$}-extensions}, Israel J. Math. \textbf{236} (2020), no.~1, 183--206.

\bibitem[Maz72]{mazur72}
Barry Mazur, \emph{Rational points of abelian varieties with values in towers
  of number fields}, Invent. Math. \textbf{18} (1972), 183--266.

\bibitem[Mor22]{Mor22}
Louis Mordell, \emph{On the rational solutions of the indeterminate equation of
  the third and fourth degree}, Proc. Camb. Phil. Soc., vol.~21, 1922,
  pp.~179--192.

\bibitem[Mur97]{Mur97}
V~Kumar Murty, \emph{Modular forms and the {Chebotarev} density theorem {II}},
  London Mathematical Society Lecture Note Series (1997), 287--308.

\bibitem[Neu88]{Neu88}
Andreas Neumann, \emph{Completed group algebras without zero divisors}, Archiv
  der Mathematik \textbf{51} (1988), no.~6, 496--499.

\bibitem[OV02]{OV02}
Yoshihiro Ochi and Otmar Venjakob, \emph{On the structure of {S}elmer groups
  over $p$-adic {L}ie extensions}, J. Alg. Geom. \textbf{11} (2002), no.~3,
  547--580.

\bibitem[OV03]{OV03}
\bysame, \emph{On the ranks of {I}wasawa modules over {$p$}-adic {L}ie
  extensions}, Math. Proc. Cambridge Philos. Soc. \textbf{135} (2003), no.~1,
  25--43.

\bibitem[PR82]{PR82}
Bernadette Perrin-Riou, \emph{Descente infinie et hauteur $p$-adique sur les
  courbes elliptiques {\`a} multiplication complexe}, Invent. Math. \textbf{70}
  (1982), no.~3, 369--398.

\bibitem[Qiu14]{Qiu14}
Derong Qiu, \emph{On quadratic twists of elliptic curves and some applications
  of a refined version of {Y}u's formula}, Communications in Algebra
  \textbf{42} (2014), no.~12, 5050--5064.

\bibitem[Ray21]{Ray21_noncom_rank}
Anwesh Ray, \emph{Asymptotic growth of {M}ordell--{W}eil ranks of elliptic
  curves in noncommutative towers}, 2021, preprint, arXiv:2109.07457.

\bibitem[Roh88]{rohrlich88}
David~E. Rohrlich, \emph{{$L$}-functions and division towers}, Math. Ann.
  \textbf{281} (1988), no.~4, 611--632.

\bibitem[RS20]{ray2020euler2}
Anwesh Ray and R~Sujatha, \emph{Euler characteristics and their congruences for
  multi-signed selmer groups}, preprint arXiv:2011.05387 (2020).

\bibitem[RS21]{ray2021euler1}
Anwesh Ray and Ramdorai Sujatha, \emph{Euler characteristics and their
  congruences in the positive rank setting}, Can. Math. Bull. \textbf{64}
  (2021), no.~1, 228--245.

\bibitem[Sch82]{Schneider82}
Peter Schneider, \emph{$p$-adic height pairings {I}}, Invent. Math. \textbf{69}
  (1982), no.~3, 401--409.

\bibitem[Sch85]{Schneider85}
\bysame, \emph{$p$-adic height pairings {II}}, Invent. Math. \textbf{79}
  (1985), no.~2, 329--374.

\bibitem[Sch87]{Schoof87}
Ren{\'e} Schoof, \emph{Nonsingular plane cubic curves over finite fields}, J.
  Comb. Theory Ser. A \textbf{46} (1987), no.~2, 183--211.

\bibitem[Ser68]{serre68}
Jean-Pierre Serre, \emph{Abelian {$l$}-adic representations and elliptic
  curves}, W. A. Benjamin, Inc., New York-Amsterdam, 1968, McGill University
  lecture notes written with the collaboration of Willem Kuyk and John Labute.

\bibitem[Ser72]{Ser72}
\bysame, \emph{Propri{\'e}t{\'e}s {g}aloisiennes des points d’ordre fini des
  courbes elliptiques}, Invent. math \textbf{15} (1972), 259--331.

\bibitem[Sil09]{silverman2009}
Joseph~H Silverman, \emph{The arithmetic of elliptic curves}, Graduate Texts in
  Mathematics, vol. 106, Springer, 2009.

\bibitem[Ven02]{venjakob2002structure}
Otmar Venjakob, \emph{On the structure theory of the {I}wasawa algebra of a
  $p$-adic {L}ie group}, J. European Math. Soc. \textbf{4} (2002), no.~3,
  271--311.

\bibitem[Ven03]{venjakob2003iwasawa}
\bysame, \emph{On the {I}wasawa theory of $p$-adic {L}ie extensions}, Compos.
  Math. \textbf{138} (2003), no.~1, 1--54.

\bibitem[Was97]{Was97}
Lawrence~C Washington, \emph{Introduction to cyclotomic fields}, vol.~83,
  Springer Science \& Business Media, 1997.

\bibitem[Wei29]{Wei29}
Andr{\'e} Weil, \emph{L'arithm{\'e}tique sur les courbes alg{\'e}briques}, Acta
  mathematica \textbf{52} (1929), no.~1, 281--315.

\bibitem[Zer09]{zerbes2009generalised}
Sarah~Livia Zerbes, \emph{Generalised {E}uler characteristics of {S}elmer
  groups}, Proc. London Math. Soc. \textbf{98} (2009), no.~3, 775--796.

\bibitem[Zer11]{zerbes11}
\bysame, \emph{Akashi series of {S}elmer groups}, Math. Proc. Cambridge Philos.
  Soc. \textbf{151} (2011), no.~2, 229--243.

\end{thebibliography}

\end{document}